\documentclass[11pt]{article}
\usepackage{amsmath, amsfonts, amssymb,amsthm,mathrsfs}
\usepackage{color,verbatim,graphicx}
\usepackage{fullpage,algorithm, multicol, multirow}
\usepackage{graphicx,subfigure}
\usepackage{algorithm}
\usepackage{algorithmic}
\usepackage[title]{appendix}

\newtheorem{theorem}{Theorem}
\newtheorem{proposition}{Proposition}

\newcommand{\commentout}[1]{}

\newcommand{\ba}{\mathbf{a}}
\newcommand{\bb}{\mathbf{b}}
\newcommand{\be}{\mathbf{e}}
\newcommand{\bz}{\mathbf{z}}
\newcommand{\bg}{\mathbf{g}}
\newcommand{\bx}{\mathbf{x}}
\newcommand{\by}{\mathbf{y}}
\newcommand{\bbeta}{\boldsymbol{\eta}}
\newcommand{\calD}{\mathcal{D}}

\newcommand{\hF}{\widehat{F}}

\newcommand{\opt}{{\rm opt}}
\newcommand{\supp}{{\rm supp}}

\begin{document}

\title{IDENT: Identifying Differential Equations \\ with Numerical Time evolution}
\author{Sung Ha Kang \thanks{School of Mathematics, Georgia Institute of Technology. Email: kang@math.gatech.edu. Research is supported in part by Simons Foundation grant 282311 and 584960.}
\and
 Wenjing Liao
 \thanks{School of Mathematics, Georgia Institute of Technology. Email: wliao60@gatech.edu. Research is supported in part by the NSF grant DMS 1818751.}
 \and
  Yingjie Liu
  \thanks{School of Mathematics, Georgia Institute of Technology. Email: yingjie@math.gatech.edu. Research is supported in part by NSF grants DMS-1522585 and DMS-CDS\&E-MSS-1622453.} }

\maketitle

\begin{abstract}

Identifying unknown differential equations from a given set of discrete time dependent data is a challenging problem.   
A small amount of noise can make the recovery unstable, and nonlinearity and differential equations with varying coefficients add complexity to the problem. 
We assume that the governing partial differential equation (PDE) is a linear combination of a subset of a prescribed dictionary containing different  differential terms, and the objective of this paper is to find the correct coefficients.   

We propose a new direction  based on  the fundamental idea of convergence analysis of numerical PDE schemes.  We utilize Lasso for efficiency, and a performance guarantee is established based on an incoherence property.   The main contribution is to validate and correct the results by Time Evolution Error (TEE).  The new algorithm, called Identifying Differential Equations with Numerical Time evolution (IDENT), is  explored for  data with non-periodic boundary conditions, noisy data and PDEs with varying coefficients.  
From the recovery analysis of Lasso, we propose a new definition of Noise-to-Signal ratio, which better represents the level of noise in the case of PDE identification. 
We systematically analyze the effects of data generations and downsampling, and propose an order preserving denoising method called Least-Squares Moving Average (LSMA), to preprocess the given data.  For  the identification of PDEs with varying coefficients, we propose to add  Base Element Expansion (BEE) to aide the computation.  Various numerical experiments from basic tests to noisy data, downsampling effects and varying coefficients are presented.

 \end{abstract}

\section{Introduction}

Physical laws are often presented  by the means of differential equations. The original discoveries of differential equations associated with real-world physical processes typically require a good understanding of the physical laws, and supportive evidence from empirical observations. 
We consider an inverse problem of this -  from the experimental real data, how to directly recognize the underlying PDE.  We combine tools from machine learning and numerical PDEs to explore the given data and automatically identify the underlying dynamics.

Let $ \{u_{i}^n | i=1, \dots, N_1 \text{ and } n= 1, \dots, N_2\}$ be the given discrete time dependent data, where the index $i$ and $n$ represent the spacial and time discrete domain, respectively. The objective is to find the differential equation, i.e., an operator $\mathcal{F}$:
\[  u_t = \mathcal{F}(x,u,u_x,u_{xx})  \text{ such that }  u(x_i,t_n) \approx u_i^n.  \]

Recently there have been a number of important works on learning dynamical systems or differential equations. Two pioneering works can be found in \cite{bongard2007automated,schmidt2009distilling}, where symbolic regression was used to recover the underlying physical systems from experimental data. 
In \cite{brunton2016discovering}, Brunton, et al. considered the discovery of nonlinear dynamical systems with sparsity-promoting techniques. The underlying dynamical systems are assumed to be governed by a small number of active terms in a prescribed dictionary, and sparse regression is used to identify these active terms.
Various extensions of this sparse regression approach can be found in \cite{kaiser2018sparse,loiseau2018constrained,mangan2017model,rudy2017data}.
In \cite{schaeffer2017learning}, Schaeffer considered the problem of learning PDEs using spectral method, and focused on the benefit of using $L^1$ minimization for sparse coefficient recovery. 
Highly corrupted and undersampled data are considered in \cite{tran2017exact,schaeffer2018extracting} for the recovery of dynamical systems.
In \cite{schaeffer2018extracting}, Schaeffer et al. developed a random sampling theory for the selection dynamical systems from undersampled data. 
These nice series of works focused on the benefit and power of using $L^1$ minimization to resolve dynamical systems or PDEs with certain sparse pattern \cite{schaeffer2013sparse}. 
A Bayesian approach was considered in  \cite{zhang2018robust} where Zhang et al. used dimensional analysis and sparse Bayesian regression to recover the underlying dynamical systems. Another related problem is to infer the interaction function in a system of agents from the trajectory data. In \cite{bongini2017inferring,lu2018nonparametric}, nonparametric regression was used to predict the interaction function and a theoretical guarantee was established.

There are approaches using deep learning techniques. In \cite{long2017pde}, Long et al. proposed  a PDE-Net to learn differential operators by learning convolution kernels. 
In \cite{raissi2017physics}, Raissi et al. used neural networks to learn and predict the solution of the equation without finding its explicit form.  In \cite{raissi2018hidden}, neural networks were further used to learn certain parameters in the PDEs from the given data. In \cite{qin2018data},  Residual Neural Networks (ResNet) are used as building blocks for equation approximation. In \cite{khoo2018switchnet}, neural networks are used to solve the  wave equation based inverse scattering problems by providing maps between the scatterers and the
scattered field (and vice versa).  Related works showing the advantages of deep learning include \cite{khoo2018switchnet,lusch2018deep,qin2018data,raissi2018deep}. 

In this paper,  we propose a new algorithm based on the convergence analysis of numerical PDE  schemes. We assume that the governing PDE is a linear combination of a subset of a prescribed dictionary containing different  differential terms,  and the objective is to find the correct set of coefficients.   We use finite difference methods, such as the 5-point  ENO scheme, to approximate the spatial derivatives in the dictionary.  While we utilize $L^1$ minimization to aid the efficiency of the approach,  the main idea is to validate and correct the results by Time Evolution Error (TEE).  This approach, we call Identifying Differential Equations with Numerical Time evolution (IDENT) is  explored for  data with non-periodic boundary conditions, noisy data and PDEs with varying coefficients for nonlinear PDE identification.  
For noisy data, we propose an order preserving denoising method called Least Square Moving Average (LSMA) to effectively denoise the given data.  To tackle varying coefficients, we expand the number of coefficients in terms of finite element bases. This procedure called Base Element Expansion (BEE), again uses the fundamental idea of convergence in finite element approximation.   
From a theoretical perspective, we establish a performance guarantee based on an incoherence property, and define a new noise-to-signal ratio for the PDE identification problem.  Contributions of this paper include:
\begin{enumerate}\vspace{-0.2cm}
\item{establishing a new direction of using numerical PDE techniques for PDE identification,  } \vspace{-0.2cm}
\item{proposing a flexible approach which can handle different boundary conditions, are more robust against noise, and can identify nonlinear PDEs with varying coefficients,  }\vspace{-0.2cm}
\item{establishing a recovery theory of Lasso for weighted $L^1$ minimization, which leads to the new definition of noise-to-signal ratio for PDE identification,}\vspace{-0.2cm}
\item{systematically analyzing the noise and downsampling, and proposing a new denoising method called Least Square Moving Average (LSMA).  }
\end{enumerate}

This paper is organized as follows: The main algorithm is presented in Section \ref{sec:ident}, aspects of denoising and downsampling effects are in Section \ref{sec:noise}, and PDEs with varying coefficients are in Section \ref{sec:varying}, followed by a concluding remark in  \ref{sec:summary} and some details in the Appendix. 
Specifically, the set-up of the problem is presented in subsection \ref{ssec:setup}; details of the IDENT algorithm are in subsection \ref{ssec:algo};  a recovery theory for Lasso and the new noise-to-signal ratio are  in subsection \ref{ssec:recovery};  and  the first set of numerical experiments are in subsection \ref{ssec:num_nonoise}.  
In Section \ref{sec:noise} of denoising and downsampling,  LSMA denoising method is introduced in subsection \ref{subsec:lsma},  numerical experiments for noisy data are presented in subsection \ref{subsecnoisyexperiment}, and downsampling effects are considered in subsection \ref{subsec:down}.
In Section \ref{sec:varying}, we consider nonlinear PDEs with varying coefficients and introduce BEE  motivated by finite element approximation.

\section{Identifying Differential Equations with Numerical Time evolution (IDENT)} \label{sec:ident}

We start with general notations in Section \ref{ssec:notation} and the set-up of the problem in Section \ref{ssec:setup}, then present our IDENT algorithm with the time evolution error check  in Section \ref{ssec:algo}.  A recovery theory is established in Section \ref{ssec:recovery}, and the first set of numerical experiments is presented in Section \ref{ssec:num_nonoise}.

\subsection{Notations}  \label{ssec:notation}

We use bold letter to denote vectors, such as $\ba,\bb$. The support of a vector $\mathbf{x}$ is the set of indices at which it is nonzero: $\supp(\mathbf{x}) := \{j : x_j \neq 0\}$. 
We use $A^T$ and $A^*$ to denote the transpose and the conjugate transpose of the matrix $A$.
We use $x \rightarrow \varepsilon^+$ to denote $x > \varepsilon$ and $x\rightarrow \varepsilon$.
Let $\mathbf{f} = \{f(x_i,t_n) | i=1,\ldots,N_1, n=1,\ldots,N_2\} \in \mathbb{R}^{N_1N_2}$ be samples of a function $f: \calD \times [0,\infty) \rightarrow \mathbb{R}$ with spatial spacing $\Delta x$ and time spacing $\Delta t$.  The integers $N_1$ and $N_2$ are the total number of spatial and time discretization respectively.   We assume PDEs are simulated on the grid with time spacing $\delta t$ and spatial spacing $\delta x$, while data are sampled on the grid with time spacing $\Delta t$ and spatial spacing $\Delta x$.
The vector $L^p$ norm of $\mathbf{f}$ is $\|\mathbf{f}\|_p = (\sum_{i=1}^{N_1}\sum_{n=1}^{N_2}|f(x_i,t_n)|^p )^{1/p} $. Denote $\|\mathbf{f}\| = \|\mathbf{f}\|_2$. The function $L^p$ norm of $\mathbf{f}$ is $\|\mathbf{f}\|_{L^p} = (\sum_{i=1}^{N_1}\sum_{n=1}^{N_2}|f(x_i,t_n)|^p \Delta x \Delta t )^{1/p} $. Notice that $\|\mathbf{f}\|_{L^p} = \|\mathbf{f}\Delta x^{1/p} \Delta t^{1/p}\|_p$.

\subsection{The set-up of the problem} \label{ssec:setup}

We consider the parametric model of PDEs where $\mathcal{F}(x,u,u_x,u_{xx})$ is a linear combination of monomials such as $1$, $u$, $u^2$, $u_x$, $u_x^2$, $uu_x$, $u_{xx}$, $u_{xx}^2$, $uu_{xx}$, $u_x u_{xx}$ with coefficients $\ba = \{a_j\}_{j=1}^{10}$: 
\begin{equation}\label{E:constant_a}
u_t = a_1 + a_2 u + a_3 u^2 + a_4 u_x + a_5 u_x^2+ a_6 u u_x + a_7 u_{xx} + a_8 u^2_{xx} + a_9 u u_{xx} + a_{10} u_x u_{xx}.
\end{equation}
We refer to each monomial as a feature, and  let $N_3$ be the number of features, i.e.,  $N_3 = 10$ in \eqref{E:constant_a}. The right hand side can be viewed as a second-order Taylor expansion of $\mathcal{F}(u,u_x,u_{xx})$. It can easily be generalized to higher-order Taylor expansions, and operators $\mathcal{F}(u,u_x,u_{xx},u_{xxx}, \partial_x^4 u,\ldots)$ depending on higher order derivatives.  This model contains a rich class of differential equations, e.g., the heat equation, transport equation, Burger's equation, KdV equation, Fisher's equation that models gene propagation.

Evaluating \eqref{E:constant_a} at discrete time and space $(x_i,t_n), i=1,\ldots,N_1, n = 1,\ldots,N_2$ yields the discrete linear system 
$$ F \textbf{a} = \bb,$$ 
where $$\mathbf{b} = \{u_t(x_i,t_n)| i=1,\ldots,N_1, n = 1,\ldots,N_2\} \in \mathbb{R}^{N_1N_2},$$ and $F$ is a $N_1N_2 \times N_3 $ feature matrix in the form of
{\footnotesize
\begin{equation} \label{E:constant_F}
 F = \left( \begin{array}{cccccccccc}
\vdots & \vdots & \vdots & \vdots &\vdots & \vdots &\vdots & \vdots &\vdots & \vdots \\
1 & u(x_i,t_n) & u^2(x_i,t_n) & u_x(x_i,t_n) & u_x^2(x_i,t_n) & uu_x(x_i,t_n) & u_{xx}(x_i,t_n) & u_{xx}^2(x_i,t_n) & u u_{xx}(x_i,t_n) & u_x u_{xx}(x_i,t_n) \\
\vdots & \vdots & \vdots & \vdots &\vdots & \vdots &\vdots & \vdots &\vdots & \vdots \\
\end{array}  \right).
\end{equation}
}
We use $F[j]$ to denote the $j$th column vector associated with the $j$th feature evaluated at $(x_i,t_n), {i=1,\ldots,N_1, n=1,\ldots,N_2}$

The \textbf{objective} of PDE identification is to recover  the unknown coefficient vector  $\mathbf{a} \in \mathbb{R}^{N_3}$ from given data.  Real world physical processes are often presented with a few number of features in the right hand side of \eqref{E:constant_a}, so it is reasonable to assume that the coefficients are sparse. 

For differential equations with varying coefficients, we consider PDEs of the form
\begin{equation}\label{E:general_a}
u_t = a_1(x) + a_2 (x) u + a_3 (x) u^2 + a_4(x) u_x + a_5(x) u_x^2+ a_6(x) u u_x + a_7(x) u_{xx} + a_8 (x) u^2_{xx} + a_9 (x) u u_{xx} + a_{10}(x) u_x u_{xx}
\end{equation}
where each $a_j(x)$ is a function on the spatial domain of the PDE.
We expand the coefficients in terms of finite element bases $\{ \phi_l\}_{l=1}^{L} $ such that 
\begin{equation}\label{E:L}
a_j (x) \approx \sum_{l=1}^{L} a_{j,l} \phi_l(x) \text{ for } j=1,\dots,N_3,
\end{equation}
where $L$ is the number of finite element bases used to approximate $a_j(x)$. 
Let $y_1<y_2<\cdots <y_L$ be a partition of the spatial domain. We use a typical finite element basis function, e.g.,  $\phi_l(x)$ is  continuous, and  linear within each subinterval $(y_i, y_{i+1})$,  and $\phi_l(y_i)=\delta_{li}=1$ if $i=l$; $0$ otherwise. 
If the $a_j(x)$'s are Lipchitz functions, and finite element bases are defined on a grid with spacing $O(1/L)$. The approximation error of the $a_j(x)$'s satisfies 
\begin{equation}\label{E:approxError}
\|a_j - \sum_{l=1}^{L} a_{j,l} \phi_l\|_{L^p} \le O(1/L), \ p \in (0,\infty).
\end{equation}
In the case of varying coefficients, the feature matrix $F$ is of size $N_1 N_2 \times N_3 L$,
{\footnotesize
\begin{align} \label{E:general_F}
F = \left( \begin{array}{ccc|ccc|c|ccc}
\vdots &  & \vdots &  \vdots &  & \vdots &     & \vdots &  & \vdots \\
\phi_1(x_i) & \dots & \phi_{L}(x_i) & u(x_i,t_n)\phi_1(x_i) & \dots  & u(x_i,t_n)\phi_{L}(x_i) & \dots  & u_xu_{xx}(x_i,t_n)\phi_1(x_i) & \dots &u_x u_{xx}(x_i,t_n)\phi_{L}(x_i) \\
\vdots &  & \vdots &  \vdots &  & \vdots &     & \vdots &  & \vdots \\
\end{array}  \right),
\end{align}
}
and the vector to be identified is 
\[ 
\textbf{a} = \left(a_{1,1},\dots,a_{1,L} | a_{2,1}, \dots, a_{2,L} | \dots \dots \dots | a_{N_3,1}, \dots, a_{N_3,L} \right)^T
\in \mathrm{R}^{N_3 L}.  
\]
The feature matrix $F$ has a block structure. We use $F[j,l]$ to denote the column of $F$ associated with the $j$th feature and the $l$th basis. 
To be clear, $F[j]$ is the $j$th column of \eqref{E:constant_F}, and $F[j,l]$ is the $(j-1)L+l$th column of \eqref{E:general_F}. 
Evaluating \eqref{E:general_a} at $(x_i,t_n), i=1,\ldots,N_1, n = 1,\ldots,N_2$ yields the discrete linear system 
\[
F \ba = \bb+ \bbeta,
\]
where $\bbeta = \{\eta(x_i,t_n) | i=1,\ldots,N_1, n=1,\ldots,N_2\} \in \mathbb{R}^{N_1 N_2}$ represents the approximation error of the $a_j(x)$'s by finite element bases such that 
$$\eta(x_i,t_n) = \left(\sum_{l=1}^L a_{1,l}\phi_l(x_i) - a_1(x_i)\right) + \ldots + \left(\sum_{l=1}^L a_{10,l}\phi_l(x_i) - a_{10}(x_i)\right) u_xu_{xx} (x_i,t_n).$$
In the case that $u,u_x,u_{xx}$ are uniformly bounded, 
$$\|\bbeta\|_{L^p} \le O(1/L), \  p \in (0,\infty),$$
and $\bbeta = 0$ when all coefficients are constants.

\subsection{The proposed algorithm: IDENT} \label{ssec:algo}

In this paper, we assume that only the discrete data $\{u_{i}^n | i=1, \dots, N_1\text{ and } n= 1, \dots, N_2\}$ and the boundary conditions are given. If data are perfectly generated and there is no measurement noise, $u_{i}^n = u(x_i,t_n)$ for every $i$ and $n$, and we outline the proposed IDENT algorithm in this section assuming the given data do not have noise. 

\textbf{The first step} of IDENT is to construct the empirical version of the  feature matrix $F$ and the vector $\mathbf{b}$ containing time derivatives from the given data.  The derivatives are approximated by finite difference methods which gives  flexibility in dealing with different types of PDEs and boundary conditions (e.g. non-periodic).  
We approximate the time derivative $u_t$ by a first-order backward difference scheme:
$$u_t (x_j,t_k) \approx \widehat{u_t} (x_j,t_n) := \frac{u(x_j,t_n) - u(x_j,t_{n-1})}{\Delta t},$$
which yields the error
\[
\widehat{u_t} (x_j,t_n) = {u_t} (x_j,t_n) + O(\Delta t).
\]
Let $\widehat\bb$ be the empirical version of $\bb$ constructed from data: $$\widehat{\mathbf{b}} = \{\widehat{u_t}(x_i,t_n): i=1,\ldots,N_1, n = 1,\ldots,N_2\} \in \mathbb{R}^{N_1N_2}.$$
We approximate the spatial derivative $u_x$ through the five-point ENO method proposed by Harten, Engquist, Osher and Chakravarthy \cite{ENO87}. 
Let $\widehat{u_x} (x_j,t_n)$ and $\widehat{u_{xx}} (x_j,t_n)$ be approximations of ${u_x} (x_j,t_n)$ and ${u_{xx}} (x_j,t_n)$ by the five-point ENO method 
which yields the error:
\[
\widehat{u_x} (x_j,t_n) = {u_x} (x_j,t_n) + O(\Delta x^4), \quad 
\widehat{u_{xx}} (x_j,t_n) = {u_{xx}} (x_j,t_n) + O(\Delta x^3).
\]
Putting $\widehat{u_x} (x_j,t_n)$'s and $\widehat{u_{xx}} (x_j,t_n)$'s to the feature matrix $F$ in \eqref{E:general_F} gives rise to the empirical feature matrix, denoted by $\widehat{F}$. 
For example, the second column of $\widehat{F}$ is given by $\{u_i^n | i=1,\ldots,N_1 \text{ and } n=1,\ldots,N_2\}$ as an approximation of $\{u(x_i,t_n) | i=1,\ldots,N_1 \text{ and } n=1,\ldots,N_2\}$ as follows 
\[
(u_1^1, u_2^1, \dots,u_{N_1}^1,  u_1^2,  \dots, u_{N_1}^2, \dots, u_1^{N_2},  \dots,u_{N_1}^{N_2})^T \in \mathrm{R}^{N_1 N_2}.
\]
These empirical quantities give rise to the linear system
\begin{equation}
\label{E:linear1}
\widehat F \ba= \widehat\bb + \be , \quad \mathbf{e} =  \bb -\widehat{\bb} + (\widehat{F}-F)\ba +\bbeta,
\end{equation}
where the terms $\bb -\widehat{\bb}$, $(\widehat{F}-F)\ba$ and $\bbeta$ arise from errors in approximating time and spatial derivatives, and the finite element expansion of varying coefficients, respectively. The total error $\be$ satisfies
\begin{equation}
\label{E:errorbound}
\|\be\|_{L^2} \le  \varepsilon \text{ such that } \varepsilon  = O(\Delta t + \Delta x^3 + 1/L).
\end{equation}

\textbf{The second step}  is to find possible candidates for the non-zero coefficients of $ \ba$. 
We  utilize $L^1$-regularized minimization, also known as Lasso \cite{tibshirani1996regression} or group Lasso \cite{yuan2006model}, solved by Alternating Direction Method of Multipliers \cite{boyd2011distributed} to get a sparse or block-sparse vector.  We minimize the following energy:
\begin{equation}\label{eqgLasso}
\widehat{\ba}_{\text{G-Lasso}}(\lambda) ={\textstyle \arg \min_{\bz}} \left\{ \frac{1}{2} \| \widehat\bb - \widehat F_{\infty} \textbf{z} \|^2_2 + \lambda \sum_{j=1}^{N_3} \left(\sum_{l=1}^L |z_{j,l}|^2 \right)^{\frac{1}{2}} \right\},
\end{equation}
where $\lambda$ is a balancing parameter between the first fitting term and the second regularization term. The matrix  $\hF_\infty$ is obtained from $\hF$ with each column divided by the maximum magnitude of the column, namely, $\hF_\infty[j,l] = \hF[j,l]/\|\hF[j,l]\|_{\infty}$.    We use Lasso for the constant coefficient case where $L=1$, and group Lasso  for the varying coefficient case $L>1$.    
A set of possible active features is selected by thresholding the normalized coefficient magnitudes: 
\begin{equation}
\label{algthresholding}
\widehat\Lambda_{\tau} : = \left\{j : \|\hF[j]\|_{L^1}\left\| \sum_{l=1}^L \frac{\widehat{\ba}_{\text{G-Lasso}}(\lambda)_{j,l}}{\|\hF[j,l]\|_\infty} \phi_l \right\|_{L^1}   \ge \tau \right\}.
\end{equation}
with a fixed thresholding parameter $\tau\ge 0$. 

\textbf{The final step}  is to identify the correct support using the Time Evolution Error (TEE).   (i) From the candidate coefficient index set $\widehat\Lambda_{\tau}$, consider every subset $\Omega \subseteq \widehat\Lambda_{\tau}$.  For each $\Omega= \{j_1, j_2, \ldots, j_k\}$,  find the coefficients  $\widehat{\textbf{a}} = (0, 0, \widehat{a}_{j_1}, 0, \dots, \widehat{a}_{j_k}, \dots )$ by a  least-square fit such that $\widehat{\ba}_{\Omega} = \widehat{F}_{\Omega}^\dagger \widehat \bb$ and $\widehat{\ba}_{\Omega^\complement} = \mathbf{0}$.    
(ii)  Using these coefficients, construct the differential equation and numerically time evolve  
\[ u_t =  \mathcal{F} \widehat{\textbf{a}},\]
starting from the given initial data, for each $\Omega$.   It is crucial to use a smaller time step $\widetilde{\Delta t}\ll \Delta t$, where $\Delta t$ is the time spacing of the given data.  We use first-order forward Euler time discretization of the time derivative with time step 
$\widetilde{\Delta t}=O(\Delta x^r)$ where $r$ is the highest order of the spatial derivatives  associated with $\widehat{\ba}$.  
(iii) Finally,  calculate the time evolution error for each $\widehat{\textbf{a}}$:
\[ 
\text{ TEE} (\widehat{\textbf{a}}) := \sum_{i=1}^{N_1} \sum_{n=1}^{N_2} |\bar{u}_i^n - u_i^n| \Delta x \Delta t,
\]
where $\bar{u}_i^n$ is the numerically time evolved solution at $(x_i,t_n)$ of the PDE with the support $\Omega$ and coefficient $\widehat{\textbf{a}}$. 
We pick the subset  $\Omega$ and the corresponding coefficients  $\widehat{\ba}$, which give the smallest TEE, and  denote the recovered support as  $\widehat{\Lambda}$.  This is the output of the algorithm, which is the identified PDE. 
Algorithm \ref{algident} summarizes this procedure.
\begin{algorithm}
\caption{Identifying Differential Equations with Numerical Time evolution (IDENT) }
\label{algident}
\textbf{Input}: The discrete data $\{u_i^n | i=1, \dots, N_1 \text{ and } n=1,\dots,N_2\}$. \\ 
\textbf{[Step 1]}  Construct the empirical feature matrix $\hF$ and the empirical vector $\widehat\bb$ using ENO schemes. \\
\textbf{[Step 2]} Find a set of possible active features by the $L^1$ minimization \eqref{eqgLasso} followed by thresholding.  
\\
\textbf{[Step 3]} Pick the coefficient vector $\widehat{\textbf{a}}$, with minimum Time Evolution Error (TEE).\\ 
\textbf{Output:} The identified coefficients $\widehat{\ba}$ where $\widehat{\ba}_{\widehat{\Lambda}} = \widehat{F}_{\widehat{\Lambda}}^{\dagger} \widehat{\bb}$.
\end{algorithm}

We note that it is possible to skip the $L^1$ minimization step, and use TEE to recover the support of coefficients by considering all possible combinations from the beginning, however, the computational cost is very high.  The $L^1$ minimization helps to reduce the number of combinatorial trials, and make IDENT more computationally efficient. 
On the other hand, while $L^1$ minimization is effective in finding a sparse vector, $L^1$ alone is often not enough: (i) Zero coefficients in the true PDE may become non-zero in the minimizer of $L^1$. (ii) If active terms are chosen by a thresholding, results are sensitive to the choice of thresholding parameter, e.g., $\tau$ in (\ref{algthresholding}).
 (iii) The balancing parameter  $\lambda$ can effect the results.   (iv) If some columns of the empirical feature matrix $\widehat{F}$ are highly correlated, Lasso is known to have a larger support than the ground truth \cite{fannjiang2012coherence}.   TEE refines the results from Lasso, and relaxes the dependence on  the parameters.  
  
There are two fundamental ideas behind TEE:
\begin{enumerate}\vspace{-0.1cm}
\item{For nonlinear PDEs, it is impossible to isolate each term separately to identify each coefficient.  Any realization of PDE must be understood as a set of terms. }  \vspace{-0.1cm}
\item{If the underlying dynamics are identified by the true PDE, any refinement in the discretization of the time domain should not deviate from the given data.  This is the fundamental idea of the consistency, stability and convergence of a numerical scheme.  
}
\end{enumerate}
Therefore, the main effect of TEE is to evolve the numerical error from the wrongly identified differential terms. 
This method can be applied to linear or nonlinear PDEs.  The effectiveness of TEE can be demonstrated with an example. Assume that the solution $u$ is smooth and decays sufficiently fast at infinity, and consider the following linear equation with constant coefficients:
\[
\frac{\partial u}{\partial t} = a_0 u + a_1 \frac{\partial u}{\partial x} +\cdots + a_m \frac{\partial^m u}{\partial x^m}.
\]
After taking the Fourier transform for the equation and solving the ODE,  one can obtain the transformed solution: 
\[
\hat{u}(\xi,t)=\hat{u}(\xi,0) e^{a_0 t} e^{a_1\mathbf{i} \xi t} e^{-a_2 \xi^2 t }\cdots e^{a_m (\mathbf{i}\xi)^m t},
\]%
where $\mathbf{i}=\sqrt{-1}$ and $\xi$ is the variable in the Fourier domain.
If a term with an even-order derivative, such as $a_2 \frac{\partial^2 u}{\partial x^2}$, is mistakenly included in the PDE, 
it will make every frequency mode grow or decrease exponentially in time; if a term with an odd-order derivative, such as $a_1 \frac{\partial u}{\partial x}$,
is mistakenly included in the solution, it will introduce a wrong-speed ossicilation of the solution. In either case, the deviation from the correct solution grows fast in time, providing an efficient way to  distinguish the wrong terms.   Our numerical experiments show that TEE is an effective tool to correctly identify the coefficients.  Our first set of experiments are presented in subsection \ref{ssec:num_nonoise}.

\subsection{Recovery theory of Lasso, and new Noise-to-Signal Ratio (NSR) } \label{ssec:recovery}

In this subsection, we establish a performance guarantee of Lasso for the identification of PDEs with constant coefficients.  In the Step 2 of IDENT, Lasso is applied as  $L^1$ regularization in (\ref{eqgLasso}).  We consider the incoherence property proposed in \cite{donoho2001uncertainty}, and follow the ideas in \cite{fuchs2004sparse,tropp2004just,tropp2006just} to establish  a recovery theory.  While the details of the proof is presented in  Appendix \ref{A:recovery}, here we state the result which leads to the new definition of noise-to-signal ratio. 

For PDEs with constant coefficients, we set $L=1$ in (\ref{E:L}), and consider the standard Lasso:
\begin{equation}
\label{eqLasso}
\tag{Lasso}
\widehat{\ba}_{\text{Lasso}}(\lambda) ={\textstyle \arg \min_{\bz}} \left\{ \frac{1}{2} \| \widehat\bb - \widehat F_\infty \textbf{z} \|^2_2 + \lambda \|\bz\|_1 \right\}.
\end{equation}
If all columns of $\hF$ are uncorrelated, $\ba$ can be robustly recovered by Lasso. Let $\hF = [\hF[1] \ \hF[2] \ \ldots \ \hF[N_3]]$ where $\hF[j]$ stands for the $j$th column of $\hF$ in \eqref{E:constant_F}. To measure the correlation between the $j$th and the $l$th column of $\hF$, we use the pairwise coherence 
$$ \mu_{j,l}(\hF) = \frac{|\langle \hF[j] , \hF[l] \rangle|}{\|\hF[j]\|_2 \|\hF[l]\|_2}$$
and the mutual coherence of $\hF$ as in \cite{donoho2001uncertainty}:
\[
\mu(\hF) = \max_{j \neq l} \mu_{j,l}(\hF) = \max_{j\neq l} \frac{|\langle \hF[j] , \hF[l] \rangle|}{\|\hF[j]\|_2 \|\hF[l]\|_2}.
\] 
Since normalization does not affect the coherence, we have $ \mu_{j,l}(\hF_\infty)= \mu_{j,l}(\hF)$ and $\mu(\hF_\infty)=\mu(\hF)$.
The smaller $\mu(\hF)$, the less correlated are the columns of $\hF$, and $\mu(\hF) = 0$ if and only if the columns are orthogonal. Lasso will recover the correct coefficients if $\mu(\hF)$ is sufficiently small.

\begin{theorem}
\label{thmLasso}
Let $\mu =\mu(\hF)$, $w_{\max} = \max_j \|\hF[j]\|_\infty \|\hF[j]\|_{L^2}^{-1}$ and $w_{\min} = \min_j \|\hF[j]\|_\infty \|\hF[j]\|_{L^2}^{-1}$. Suppose the support of $\ba$ contains no more than $s$ indices, $\mu(s-1) < 1$ and 
$$\frac{\mu s}{1-\mu(s-1)} < \frac{w_{\min}}{w_{\max}}.$$
Let 
\begin{equation}
\label{thmlambda}
\lambda = \frac{[1-(s-1)\mu] }{w_{\min}[1-\mu(s-1)] - w_{\max} \mu s  }\cdot \frac{\varepsilon^+}{\Delta x \Delta t}.
\end{equation}
Then
\begin{enumerate}
\item[1)] the support of $\widehat{\ba}_{\text{Lasso}}(\lambda)$ is contained in the support of $\ba$;

\item[2)] the distance between $\widehat{\ba}_{\text{Lasso}}(\lambda)$ and $\ba$ satisfies
\begin{equation}
\label{thmdist}
\max_j \|\hF[j]\|_{L^2} \left|\|\hF[j]\|_\infty^{-1}\widehat{\ba}_{\text{Lasso}}(\lambda)_j-a_j\right| 
\le \frac{w_{\max} + \varepsilon/\sqrt{\Delta t \Delta x}}{w_{\min}[1-\mu(s-1)] -w_{\max} \mu s } \varepsilon;
\end{equation}

\item[3)] if 
\begin{equation}
\label{thmnsr}
\min_{j:\ a_j\neq 0} \|\hF[j]\|_{L^2} | a_j| >  \frac{w_{\max}+\varepsilon/\sqrt{\Delta t \Delta x}}{w_{\min}[1-\mu(s-1)] -w_{\max} \mu s } \varepsilon,
\end{equation}
then the support of $\widehat{\ba}_{\text{Lasso}}(\lambda)$ is exactly the same as the support of $\ba$.
\end{enumerate}
\end{theorem}

Theorem \ref{thmLasso} shows that Lasso will give rise to the correct support when the empirical feature matrix $\hF$ is incoherent, i.e. $\mu(\hF) \ll 1$, and all underlying coefficients are sufficiently large compared to  noise. 
When the empirical feature matrix is coherent, i.e., some columns of $\hF$ are correlated, it has been observed that $\widehat{\ba}_{\text{Lasso}}(\lambda)$ are usually supported on $\supp(\ba)$ and the indices that are highly correlated with $\supp(\ba)$ \cite{fannjiang2012coherence}. We select possible features by thresholding in \eqref{algthresholding} which is equivalent to $\widehat\Lambda_{\tau} : = \left\{j : \|\hF[j]\|_{L^1} \|\hF[j]\|_{\infty}^{-1} |\widehat{\ba}_{\text{Lasso}}(\lambda)_j| \ge \tau \right\}$ in the case of constant coefficients. After this, TEE  is an effective tool in complement of Lasso to distinguish the correct features from the wrong ones.  The details of Theorem \ref{thmLasso} can be found in Appendix \ref{A:recovery}.

This analysis also gives rise to a \textbf{new noise-to-signal ratio}:
\begin{equation}
\label{eqnsr}
\text{Noise-to-Signal Ratio (NSR)} := \frac{\|\hF \ba -\widehat\bb\|_{L^2}}{\min_{j:\ a_j\neq 0} \|\hF[j]\|_{L^2} | a_j| }.
\end{equation}
The definition is derived from \eqref{thmnsr}, showing that the signal level is contributed by the minimum of the product of the coefficient and the column norm in the feature matrix - $\min_{j:\ a_j\neq 0} \|\hF[j]\|_{L^2} | a_j|$.  This term represents the dynamics resulted from the feature.
It is important to consider the multiplication rather than the magnitude of the coefficient only.    We also use this new definition of NSR to measure the level of noise in the following sections, which gives a more consistent representation.

\subsection{First set of  IDENT experiments }\label{ssec:num_nonoise}

We present the first set of  numerical experiments to illustrate the effects of IDENT.  Here data are sampled from exact or simulated solutions of PDEs with constant coefficients.   For boundary conditions, we use zero Dirichlet boundary conditions throughout the paper.  Modification to periodic or other boundary conditions is trivial, and numerical schemes with periodic boundary conditions  can achieve higher accuracy, for the cases without noise.   We observe that the Lasso results are not very sensitive to the choice of $\lambda$ using TEE, and we set $\lambda =500$ in all experiments.

\begin{figure}
\centering
\begin{tabular}{ccc}
(a) Given data & 
(b) Coherence pattern & 
(c) Result from Lasso \\
\includegraphics[width = 2.05in]{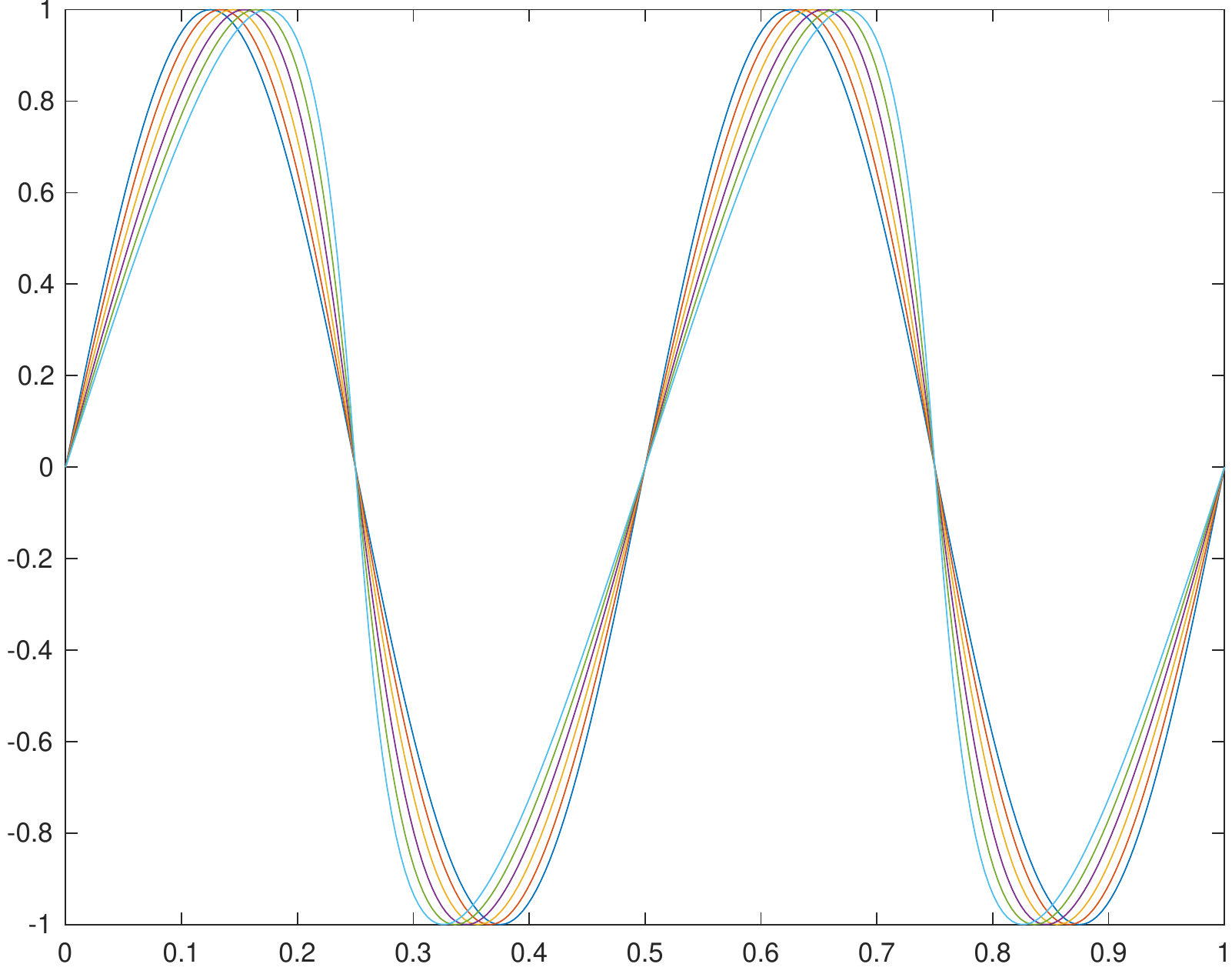} &
\includegraphics[width = 2.05in]{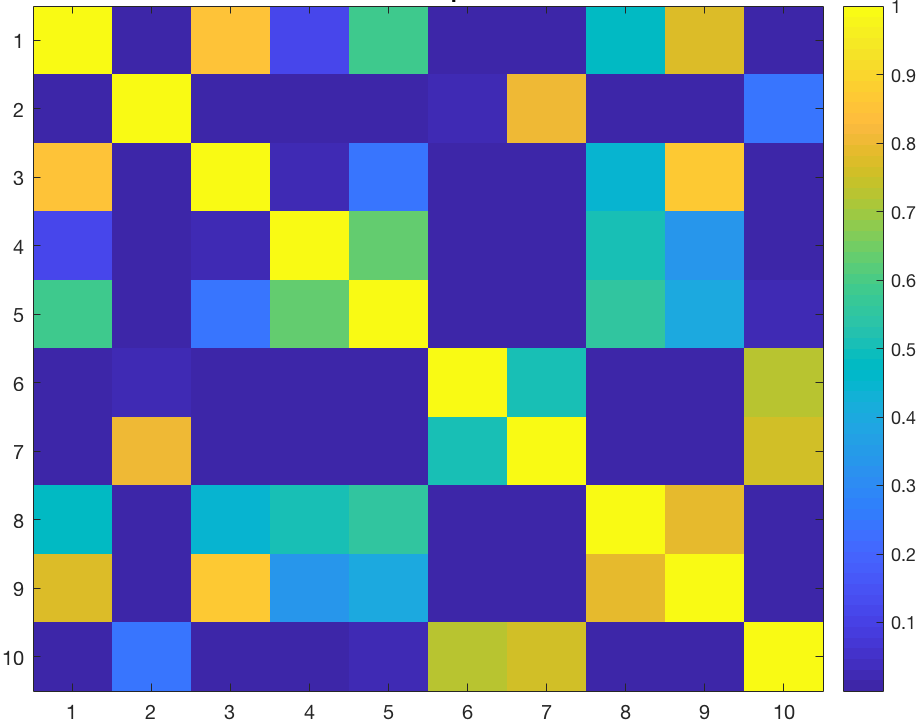} &
\includegraphics[width = 2.05in]{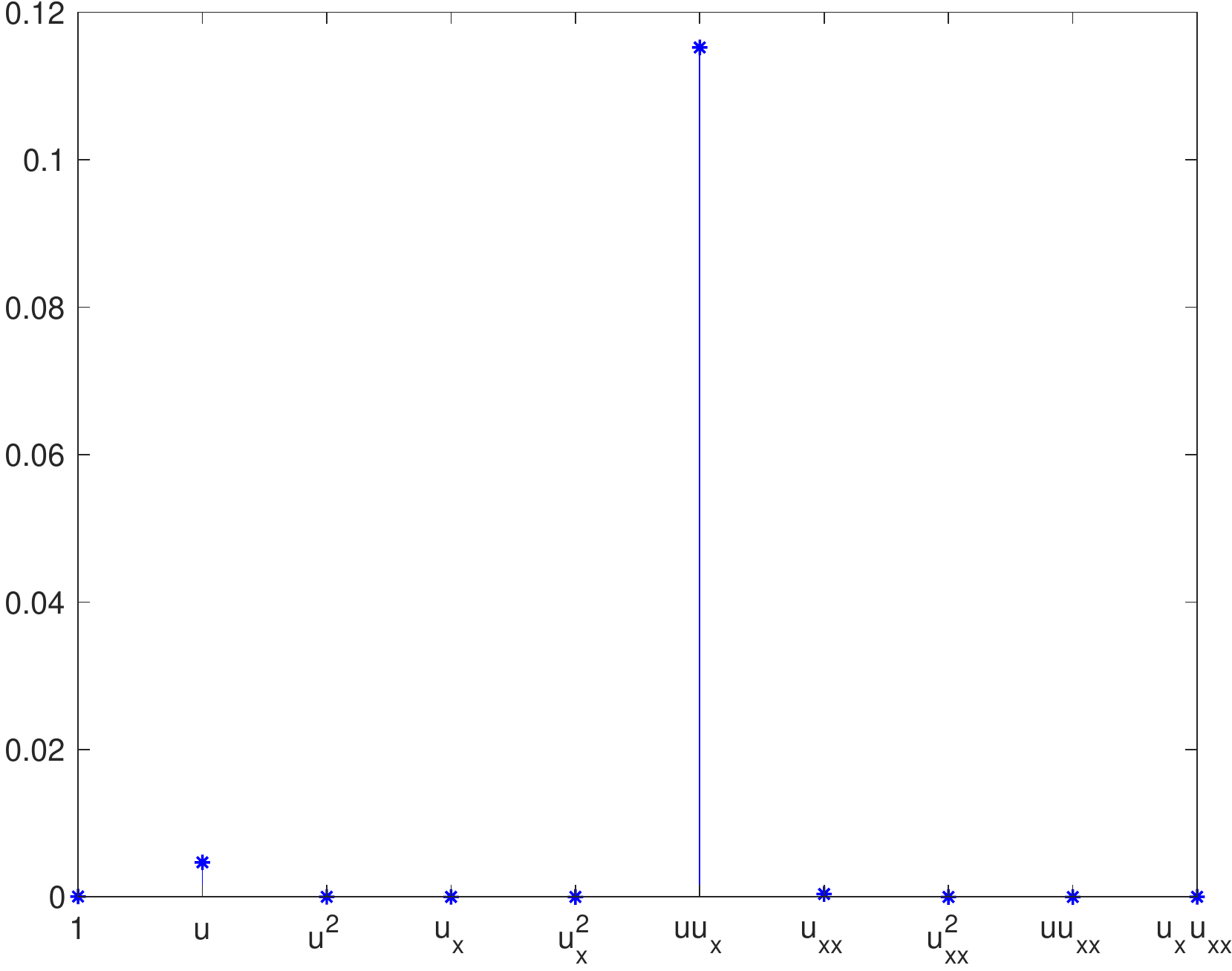}
\end{tabular}
\caption{Experiment with the Burger's equation \eqref{E:burger}.  (a) The given data  are sampled from true analytic solution. (b) The coherence pattern of $\hF$.  (c) Normalized coefficient magnitudes from Lasso.  Two possible features are identified, which are $u$ and $uu_x$.  }
\label{Fig-BurgerExactDemo}
\end{figure}

The first experiment is on the  Burger's equation with Dirichlet boundary conditions:
\begin{align}
& u_t + \left( \frac{u^2}{2}\right)_x =0, \ x \in [0,1] 
 \label{E:burger} \\
 &u(x,0)= \sin 4\pi x   \text{ and }u(0,t) = u(1,t) = 0. \nonumber
\end{align}
The given data are sampled from the true  analytic solution, shown in Figure \ref{Fig-BurgerExactDemo} (a), with $\Delta x = 1/56$ and $\Delta t = 0.004$, for  $t \in [0,0.05]$. Figure \ref{Fig-BurgerExactDemo} (b) displays the coherence pattern of the empirical feature matrix: the absolute values of $\hF_{\rm unit}^* \hF_{\rm unit}$ where $\hF_{\rm unit}$ is obtained from $\hF$ with column normalized to unit $L^2$ norm. This pattern shows the correlation between any pair of the columns in $\hF$. (c) shows the normalized coefficient magnitudes $\{\|\hF[j]\|_{L^1} \|\hF[j]\|_{\infty}^{-1} |\widehat{\ba}_{\text{Lasso}}(\lambda)_j|\}$ after $L^1$ minimization. 
The magnitudes of $u$ and $uu_x$ are not negligible, so they are picked as a possible set of active features in $\widehat\Lambda_{\tau}$.  Then,  TEE is computed for all subsets $\Omega \subseteq \widehat\Lambda_{\tau}$, i.e., $u_t = a u$, $u_t = b uu_x$ and $u_t = cu +d uu_x$ where the coefficients $a,b,c,d$ are calculated by least-squares:
\begin{center} 
\begin{tabular}{ | c | c | c |   }    
\hline
     Active terms  &  Coefficients of active terms by least-squares & TEE \\ \hline
    $u$ & $0.27$ & $78.76$ \\ 
   \textcolor{red}{ $uu_x$} & \textcolor{red}{$ -0.99$} & \textcolor{red}{$0.48$}
     \\ 
    $[u \ uu_x]$ & $ [0.10 \ -0.99]$ & $1.40$
    \\
    \hline
  \end{tabular}
\end{center}
The red line with only $uu_x$ term has the smallest TEE, and therefore is identified as the result of IDENT.  Since the true PDE is $u_t = - u u_x$,  the computed result shows a small coefficient error.

The second experiment is on the Burger's equation with a diffusion term:
\begin{align}
& u_t + \left( \frac{u^2}{2}\right)_x =0.1u_{xx},\ x \in [0,1]  
\label{E:burger_diff}
 \\
 &u(x,0)= \sin 4\pi x   \text{ and }u(0,t) = u(1,t) = 0. \nonumber
\end{align}
The given data are simulated with a first-order explicit method where $\delta x = 1/256$ and $\delta t = (\delta x)^2$ for  $t \in [0,0.1]$. Data are  downsampled from the numerical simulation by a factor of $4$ such that $\Delta x = 4\delta x$ and $\Delta t = 4\delta t$.  (We explore the effects of downsampling in more detail  in Section \ref{sec:noise}.)
\begin{figure}
\centering
\begin{tabular}{ccc}
(a) Given data & 
(b) Coherence pattern & 
(c) Result from Lasso \\
\includegraphics[width = 2.05in]{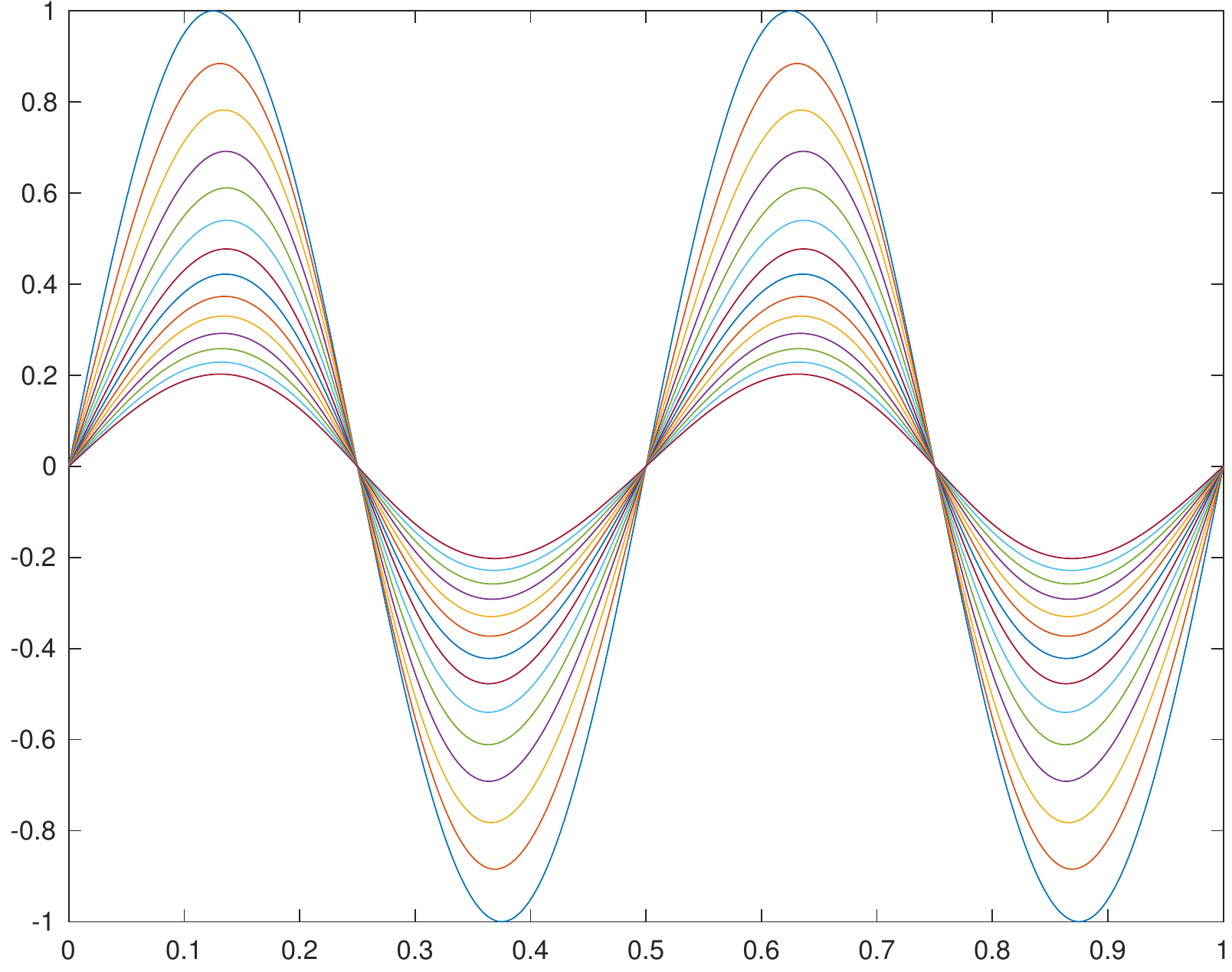} &
\includegraphics[width = 2.05in]{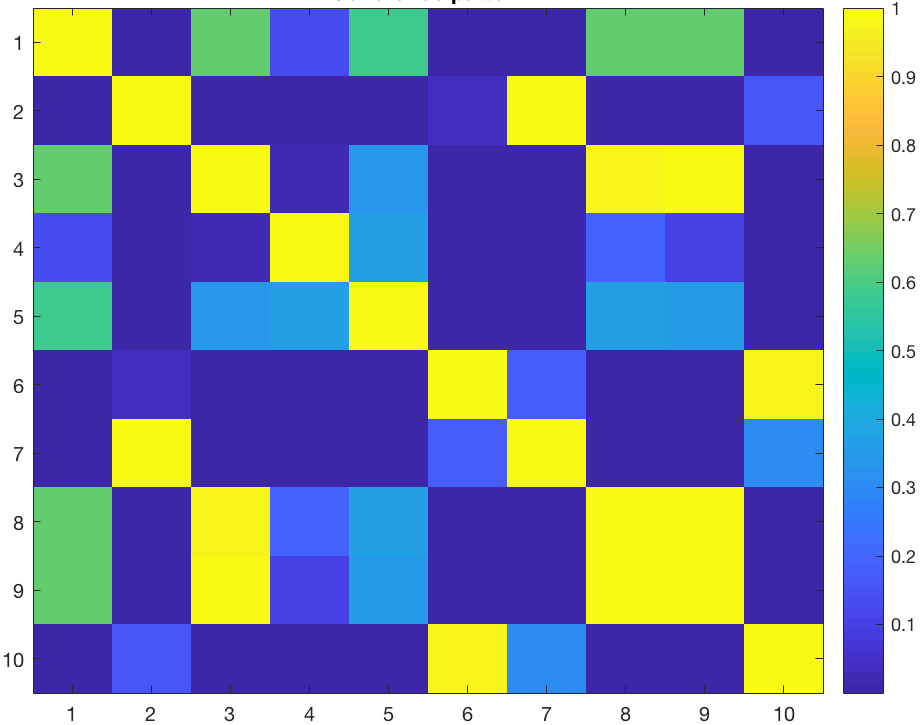} &
\includegraphics[width = 2.05in]{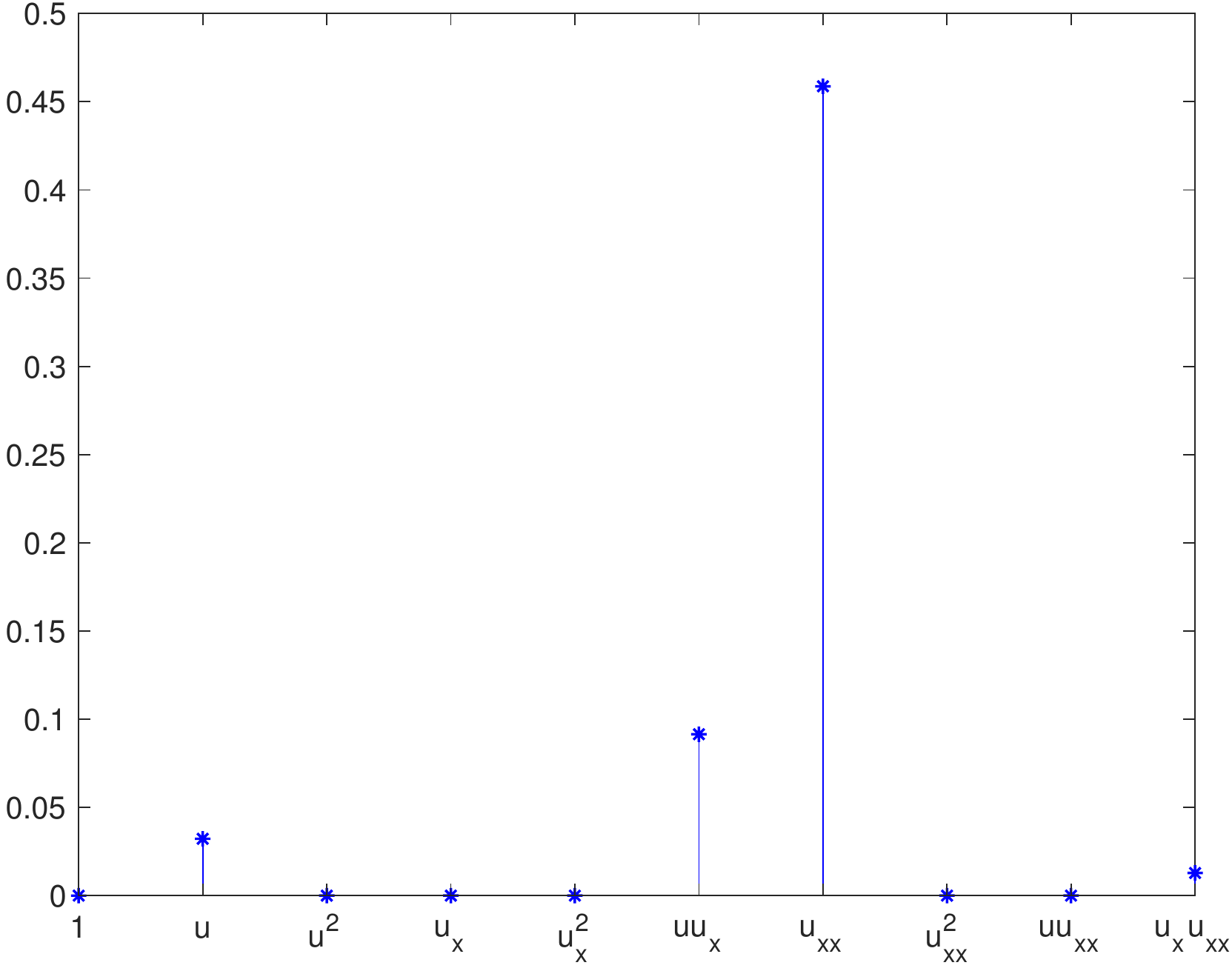}
\end{tabular}
\caption{Experiment with Burger's equation with a diffusion term \eqref{E:burger_diff}.  (a) The given data are numerically simulated and downsampled.   (b) shows that $u$ and $u_x u_{xx}$ are highly correlated with $u_{xx}$ and $u u_x$, respectively. From (c), four terms $u, uu_x, u_{xx} $ and $u_x u_{xx}$ are selected for TEE.   }
\label{Fig-BurgerDiffDemo}
\end{figure}
Figure \ref{Fig-BurgerDiffDemo} (a) shows the given data,  (b) displays the coherence pattern of $\hF$, and (c) shows the normalized coefficient magnitudes $\{\|\hF[j]\|_{L^1} \|\hF[j]\|_{\infty}^{-1} |\widehat{\ba}_{\text{Lasso}}(\lambda)_j|\}$.  In this case, the coherence pattern in (b) shows that $u$ and $u_x u_{xx}$ are highly correlated with $u_{xx}$ and $u u_x$, respectively, and therefore all four terms $u,uu_x,u_{xx},u_x u_{xx}$ are identified as meaningful ones by Lasso in (c).   Considering TEE for each subset refines these results:
\begin{center}
 \begin{tabular}{ | c | c | c |  }
    \hline
 Active terms  & Coefficients of active terms by least-squares & TEE  \\ \hline
    $u$ & $-16.08$ & $3709.77$      \\ 
    $u u_x$ & $-0.34$ & $67092.21$      \\ 
     $u_{xx}$ & $0.10$ & $4345.98$      \\ 
     $u_xu_{xx}$ & $-0.0008$ & $\infty$      \\ 
    $[u \ uu_x]$ & $ [-16.17 \ -0.50]$ & $2120.14$   \\
    $[u \ u_{xx}]$ & $ [-21.42 \ -0.03]$ & $1.49 \times 10^{26}$   \\
     $[u \ u_x u_{xx}]$ & $ [-16.44 \ -0.003]$ & $\infty$   \\
    \textcolor{blue}{$[uu_x \ u_{xx}]$} & \textcolor{blue}{$ [-1.00 \ 0.10]$} & \textcolor{blue}{$82.33$}   \\
      $[uu_{x} \ u_x u_{xx}]$ & $ [-12.03 \ -0.07]$ & $\infty$   \\
         $[u_{xx} \ u_x u_{xx}]$ & $ [-0.10 \ -0.006]$ & $371.08$   \\
    $[u \ uu_x \ u_{xx}]$ & $ [-0.10\ -1.00 \ 0.10]$ & $83.73$   \\
    $[u \ uu_x \ u_x u_{xx}]$ & $ [-15.86\ -1.03 \ -0.003]$ & $\infty$   \\
    $[u \ u_{xx} \ u_x u_{xx}]$ & $ [-0.58\ 0.10 \ 0.006]$ & $367.68$   \\
       \textcolor{red}{ $[uu_x \ u_{xx} \ u_x u_{xx}]$} & \textcolor{red}{$ [-1.00 \ 0.10 \ -1.35\times 10^{-5}]$} & \textcolor{red}{$82.29$}   \\
         $[u \ uu_x \ u_{xx} \ u_x u_{xx}]$ & $ [-0.11 \ -1.00 \ 0.10 \ -2.8\times 10^{-5}]$ & $83.85$   \\
    \hline
  \end{tabular}
\end{center}
The red line is the result of IDENT, while the blue line is the ground truth.   The TEE of  $[uu_x \ u_{xx} \ u_x u_{xx}]$ is the smallest, which is comparable with the TEE of the true equation with $[ uu_x \  u_{xx}]$.   One wrongly identified term in red, $u_x u_{xx}$, has a coefficient magnitude of $ -1.35 \times 10^{-5}$ which is negligible.  The level of error in the identification is also related to the total error to be explored in (\ref{E:enoise}).   Without TEE,  if all four terms are used from $L^1$ minimization, an additional wrong term $u$ is identified  with the coefficient $-0.11$. This is comparable to other terms with coefficients like -1 or 0.1, and cannot be ignored. 

Theorem \ref{thmLasso} proves that the identified coefficients from Lasso will converge to the ground truth as $\Delta t \rightarrow 0$ and $\Delta x \rightarrow 0$ (see Equation \eqref{E:errorbound} and \eqref{thmdist}), when there is no noise and the empirical feature matrix has a small coherence. Figure \ref{Fig-BurgerCoeffVersusDeltax}  shows the recovered  coefficients from Lasso versus $\Delta t$ and $\Delta x$ for the Burger's equation \eqref{E:burger} and Burger's equation with diffusion \eqref{E:burger_diff}.  
In Figure \ref{Fig-BurgerCoeffVersusDeltax} (a), data are sampled from the analytic solution of the Burger's equation \eqref{E:burger} with spacing $\Delta x = 2^{k}$ for $k=-12, -11, \dots, -5$ respectively and $\Delta t = \Delta x$ for $t \in [0,0.05]$.   Figure \ref{Fig-BurgerCoeffVersusDeltax} (a) shows the result from Lasso, namely, $\{ \|\hF[j]\|_{\infty}^{-1} \widehat{\ba}_{\text{Lasso}}(\lambda)_j\}$, versus $\log_2 \Delta x$. Notice that the coefficient of $u u_x$ converges to $-1$ and all other coefficients converge to $0$ as  $\Delta t$ and $\Delta x$ decrease.
For Figure \ref{Fig-BurgerCoeffVersusDeltax} (b), data are sampled from the numerical simulation of the Burger's equation with diffusion in \eqref{E:burger_diff}, where the PDE is solved by a first-order method with $\delta x = 2^{-10}$ and $\delta t = (\delta x)^2$ for $t \in [0,0.1]$.  Data are sampled with $\Delta x = 2^{-10},2^{-9},2^{-8},2^{-7},2^{-6}$ respectively, and $\Delta t = (\Delta x)^2$ for  $t \in [0,0.1]$.  Figure \ref{Fig-BurgerCoeffVersusDeltax} (b) shows the recovered coefficients from Lasso versus $\log_2 \Delta x$.  Here the coefficients of $u u_x$ and $u_{xx} $ converge to $-1$ and $0.1$ respectively, and all other coefficients except the one of $u$, converge to $0$, as $\Delta t$ and $\Delta x$ decrease.  
The coefficient of $u$ does not converge to $0$ because data are generated by a first order numerical scheme with the error $O[\delta t + (\delta x)^2]$, i.e.,  the error for Lasso $\|\be\|_{L^2}$ does not decay to $0$ as $\Delta t$ and $\Delta x$ decrease. We further discuss this aspect of data generations in Section \ref{sec:noise}.
 
\begin{figure}
\centering
\begin{tabular}{cc}
(a) Burger's equation & 
(b) Burger's equation with diffusion \\
\includegraphics[width = 2.7in]{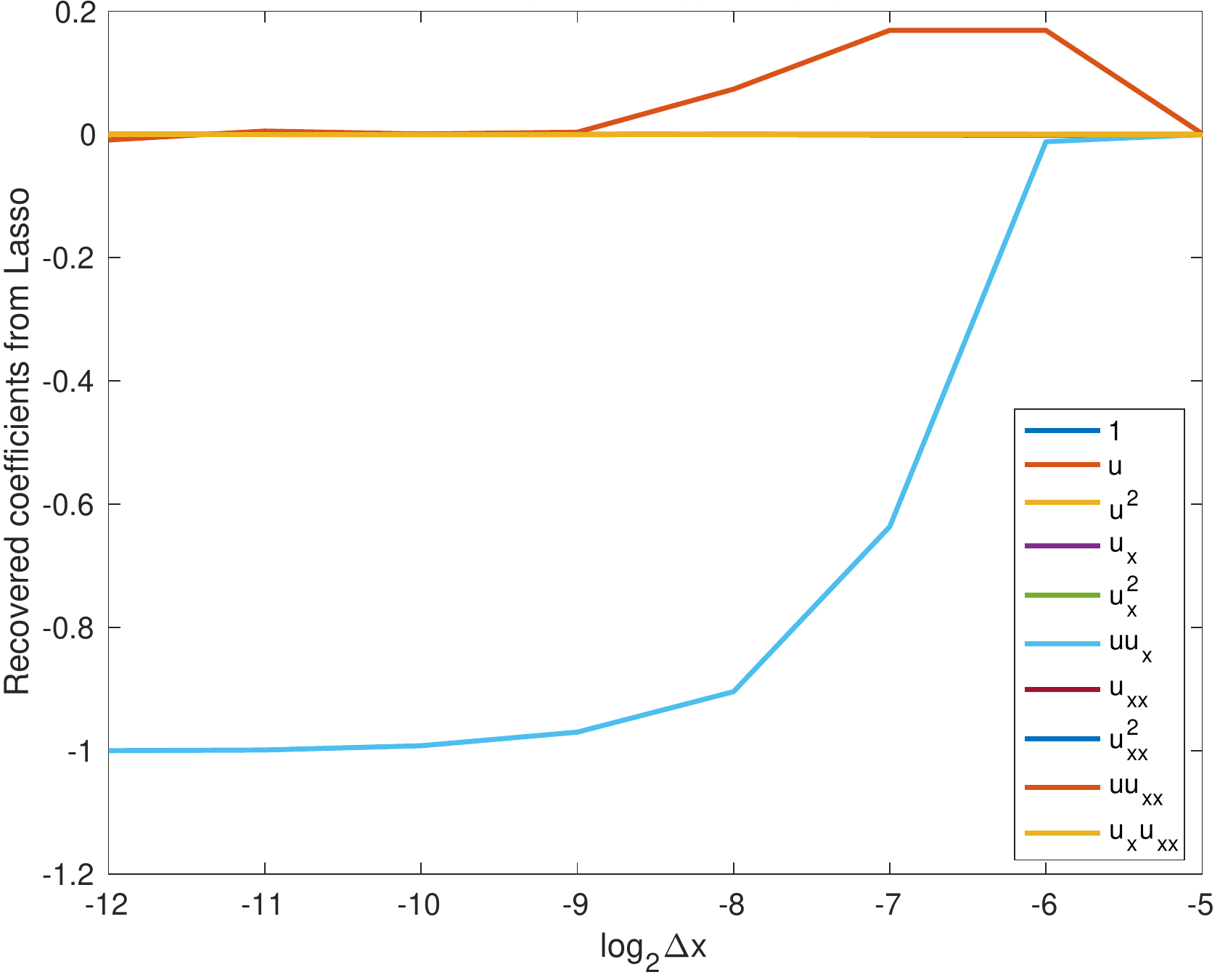} &
\includegraphics[width = 2.7in]{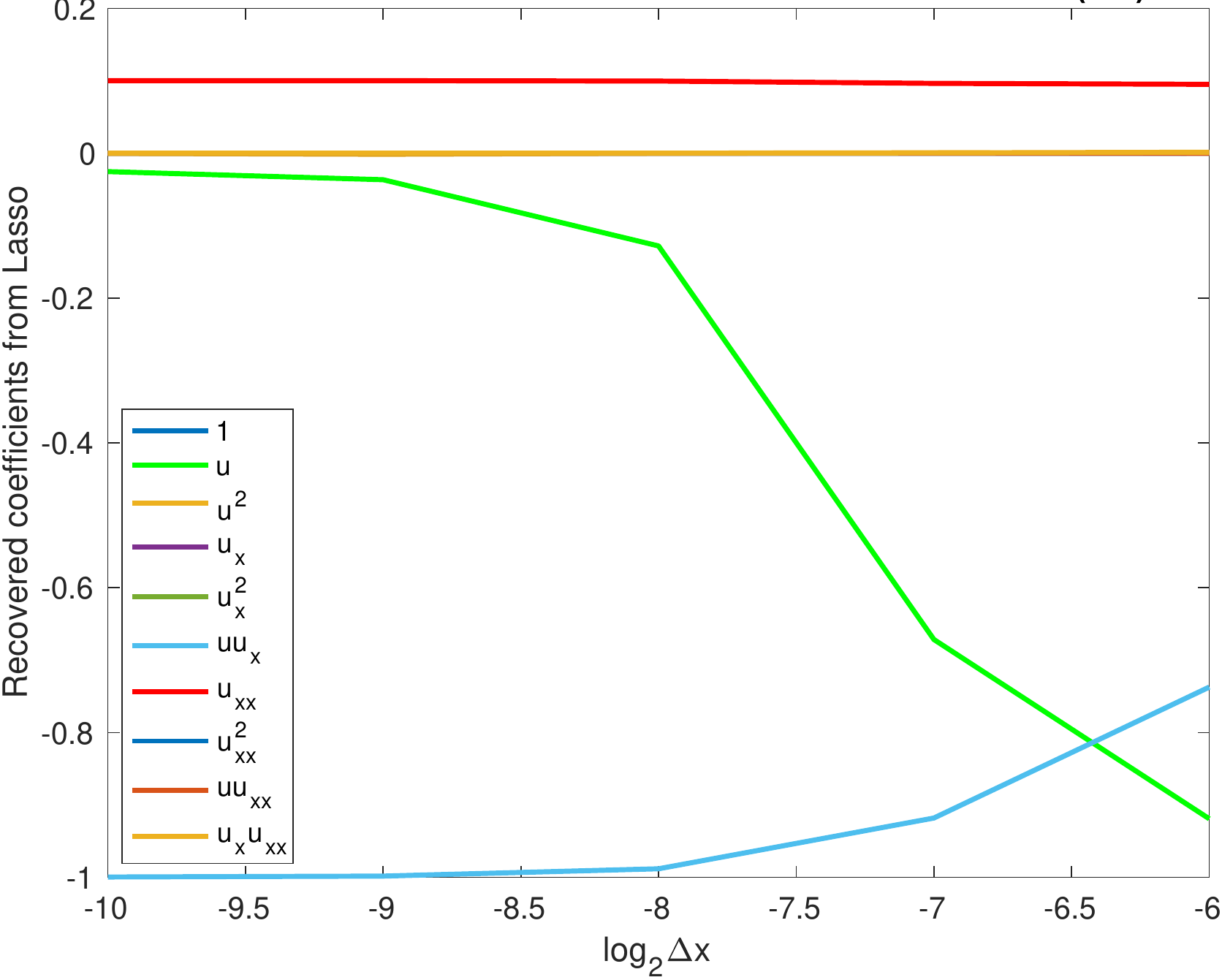} 
\end{tabular}
\caption{Identified coefficients from Lasso (Step 2 only) versus $\log_2\Delta x$.  In (a),  as $\Delta t$ and $\Delta x$ decrease (from right to left),  the coefficient of $u u_x$  correctly converges to 1, and all other terms correctly converge to 0.  In (b), as $\Delta t$ and $\Delta x$ decrease (from right to left), while the coefficients of $u_{xx}$ and $u u_x$ correctly converge to 0.1 and -1 respectively, one wrong term $u$ does not converge to 0, due to the error from data generations. }
\label{Fig-BurgerCoeffVersusDeltax}
\end{figure}

 \section{Noisy data, Downsampling and IDENT}\label{sec:noise}  

As noticed above, identification results depend on the accuracy of the given data.  
In this section,  we explore the effects of inaccuracy in data generations, noise and downsampling.  We derive an error formula to incorporate the errors arising from these three aspects, which provides a theoretical guidance of the difficulty of identification.  

The given data $\{\widetilde u_i^n\}$ may contain noise, such that 
$$\widetilde u_i^n = u_i^n + \xi_i^n,$$
where the noise $\xi_i^n$ arises from inaccuracy in data generations and/or the measurement error. 
Consider a $r$th order PDE with the highest-order spatial derivative $\partial_x^r u$.  
Suppose data are numerically simulated by a $q$th-order method with time step $\delta t$ and spacial spacing $\delta x$, and the measurement error is independently drawn from the normal distribution with mean $0$ and variance $\sigma^2$. Then
$$\xi_i^n =  O(\delta t + \delta x^q + \sigma).$$
We use the five-point ENO method to approximate the spatial derivatives in the empirical feature matrix $\hF$ in Section \ref{sec:ident}. In general, one could interpolate the data with a $p$th-order polynomial and use the derivatives of the polynomial to approximate $u_x$ and $u_{xx}$, etc. In this case, the error for the $k$th-order spacial derivative $\partial_x^k u$ is $O(\Delta x^{p+1-k})$.

The error for Lasso given by \eqref{E:linear1} is $\mathbf{e} =  \bb -\widehat{\bb} + (\widehat{F}-F)\ba +\bbeta $, where 
$\bb -\widehat{\bb}$ is from the approximation of $u_t$,  $(\widehat{F}-F)\ba$ is from the approximation of the spatial derivatives of $u$, and $\bbeta$ arrises from the finite element basis expansion for the varying coefficients. If $u,u_x,u_{xx},\ldots$ are bounded, these terms satisfy
\begin{align*}
\| (\widehat{F}-F)\ba\|_\infty \le O\left(\Delta x^{p+1-r} + \frac{\delta t + \delta x^q +\sigma}{\Delta x^r}\right) \text{ and } 
\|\bb -\widehat{\bb}\|_\infty \le O\left(\Delta t + \frac{\delta t + \delta x^q +\sigma}{\Delta t}\right),
\end{align*}
and $\|\bbeta\|_\infty  = O\left( 1/ L \right)$
so that 
\begin{equation}
\label{E:enoise}
\|\be\|_{L^2} \le \varepsilon, \text{ with }
\varepsilon = O\left(\Delta t + \Delta x^{p+1-r}+ \underbrace{\frac{\delta t + \delta x^q}{\Delta t}+ \frac{\delta t + \delta x^q }{\Delta x^r}}_{\text{errors from data generations}} + \underbrace{\frac{\sigma}{\Delta t} + \frac{\sigma}{\Delta x^r}}_{\text{measurement noise}} + \frac 1 L\right). 
\end{equation}
This error formula suggests the followings: 
\begin{description}

\item[Sensitivity to measurement noise] Finite differences are sensitive to measurement noise since  Gaussian noise with mean $0$ and variance $\sigma^2$ results in $O(\sigma/\Delta t + \sigma/\Delta x^r)$ in the error formula. Higher-order PDEs are more sensitive to measurement noise than lower-order PDEs. Denoising the given data is helpful to Lasso in general. 

\item[Downsampling of data] In applications, the given data are downsampled such that $\Delta t =  C_t \delta t$ and $\Delta x = C_x \delta x$ where $C_t$ and $C_x$ are the downsampling factors in time and space. Downsampling could help to reduce the error depending on the balances among the terms in \eqref{E:enoise}.
\end{description}

We further explore these effects below.  We propose an order preserving denoising method in Section \ref{subsec:lsma}, experiment IDENT with noisy data in Section \ref{subsecnoisyexperiment}, and discuss the downsampling of data in Section \ref{subsec:down}.

\subsection{An order preserving denoising method: Least-Squares Moving Average}   \label{subsec:lsma}

Our error formula in \eqref{E:enoise} shows that a small amount of noise can quickly increase the complexity of the recovery, especially for higher-order PDEs. Denoising is helpful in general. We propose an order preserving method which keeps the order of the approximation to the underlying function, while smooths out possible noise.  

Let the data $\{d_i\}$ be given on a one-dimensional uniform grid $\{x_i\}$ and define its five-point moving average as  $\tilde{d_i} = \frac{1}{5}\sum_{l=0,\pm1, \pm2} d_{i+l}$ for all $i$. 
At each grid point $x_i$, we determine a quadratic polynomial $p(x)= a_0 + a_1 (x-x_i) + a_2 (x-x_i)^2$ fitting the local data, which preserves the order of smoothness, up to the degree of polynomial. 
There are a few possible choices for denoising, such as (i) Least-Squares Fitting (LS): find $a_0, a_1$ and $a_2$ to minimize the functional
$ F( a_0, a_1, a_2 )=\sum_{{\rm some} \; j\; {\rm near} \; i} (p(x_j)- d_j)^2$;
(ii) Moving-Average Fitting (MA): find $a_0, a_1$ and $a_2$, such that the local average of the fitted polynomial matches with the local average of the data, $ {1}/{5} \sum_{l=0,\pm1, \pm2} p(x_{j+l}) = \tilde{d_j},$ for $j=i, i\pm 1$ (or another set of $3$ grid points near $\{x_i\}$).  
The polynomial generated by LS may not represent the underlying true dynamics.  Moving average fitting is better in keeping the underlying dynamics, however, the matrix may be ill-conditioned when a linear system is solved to determine $a_0,a_1,a_2$.  

We propose to use (iii) Least-Squares Moving Average (LSMA): find $a_0, a_1$ and $a_2$ to minimize the functional
\[G( a_0, a_1, a_2 )=  \sum_{j=i,i\pm1, i\pm2} \left\{\left[\frac{1}{5} \sum_{l=0,\pm1, \pm2} p(x_{j+l})\right] - \tilde{d_j}\right\}^2.\]
The condition number of this linear system tends to be better in comparison with MA, because $j$ is chosen from a larger set of indices. 
%
This LSMA denoising method preserves the approximation order of data and can easily be incorporated into numerical PDE techniques. MA fitting and LSMA are similar to the non-oscillatory polynomial reconstruction from cell averages which is a key step in high-resolution shock capturing schemes, see e.g. 
\cite{barth1990higher,ENO87,hu1999weighted}. The quadratic polynomials computed by the methods above are locally third-order approximation to the underlying function. We prove that, if the given data are sampled from a third-order approximation to a smooth function, then LSMA will keep the same order of the approximation. This theorem can be easily generalized to any higher order, we kept to $3$rd order to be consistent with our experiments in this paper.

\begin{theorem}
If data are given as a $3$rd order approximation to a smooth function, with or without additive noise, then denosing the data (to obtain a piecewise quadratic function) with the  Least-Squares Moving Average (LSMA) method will keep the same order of accuracy to the function. 
\end{theorem}

\begin{proof} 
Let $f(x)$ be the smooth function. The proof can be done by comparing the quadratic function to that of the Taylor expansion of $f(x)$ at a grid point $x_{i_0}$, see e.g., \cite{LSTZ07}.
Let $p(x)=a_0+a_1(x-x_{i_0})+a_2(x-x_{i_0})^2$ be the quadratic function to be determined near $x_{i_0}$. The least-squares method solves the linear system $A^T(Ac-b)=0$ for the coefficient vector $c=[a_0, a_1,a_2]^T $, where $A$ is a $5\times 3$ matrix whose rows can be written as 
$[1, \frac15 \sum_{j=0,\pm 1, \pm 2}(x_{i+j}-x_{i_0}), \frac15 \sum_{j=0,\pm 1, \pm 2}(x_{i+j}-x_{i_0})^2], $
 for $i=i_0-2, \cdots, i_0+2$, and
$b=[\tilde{d}_{i_0-2},\cdots, \tilde{d}_{i_0+2}]^T.$

According to the assumption we have 
$$\tilde{d_i} = f(x_{i_0})+f'(x_{i_0})\frac15 \sum_{j=0,\pm 1, \pm 2}(x_{i+j}-x_{i_0})+
\frac12 f''(x_{i_0})\frac15 \sum_{j=0,\pm 1, \pm 2}(x_{i+j}-x_{i_0})^2+O(\Delta x^3),
$$
for any grid point $x_i$ near $x_{i_0}$, i.e.,  $|x_i-x_{i_0}|= O(\Delta x)$. Let $s=[f(x_{i_0}), f'(x_{i_0}), \frac12 f''(x_{i_0})]^T$. We have
$$A^T(Ac-b)= HB^T\{BH(c-s)+O(\Delta x^3)\},$$ 
where
$H$ is the $3\times 3$ diagonal matrix ${\rm diag}\{1, \Delta x, \Delta x^2\} $, and
$$B=A H^{-1}=\left [
\begin{array}{lcr}
&\vdots & \\
1 & \frac15 \sum_{j=0,\pm 1, \pm 2}\frac{x_{i_0+j}-x_{i_0}}{\Delta x} & \frac15 \sum_{j=0,\pm 1, \pm 2}\frac{(x_{i_0+j}-x_{i_0})^2}{\Delta x^2} \\
&\vdots & 
\end{array}
\right ].
$$
Therefore $H(c-s)= (B^TB)^{-1}B^T\cdot O(\Delta x^3)$. Note that $B$ is independent of $\Delta x$, we have $|p(x)-f(x)|= O(\Delta x^3)$ for all $x$ such that $|x-x_{i_0}|= O(\Delta x)$.
\end{proof}


\subsection{IDENT experiments for noisy data}\label{subsecnoisyexperiment}
We next present numerical experiments with noisy data. We say $P$ percent Gaussian noise is added to the noise-free data $\{u_i^n : i=1,\ldots,N_1 \text{ and } n=1,\ldots,N_2\}$, if the observed data are $\{\widetilde u_i^n\}$ where $\widetilde u_i^n = u_i^n + \xi_i^n$ and  $\xi_i^n \sim \mathcal{N}(0,\sigma^2)$ with 
$\sigma = \frac{P}{100} \sqrt{\sum_{i=1}^{N_1} \sum_{n=1}^{N_2} |u_i^n|^2}/\sqrt{N_1 N_2}$.

\begin{figure}
\centering
\begin{tabular}{ccc}
(a) Given data & 
(b) Coherence pattern & 
(c) Result from Lasso \\
\includegraphics[width = 2.05in]{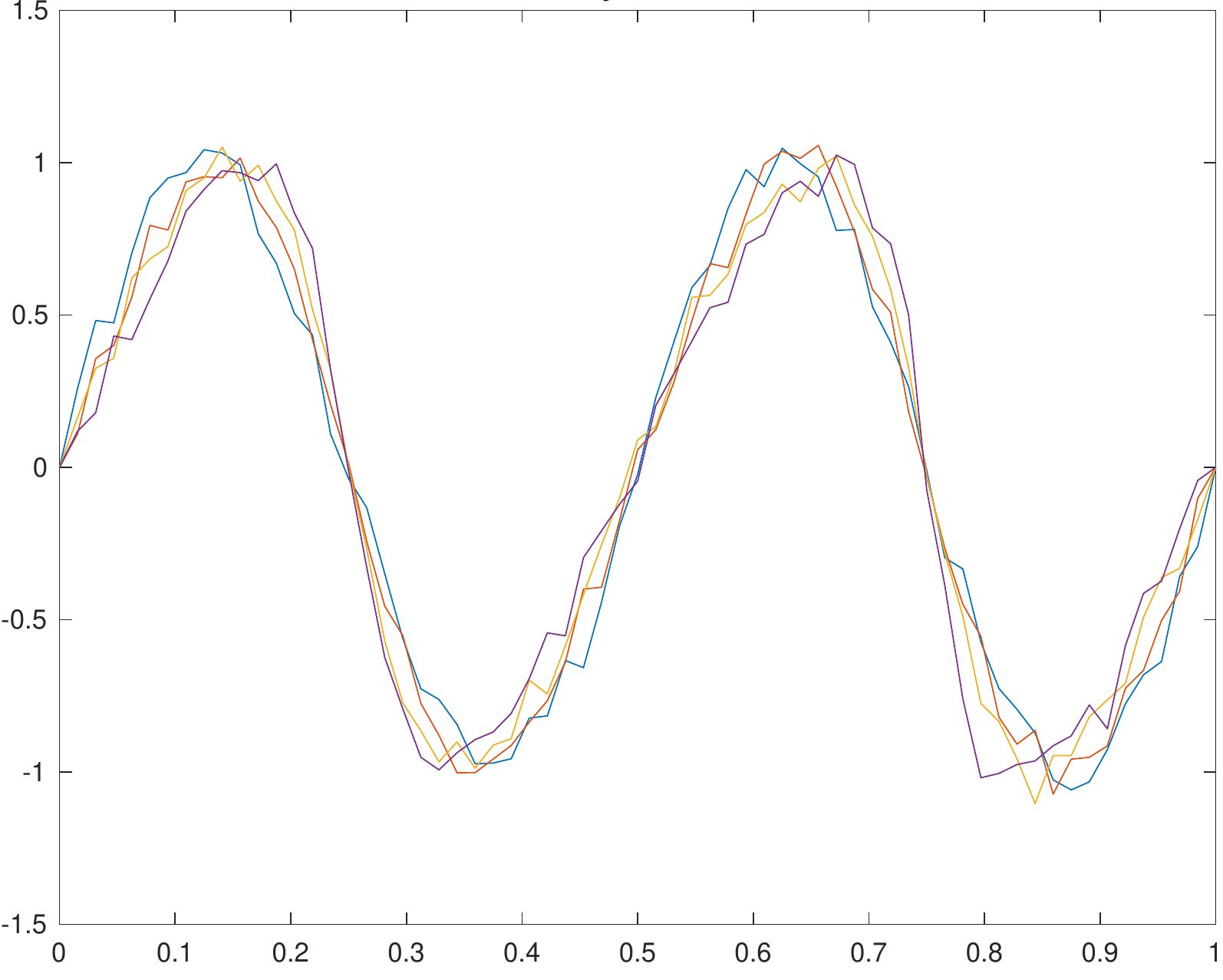} &
\includegraphics[width = 2.05in]{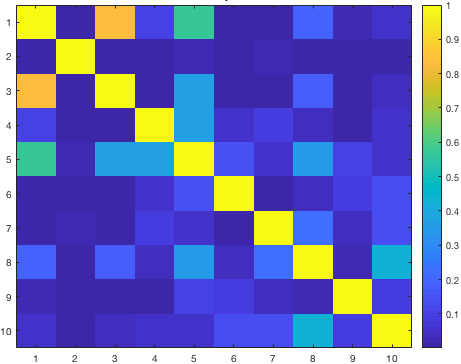} &
\includegraphics[width = 2.05in]{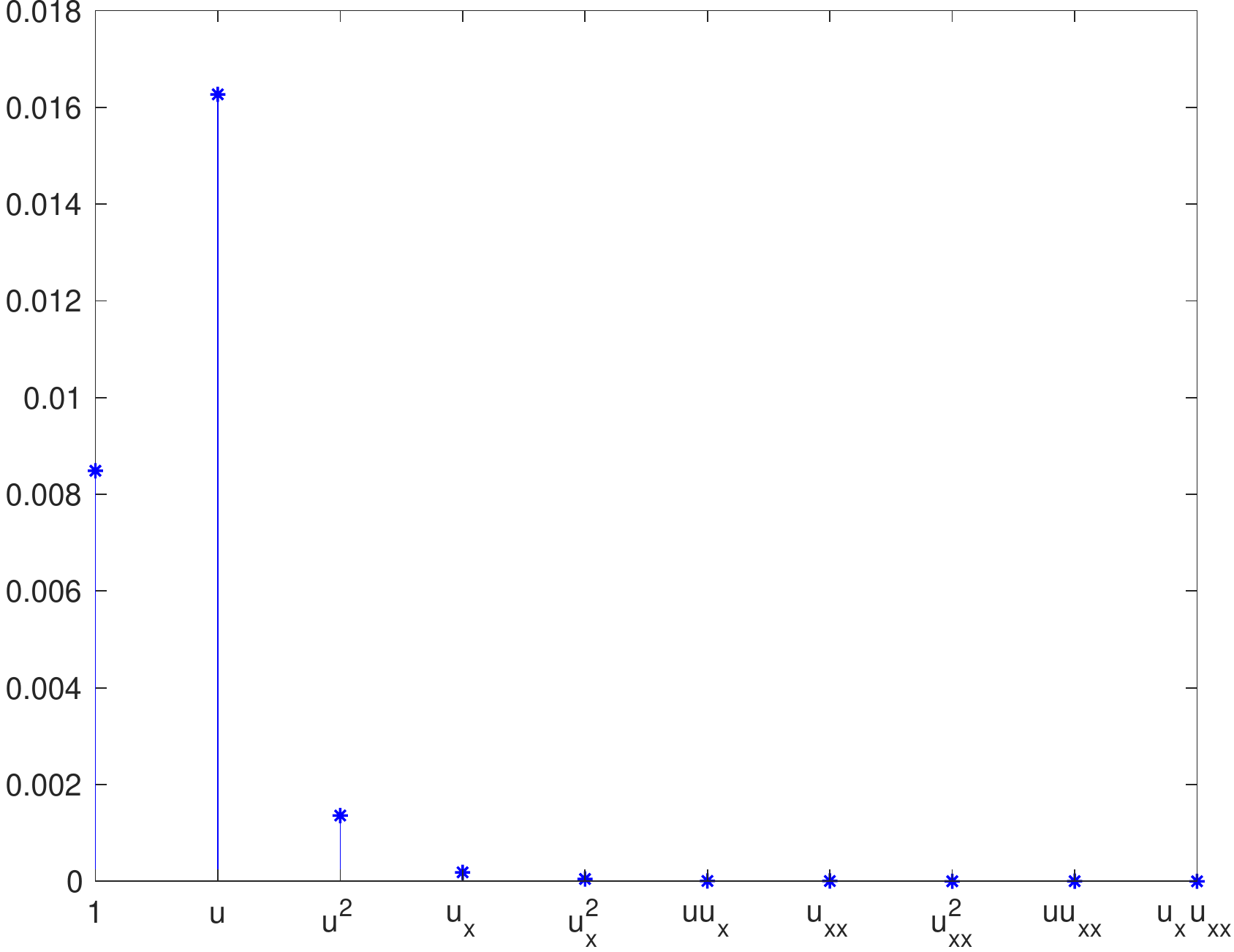}
\end{tabular}
\caption{Burger's equation in \eqref{E:burger} with $8\%$ Gaussian noise. (a) Given noisy data, (b) Coherence pattern of the feature matrix. (c) The normalized coefficient magnitudes from Lasso. This fails to identify the correct term $uu_x$. }\label{Fig-BurgerExactDemoNoise8}
\end{figure}

Our first experiment is on the Burger's equation in \eqref{E:burger} with $8\%$ Gaussian noise. Data are sampled from the analytic solution with $\Delta x = 1/56$ and $\Delta t = 0.004$ for $t \in [0,0.05]$, and then $8\%$ Gaussian noise is added. 
For comparison, we do not denoise the given data, but directly applied IDENT.  
Figure \ref{Fig-BurgerExactDemoNoise8} (a) shows the noisy given data, (b) shows the coherence pattern, and (c) shows the normalized coefficient magnitudes from Lasso.  The NSR for Lasso defined in \eqref{eqnsr} is 3.04.  Lasso fails to include the correct set of terms, thus TEE identifies the wrong equation $u_t=-0.59u^2$ as a solution. 

\begin{center}
\begin{tabular}{ | c | c | c ||  c |c|c| } 
     \hline
    Active terms & Coefficients & TEE  
    & Active terms & Coefficients & TEE   \\ \hline
    $1$ & $-0.27$ & $94.20$
    &   $[1 \ u  ]$ & $[-0.27 \ 1.17]$ & $99.15$ \\ 
    $u$ & $ 1.17$ & $99.36$
    &   $[1 \ u^2 ]$ & $[0.18\ -0.83]$ & $94.04$ \\ 
    \textcolor{red}{$u^2$} & \textcolor{red}{$ -0.59$} & \textcolor{red}{$94.03$}
    &   $[u \ u^2]$ & $[1.17\ -0.59]$ & $98.85$ \\ 
    $[1 \ u \ u^2]$ & $[0.19 \ 1.17 \ -0.84]$ & $98.81$
    &  & & \\
      \hline
  \end{tabular}
\end{center}

\begin{figure}[h!]
\centering
\begin{tabular}{ccc}
(a) Given data & 
(b) Coherence pattern & 
(c) Result from Lasso \\
\includegraphics[width = 2.05in]{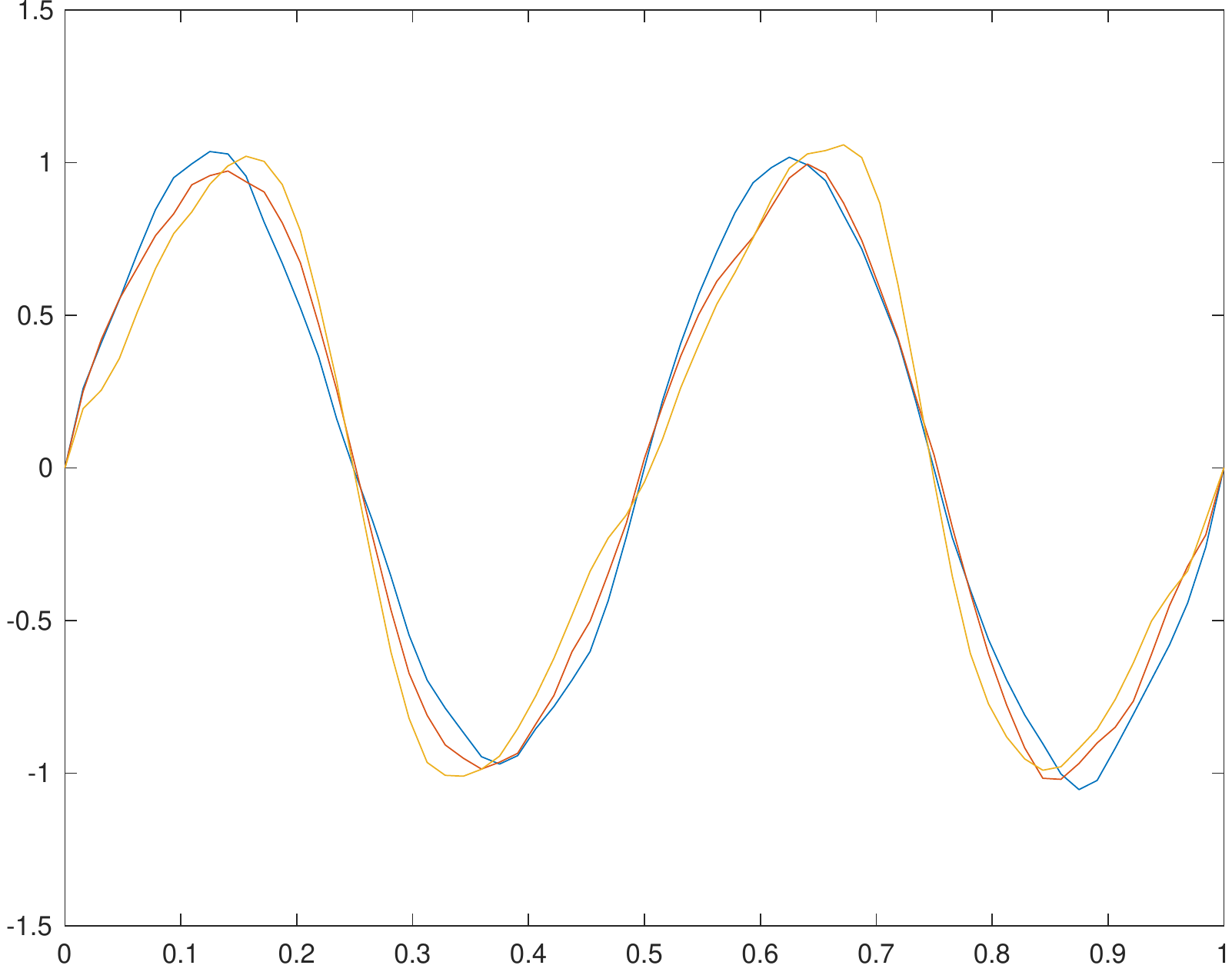} &
\includegraphics[width = 2.05in]{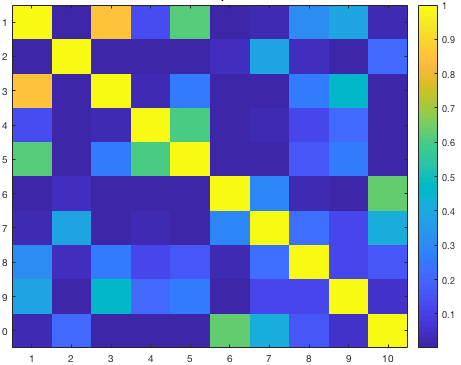} &
\includegraphics[width = 2.05in]{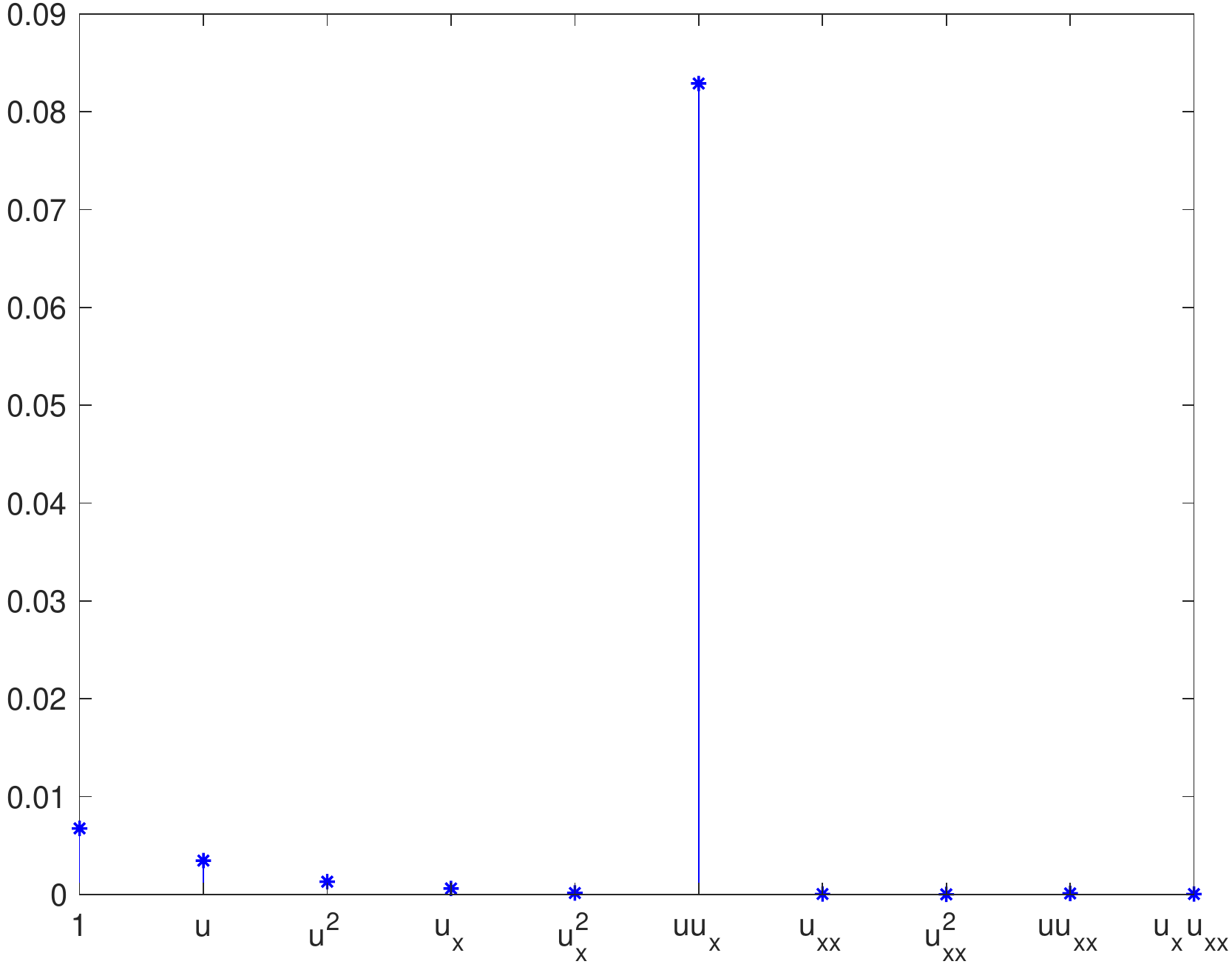}
\end{tabular}
\caption{Burger's equation in \eqref{E:burger} with $8\%$ Gaussian noise as in Figure \ref{Fig-BurgerExactDemoNoise8}. (a)  The data  after the LSMA denoising.  (b) Coherence pattern of $\hF$. (c) The normalized coefficient magnitudes from Lasso identifies $1,u$ and $uu_x$ which include the correct term $u u_x$.  }\label{Fig-BurgerExactDemoNoise8_LSMA}
\end{figure}
For the same data, Figure \ref{Fig-BurgerExactDemoNoise8_LSMA} and the table below show the  results when LSMA denoising is applied.  After denoising, and the given data is  noticeably smoother when Figure \ref{Fig-BurgerExactDemoNoise8_LSMA} (a) is compared with Figure \ref{Fig-BurgerExactDemoNoise8} (a).  
The Lasso result shows great improvement in Figure \ref{Fig-BurgerExactDemoNoise8_LSMA} (c).  With the correct terms included in the Step 2 of Lasso,  TEE determines the PDE with correct feature: $u_t = -0.92 u u_x$.
\begin{center}
\begin{tabular}{ | c | c | c |  c |c|c| } 
     \hline
    Active terms & Coefficients & TEE 
    & Active terms & Coefficients & TEE   \\ \hline
    $1$ & $-0.25$ & $87.10$
    &   $[1 \ u  ]$ & $[-0.25 \ 0.27]$ & $87.94$ \\ 
    $u$ & $ 0.27$ & $87.89$
    &   $[1 \ u u_x ]$ & $[-0.22\ -0.92]$ & $29.5412$ \\ 
    \textcolor{red}{$uu_x$} & \textcolor{red}{$ -0.92$} & \textcolor{red}{$29.5409$}
    &   $[u \ uu_x]$ & $[0.07\ -0.92]$ & $29.81$ \\ 
    $[1 \ u \ u u_x]$ & $[-0.22 \ 0.07 \ -0.92]$ & $29.80$
    &  & &\\
      \hline
  \end{tabular}
\end{center}

In the next set of experiments, we explore different levels of noise for denosing+IDENT.  In Figure \ref{Fig-BurgerExactVersusNoise}, we experiment on the Burger's equation \eqref{E:burger} with its analytic solution sampled in the same way as above, while the noise level increases from $0$ to $30\%$.   For each noise level, we (i) first generate data with $100$ sets of random noises, (ii) denoise by LS and LSMA  for a comparison respectively, then (iii) run IDENT. The parameter $\tau$ is chosen as $10\%$ of the largest coefficient magnitude.  
Figure \ref{Fig-BurgerExactVersusNoise} (a) represents how likely wrong results are found.  It is computed  by  the average ratio between the wrong coefficients and all computed coefficients: $\textstyle \sum_{j \in \widehat{\Lambda}\setminus \Lambda} |\widehat{\ba}_j|/\|\widehat{\ba}\|_{1}$, where $\Lambda$ and $\widehat{\Lambda}$ are the exact support and the identified support, respectively. Each bar plot represents the standard deviation of the results among 100 trials. 
The green curves denoised by LSMA show the most stable results even as the noise level increases. 
Figure \ref{Fig-BurgerExactVersusNoise} (b) shows the recovered coefficient of $uu_x$, where the true value is $-1$.   Notice that the LSMA+IDENT (green points) results are closer to $-1$, while others find wrong coefficients more often.  In general, denoising the given data with LSMA improves the result significantly.

\begin{figure}
\centering
\begin{tabular}{cc}
(a) Ratio of wrong coefficients versus noise & 
(b) Coefficient of $uu_x$, if $uu_x$ is identified \\ 
\includegraphics[width = 2.5in]{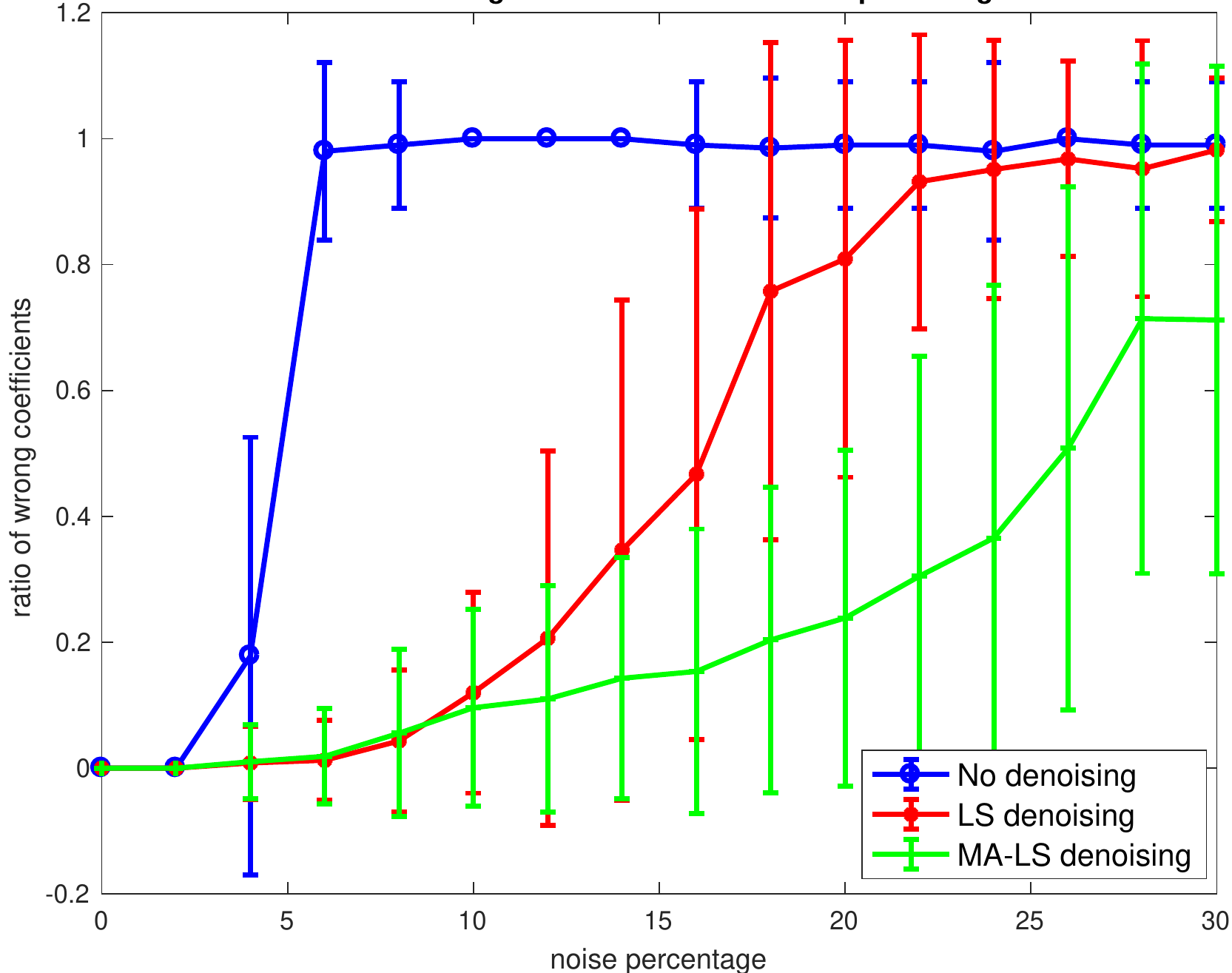} &
\includegraphics[width = 2.5in]{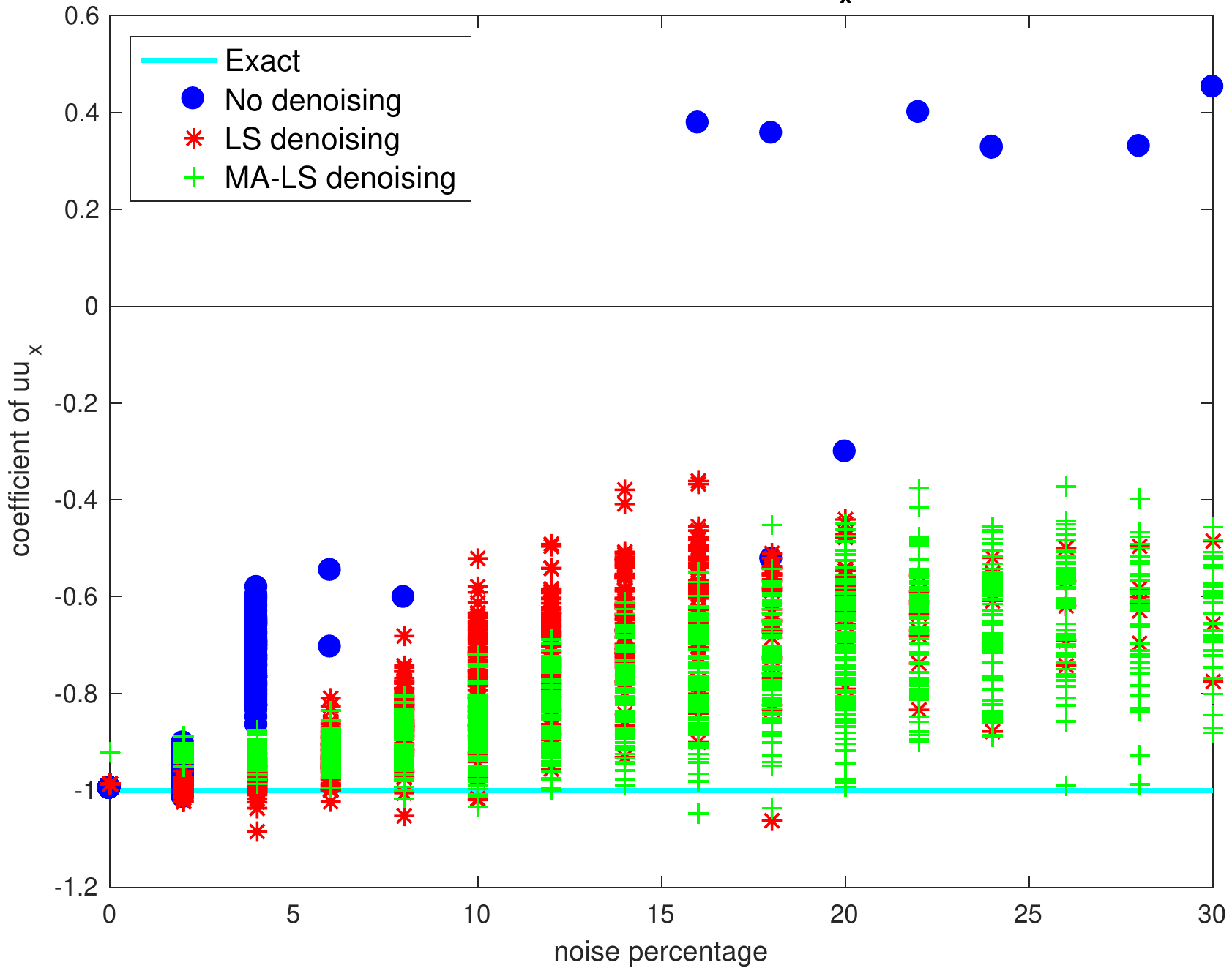} \\
\end{tabular}
\begin{tabular}{ | c   |c|c|c|c|c|c|c| c|} 
    \hline
  Added noise in  \% & 0 & 4  & 8  & 12 & 16 & 20 & 24 & 28
    \\ 
    \hline
    Added noise in new NSR \eqref{eqnsr} & 0.04 &   2.18&    3.09&    3.33&    3.40&    3.31&    3.31 &     3.23  \\
    \hline
     \end{tabular}
\caption{Burger's equation \eqref{E:burger} with increasing noise levels. (a) The average ratio between the identified wrong coefficients and all identified coefficients over $100$ trails. (b) The recovered coefficient of $u u_x$ by IDENT. Denoising the given data with LSMA significantly improves the result. The table shows the new NSR \eqref{eqnsr} corresponding to the noise level given in percentage.} 
\label{Fig-BurgerExactVersusNoise}
\end{figure}

Figure \ref{Fig-BurgerDiffVersusNoise} shows the Burger's equation with diffusion in \eqref{E:burger_diff} with varying noise levels.  The given data are sampled in the same way as in Figure \ref{Fig-BurgerDiffDemo}, the noise level increases from $0$ to $0.12\%$.   (a) shows the average ratio between the wrong coefficients and the total coefficients.   (b) and (c) show the recovered coefficients of $uu_x$ and $u_{xx}$, respectively.   Again using LSMA shows better performance. 
\begin{figure}
\centering
\begin{tabular}{ccc}
(a) Ratio of wrong coefficients & 
(b) Coefficient of $uu_x$ &
(c) Coefficient of $u_{xx}$ \\ 
\includegraphics[width = 2.05in]{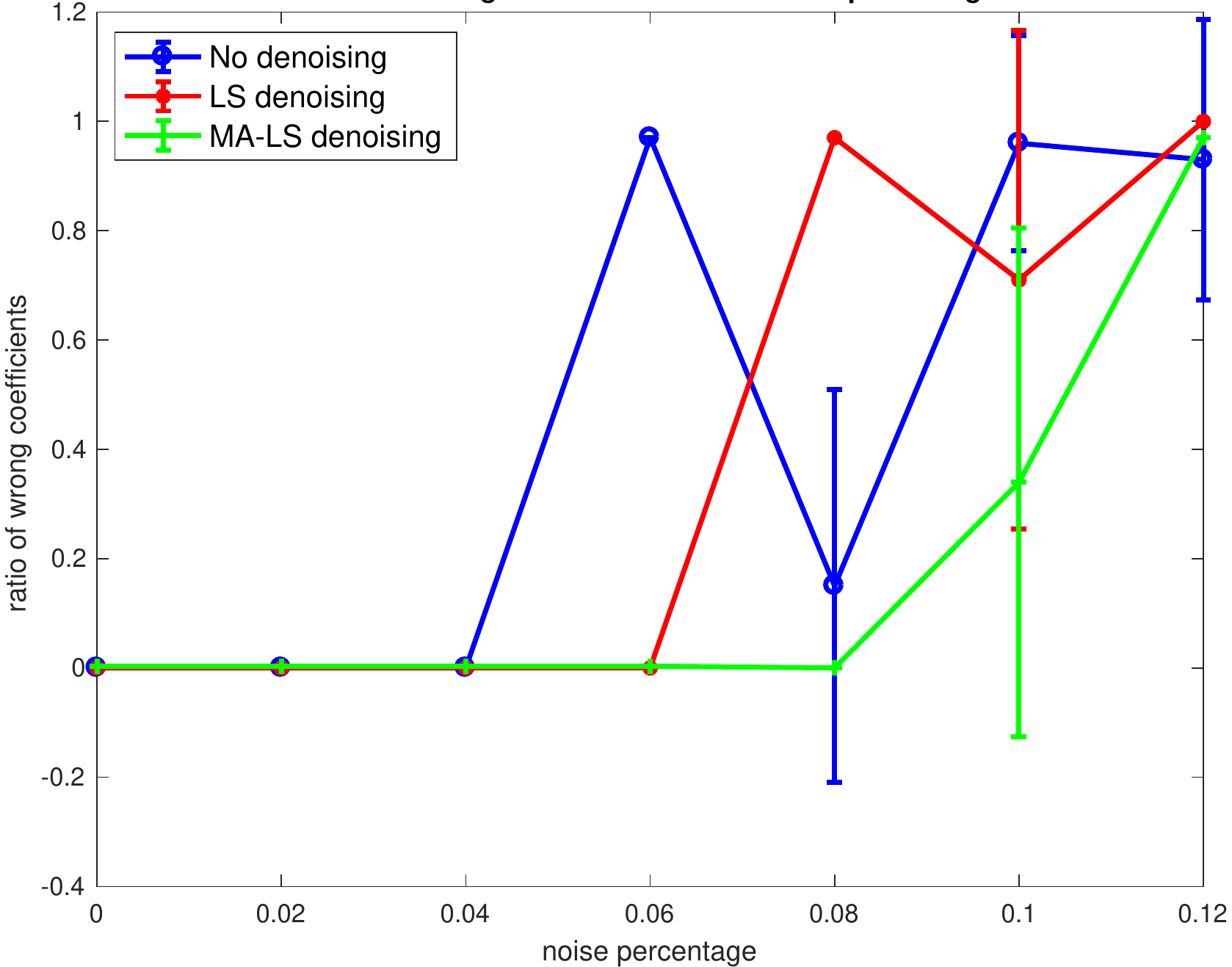} &
\includegraphics[width = 2.05in]{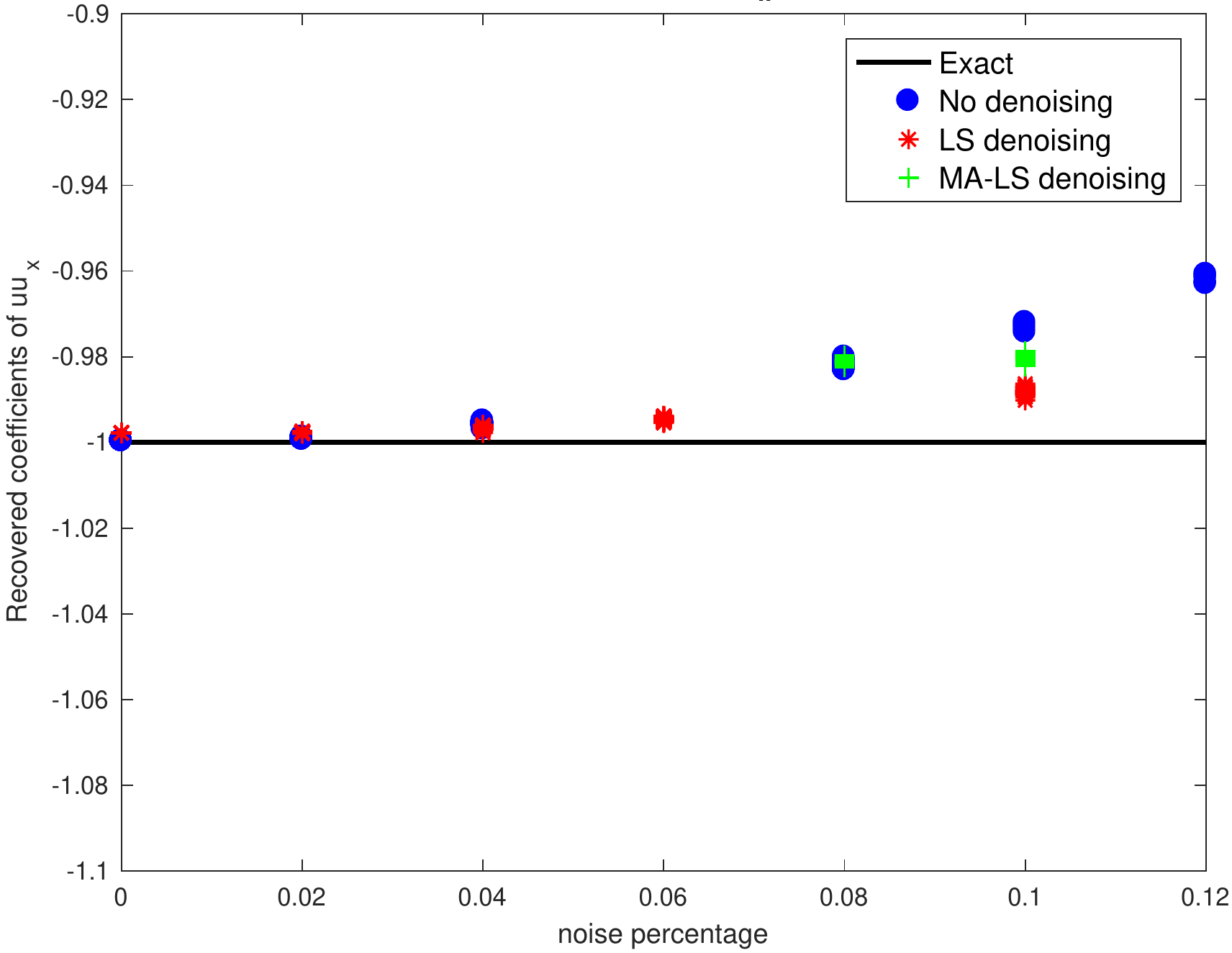} &
\includegraphics[width = 2.05in]{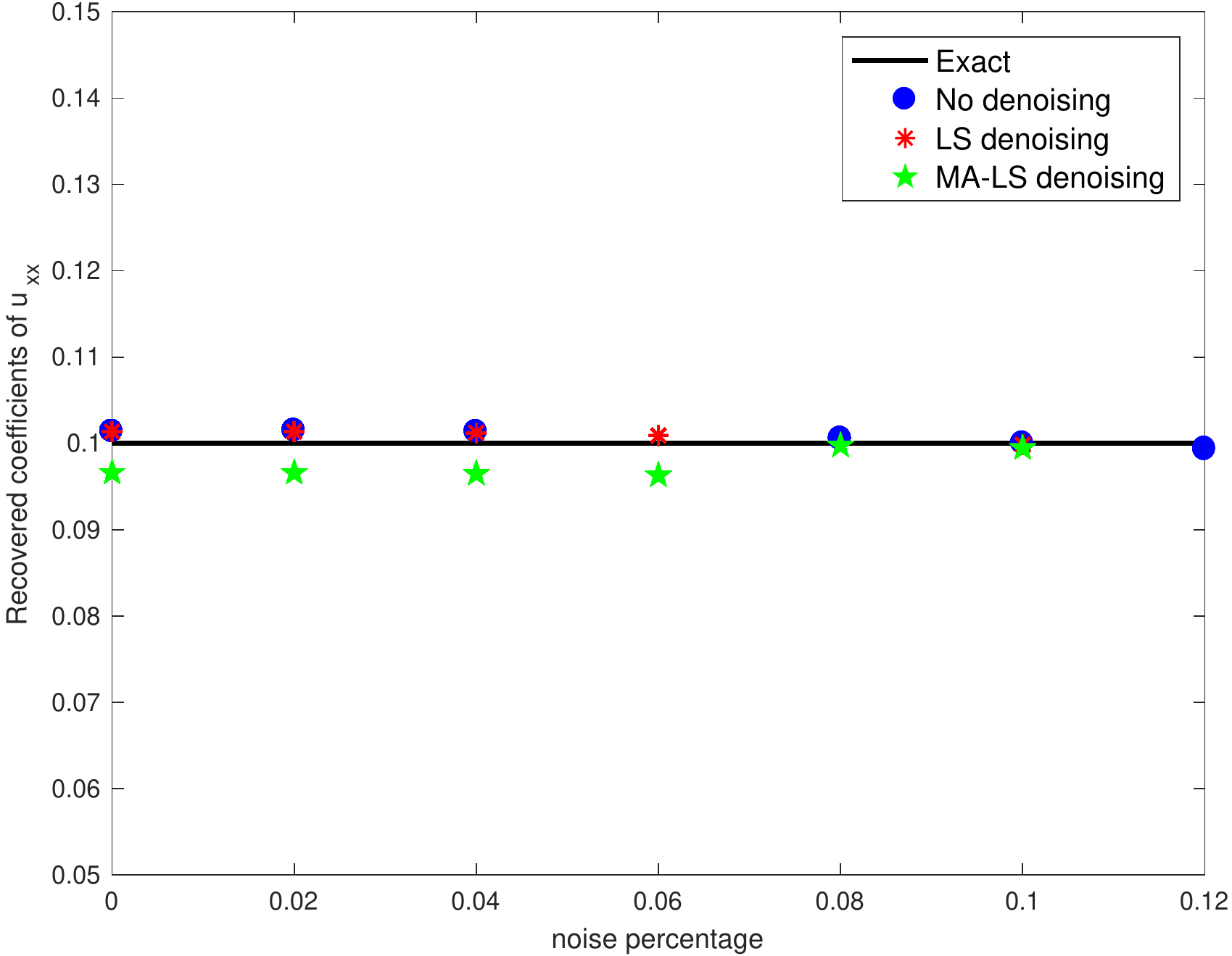} \\
\end{tabular}
\begin{tabular}{ | c | c | c |  c |c|c|c|c| } 
    \hline
  Added noise in  \%   & 0 &   0.02&    0.04 &    0.06 &    0.08&    0.10 & 0.12     \\ 
    \hline
   Added noise in new NSR \eqref{eqnsr} &     0.05&    2.02&    4.04&    6.06&    8.08&    10.10&    12.13   \\
    \hline
     \end{tabular}
\caption{Burger's equation with diffusion in \eqref{E:burger_diff} with varying noise levels. (a) The average ratio between the identified wrong coefficients and all identified coefficients over $100$ trails. (b) and (c) the computed coefficients of $u u_x$ and $u_{xx}$ respectively by IDENT.  While the noise level in percentage seems small, the new NSR represents the severeness of noise for PDE with high order derivatives. }
\label{Fig-BurgerDiffVersusNoise}
\end{figure}

For both Figures \ref{Fig-BurgerExactVersusNoise} and \ref{Fig-BurgerDiffVersusNoise}, we present the new NSR defined in \eqref{eqnsr}.  This clearly presents that noise affects different PDEs in different ways.  The Burger's equation \eqref{E:burger} only have first order derivatives, while the Burger's equation with diffusion in \eqref{E:burger_diff} has a second order derivative.  This seemingly small difference makes a big impact on the NSR and identification.  While in Figure \ref{Fig-BurgerExactVersusNoise}, the noise level is experimented up to $30 \%$, its corresponding new NSR varies only from 0 to less than 3.5.  In Figure \ref{Fig-BurgerDiffVersusNoise}, the noise level only varies from 0 to 0.12 in percentage, however, this corresponds to new NSR varying form 0 to above 12.  The level of the new NSR characterizes the difficulty of identification using IDENT (Step 2, Lasso), since having a higher-order term affects the Lasso  negatively, especially in the presents of noise.

\subsection{Downsampling effects and IDENT}\label{subsec:down} 
 
 In applications, data are often collected on a coarse grid to save the expenses of sensors. We explore the effect of downsampling in data collections in this section. 
 Consider a $r$th order PDE. Simulating its solution with a $q$th order method on a fine grid with time step $\delta t $ and spatial spacing $\delta x$ gives rise to the error $O(\delta t + \delta x^q)$.
 Suppose data are downsampled by a factor of $C_t$ in time and $C_x$ in space, such that data are sampled with spacing $\Delta t = C_t \delta t$ and $\Delta x = C_x \delta x$. Our error formula in \eqref{E:enoise} is crucially dependent on the downsampling factors $C_t$ and $C_x$. Each term is affected by downsampling differently.
 \begin{itemize}
   \setlength\itemsep{-0.02cm}
 \item{The term $\Delta t + \Delta x^{p+1-r}$ arises from the approximation of time and spatial derivatives. It increases as the downsampling factors $C_t$ and $C_x$ increase.} 
 \item{The term $\frac{\delta t + \delta x^q}{\Delta t} + \frac{\delta t + \delta x^q}{\Delta x^r}$ arises from the error in data generations. It decreases as the downsampling factors $C_t$ and $C_x$ increase. }
 \item{The term $\frac{\sigma}{\Delta t} + \frac{\sigma}{\Delta x^r}$ arises from the measurement noise. It decreases as the downsampling factors $C_t$ and $C_x$ increase.}
 \end{itemize}
Therefore,  downsampling may positively affect the identification depending on the balance among these three terms.  
 
 As a numerical example, we consider the Burger's equation in \eqref{E:burger} with different downsampling factors. The analytic solution is evaluated on the grid with spacing $\delta x = 1/1024$ and $\delta t = 0.001$ for $t \in [0,0.05]$. After evaluating the analytic solution, we generate $100$ sets of random noises then downsample the noisy data with spacing $\Delta x = C_x \delta x$ and $\Delta t = C_t \delta t$ where $C_x = C_t = 1,2,2^2,2^3,2^4$, and $2^5$  respectively. We run IDENT on the given downsampled noisy data, denoised by LS and LSMA respectively.
Figure \ref{Fig-BurgerExactVersusDownsample} displays the ratio of wrong coefficients by IDENT versus $\log_2 C_x$ in the presence of $5\%$ or $10\%$ Gaussian noise.  We observe that increasing downsampling rates can positively affect the result until the downsampling rates become too large.  LSMA also gives the best performance.
 
\begin{figure}
\centering
\begin{tabular}{cc}
(a) Ratio of wrong coefficients with $5\%$ noise & 
(a) Ratio of wrong coefficients with $10\%$ noise \\ 
\includegraphics[width = 3.3in]{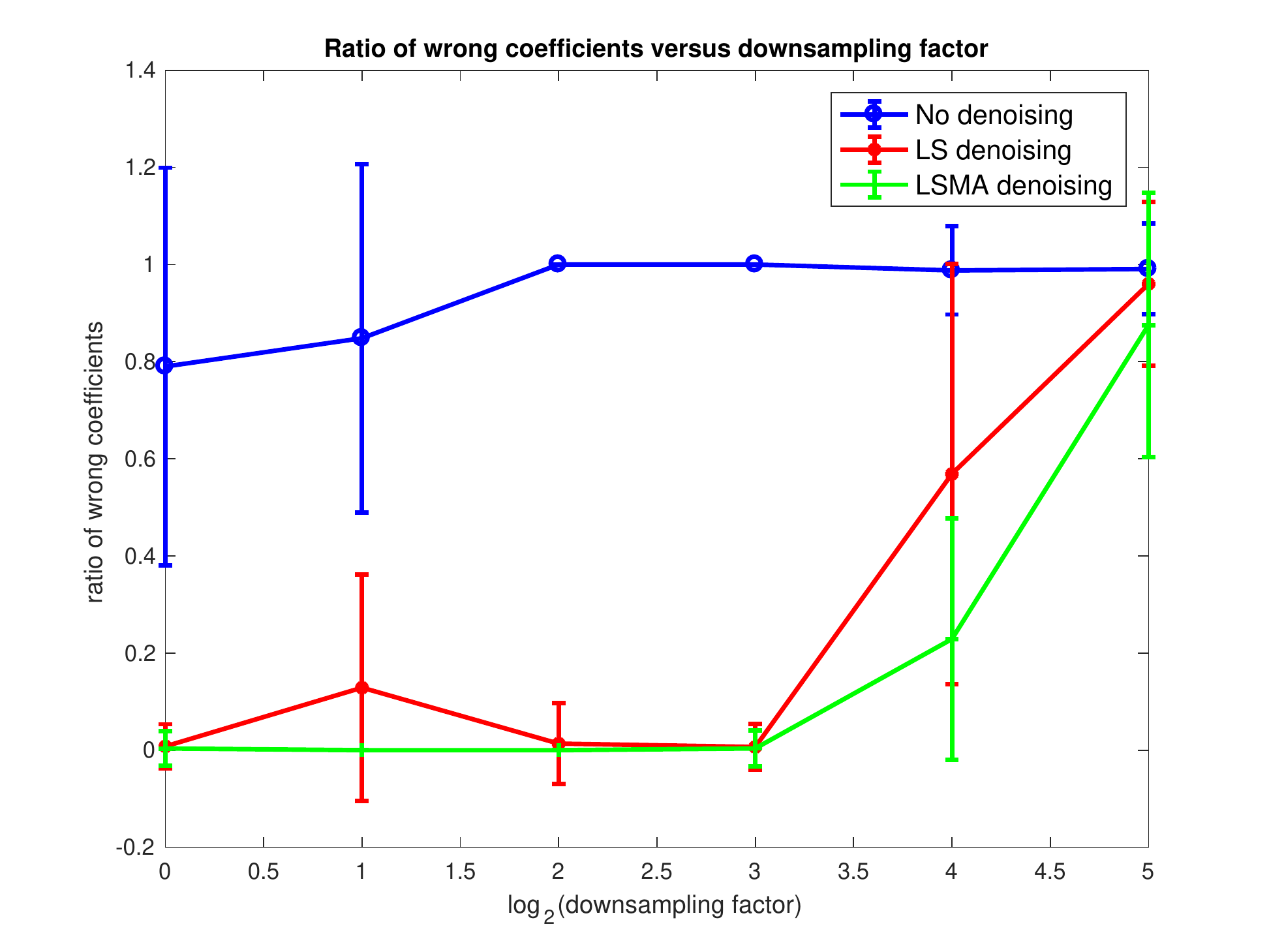} &
\includegraphics[width = 3.3in]{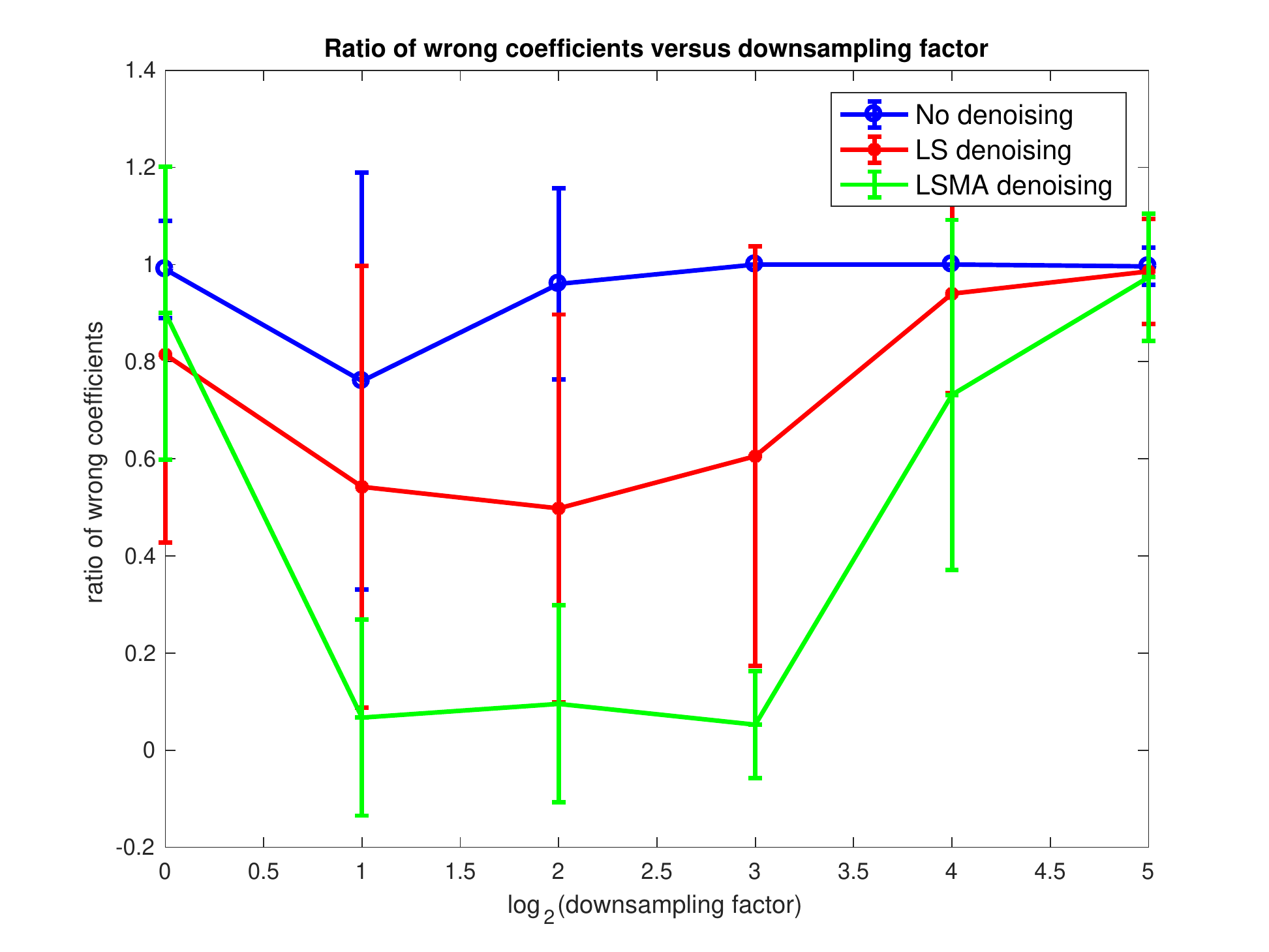} \\
\end{tabular}
\caption{Burger's equation in \eqref{E:burger} with various downsampling factors. (a) and (b) show the average ratio between the identified wrong coefficients and all identified coefficients in $100$ trails versus $\log_2(\text{downsampling factor})$ in the presence of $5\%$ (left) and $10\%$ (right) noise respectively.  Increasing the downsampling factors can positively affects the result until the downsampling factors become too large.}
\label{Fig-BurgerExactVersusDownsample}
\end{figure}

\section{Varying coefficients and Base Element Expansion}\label{sec:varying}

In this section, we consider PDEs with varying coefficients, e.g., $a_j(x)$ varying in space.  As illustrated in (\ref{E:L}), we can easily generalize the IDENT set-up to PDEs with varying coefficients, by expanding the coefficients in terms of finite element bases and solving group Lasso for $L>1$. 
Due to the increasing number of coefficients, the complexity of the problem increases as $L$ increases.  In order to design a stable  algorithm, we propose to let $L$ grow before TEE is applied. 

We refer to this extra procedure as Base Element Expansion (BEE).   From the given discrete data $ \{u_{i}^n | i=1, \dots, N_1 \text{ and } n= 1, \dots, N_2\}$, we first  compute numerical approximations of $u_t,u_x,u_{xx}$, etc, then apply BEE to gradually increase $L$ until the recovered coefficients become stable. 
For each fixed $L$, we form the feature matrix $\hF$ according to \eqref{E:general_F}, and solve group Lasso with the balancing parameter $\lambda$ to obtain $\widehat{\ba}_{\text{G-Lasso}}(\lambda)$.  We record the normalized block magnitudes from group Lasso, as $L$ increases: 
\[
\text{ BEE procedure}:= 
 \left\{\|\hF[j]\|_{L^1}\left\| \sum_{l=1}^L \frac{\widehat{\ba}_{\text{G-Lasso}}(\lambda)_{j,l}}{\|\hF[j,l]\|_\infty} \phi_l \right\|_{L^1}\right\}_{j=1,\ldots,N_3} \text{ versus  } L. 
\]
The main idea of BEE is based on the  convergence of the finite element approximation (\ref{E:approxError}) - as more basis functions are used, the more accurate the approximation is.  In the BEE procedure, the normalized block magnitudes reach a plateau as $L$ increases, i.e.,  candidate features can be selected by a thresholding according to \eqref{algthresholding} when $L$ is sufficiently large. With this added BEE procedure, IDENT continues to the Step 3 of TEE to refine the selection.

In the following, we present various numerical experiments for PDEs with varying coefficients using IDENT with BEE.   For the first set of experiments, in Figure  \ref{Fig-7VaryNoise0}, \ref{Fig-7VaryNoise0p2NoDenoising} and \ref{Fig-7VaryNoise0p2MALS},  we assume only one coefficient is known a priori to vary in $x$.  For the second set of experiments, in Figure \ref{Fig-47VaryDownsample}, we assume two coefficients are known a priori to vary in $x$, and the final experiment, in Figure \ref{Fig-AllVaryDownsample4} assumes all coefficients are free to vary without any a priori information.

\begin{figure}
\centering
\begin{tabular}{ccc}
(a) Given data & 
(b) BEE & 
(c) Group Lasso, when $L=20$ \\
\includegraphics[width = 2.05in]{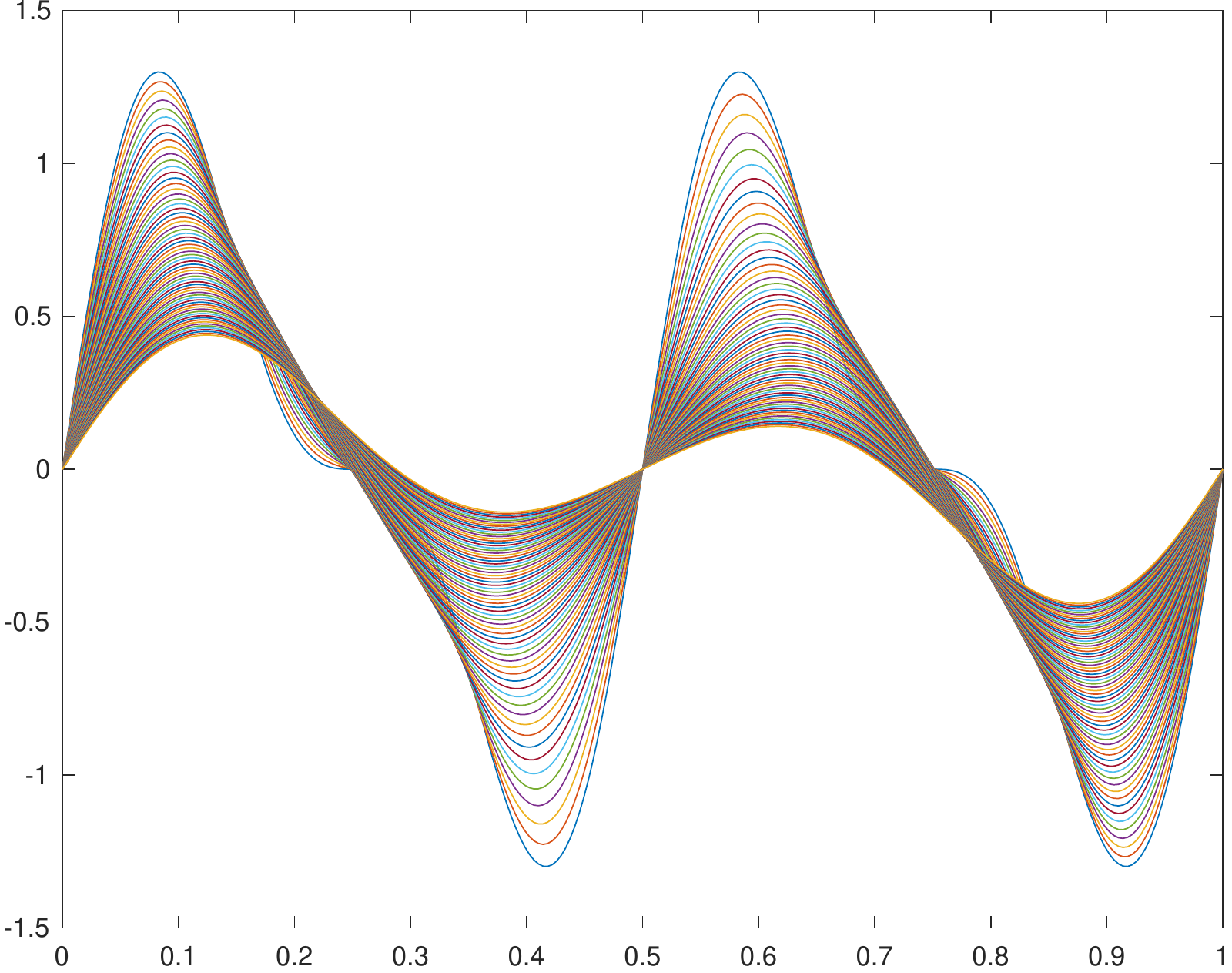} &
\includegraphics[width = 2.05in]{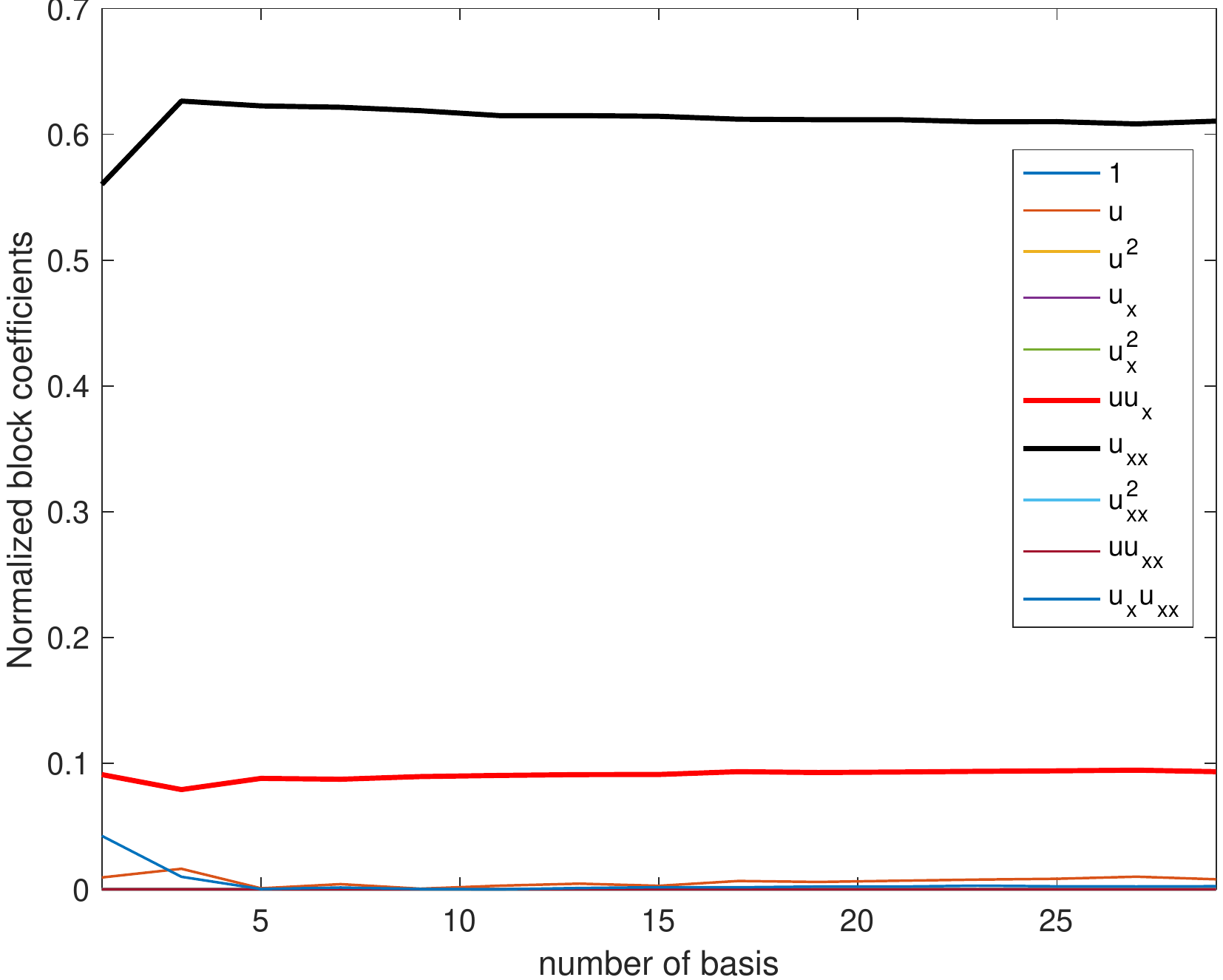} &
\includegraphics[width = 2.05in]{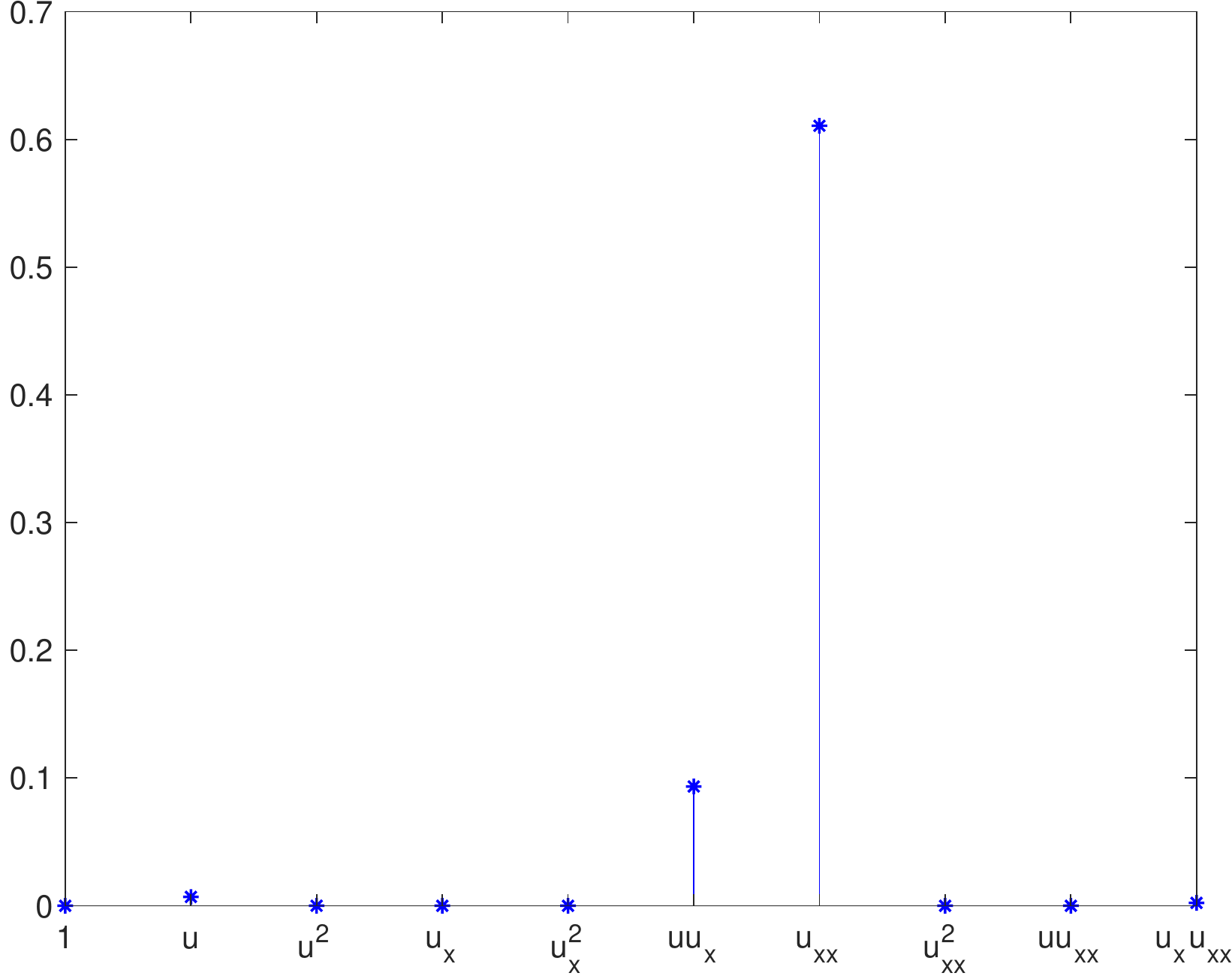}\\
(d) TEE in $\log_{10}$ scale vs. $L$& 
(e) $\widehat{c}(x)$ vs. $c(x)$  & 
(f)  $\|c(x)-\widehat{c}(x)\|_{L^1}$ vs. $L$ \\
\includegraphics[width = 2.05in]{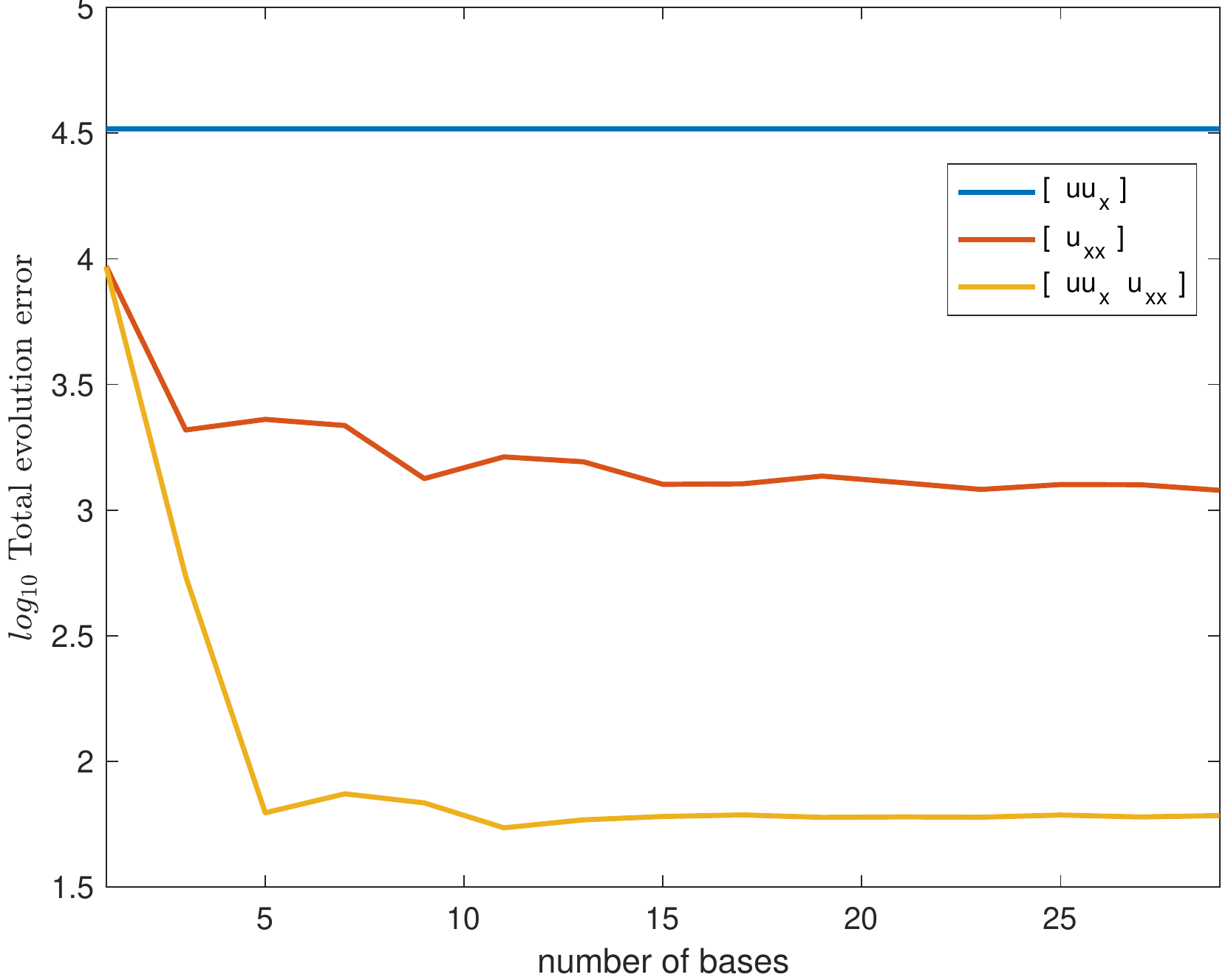} &
\includegraphics[width = 2.05in]{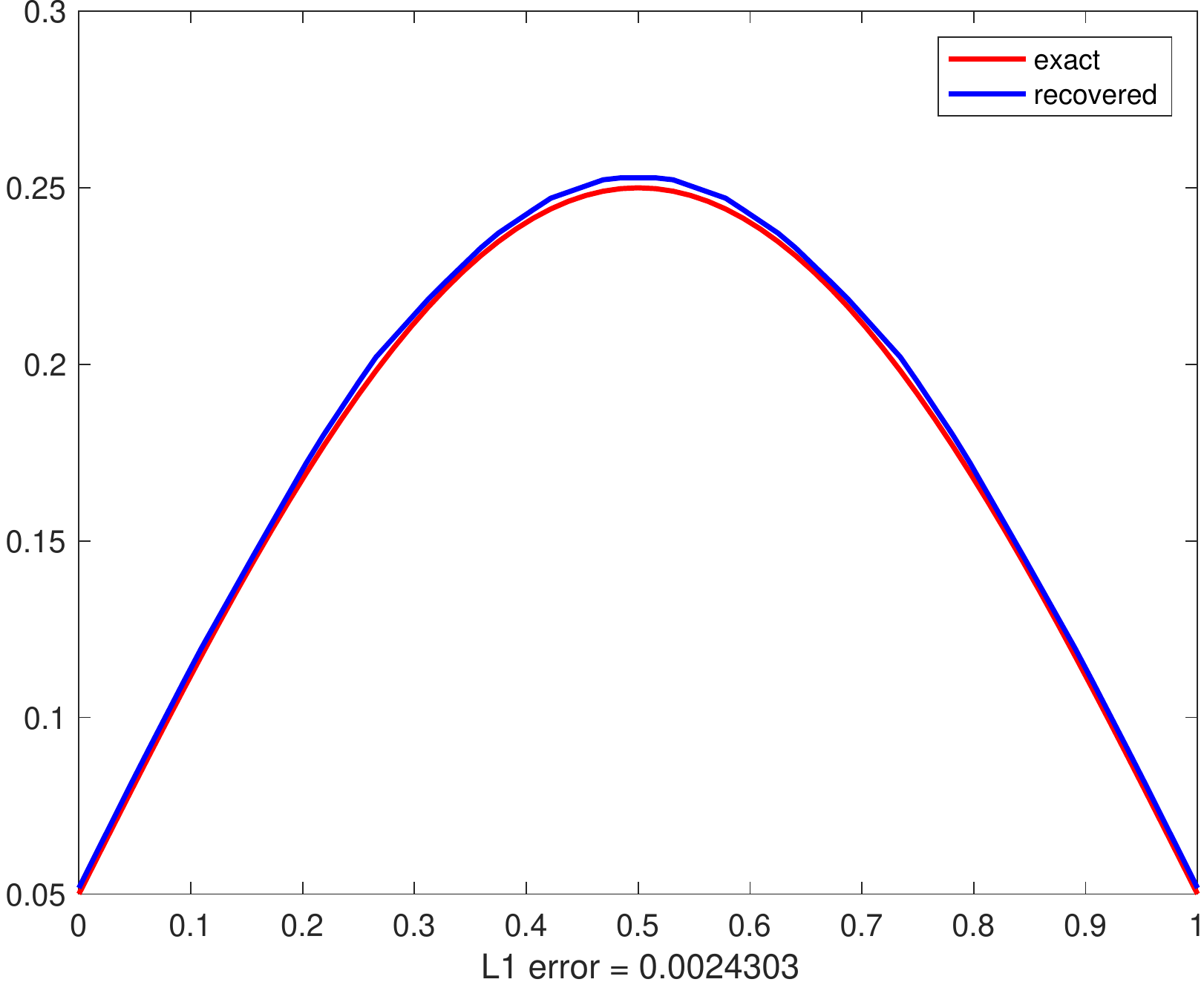} &
\includegraphics[width = 2.05in]{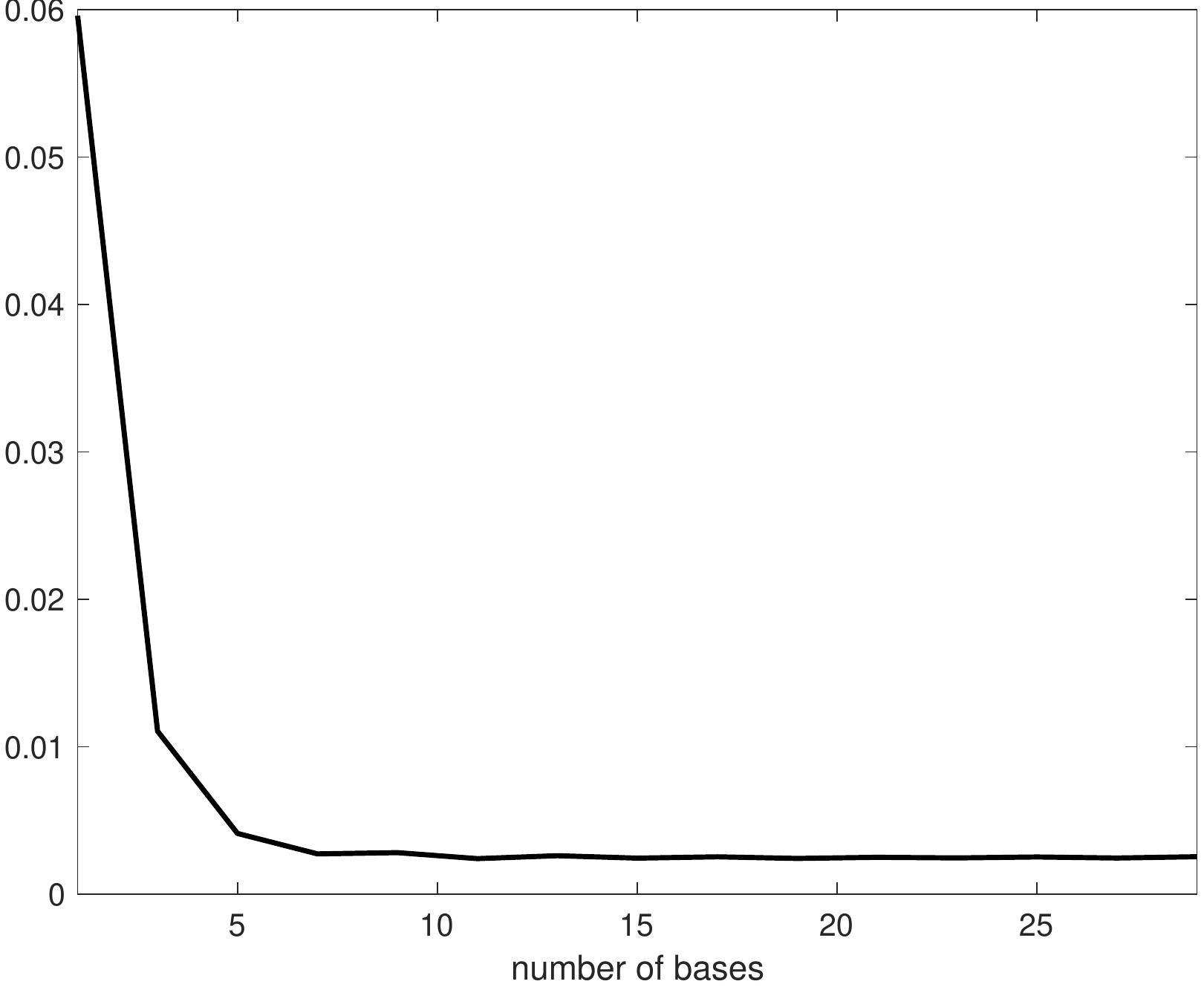}\\
\end{tabular}
\caption{Burger's equation with a varying diffusion coefficient \eqref{E:BurgerDiffVary} where data are downsampled by a factor 4.  (a) The given data. (b) BEE as $L$ increases from $1$ to $30$.  (c)  An example of the magnitudes of coefficients from Group Lasso when $L=20$.  (d) TEE versus $L$, for all subsets of coefficients of $\{uu_x \ u_{xx} \}$.  (e) Recovered diffusion coefficient $\widehat{c}(x) = \sum_{l=1}^L \widehat{a}_{7,l}\phi_l(x)$ when $L=20$ (blue), compared with the true diffusion coefficient $c(x)$ (red).   (f) The error $\|c(x)-\widehat{c}(x)\|_{L^1}$ as $L$ increases from $1$ to $30$.}
\label{Fig-7VaryNoise0}
\end{figure}

The first experiment is on the Burger's equation with a varying diffusion coefficient:
\begin{align}
& u_t + \left( \frac{u^2}{2}\right)_x =c(x)u_{xx}, \text{ where } c(x) = 0.05+0.2 \sin \pi x
\label{E:BurgerDiffVary}
\\
&\ x \in [0,1], \ u(x,0)= \sin(4\pi x)+0.5\sin(8 \pi x) \text{ and } u(0,t) = u(1,t) = 0. \nonumber
\end{align}
The given data, shown in Figure \ref{Fig-7VaryNoise0} (a), is numerically simulated by a first-order method with spacing $\delta x = 1/256$ and $\delta t = (\delta x)^2/2$ for $t \in [0,0.05]$. Data are downsampled by a factor of $4$ in time and space.  There is no measurement noise. Our objective is to identify the correct features $u u_x$ and $u_{xx}$ and recover the coefficients $-1$ and the varying coefficient $c(x)$.   After we expand the diffusion coefficient with $L$ finite element bases, the vectors to be identified can be written as $\ba = [a_1  \ldots a_6 \ a_{7,1} \ \ldots \ a_{7,L} \ a_8 \ldots a_{10}]^T$ where $c(x) \approx \sum_{\ell=1}^L a_{7,l}\phi_l(x)$.

\begin{figure}[t!]
\centering
\begin{tabular}{ccc}
(a) BEE & 
(b) TEE vs. $L$ &
(c) $\widehat{c}(x)$, when $L=20$ \\ 
\includegraphics[width = 2.05in]{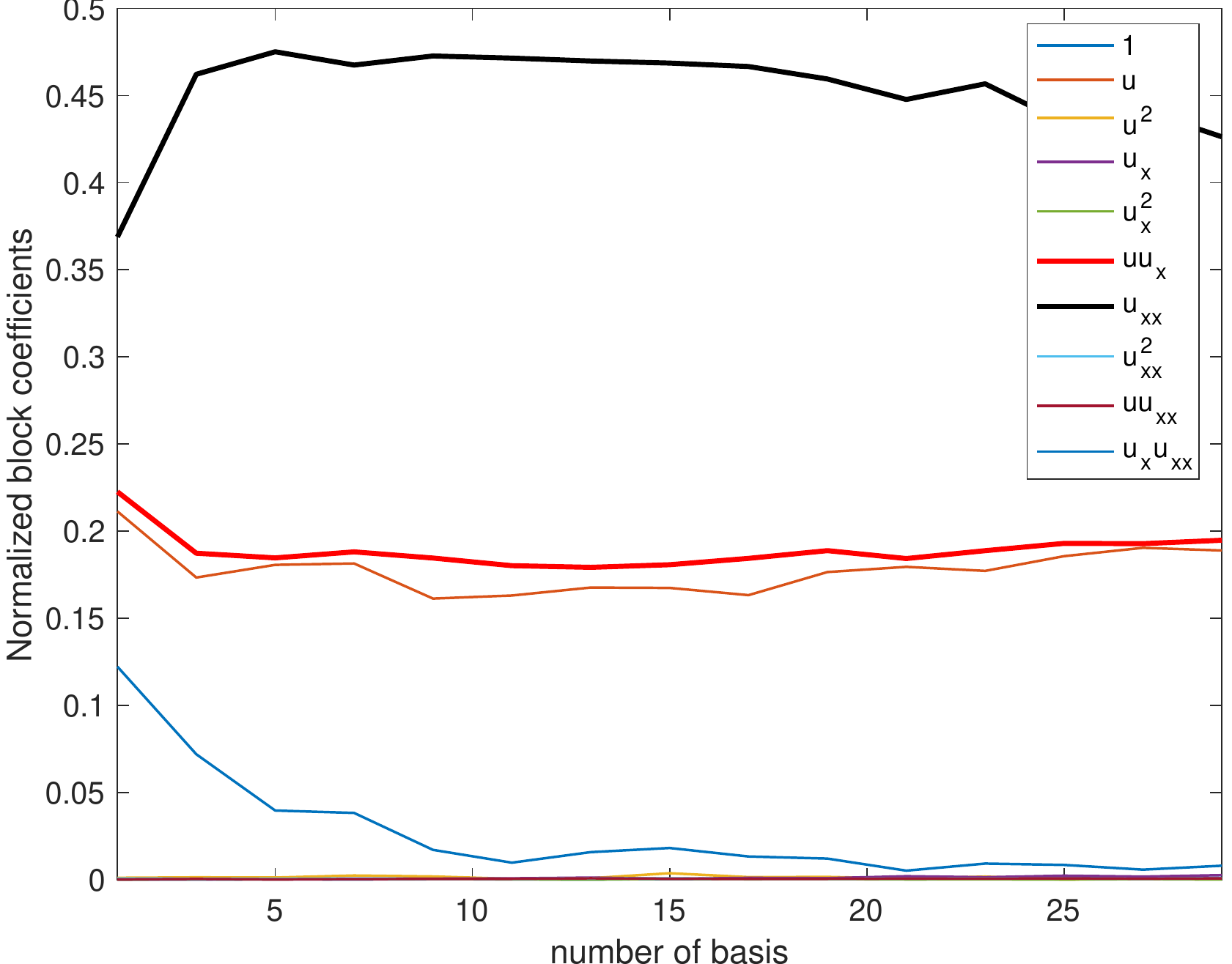} &
\includegraphics[width = 2.05in]{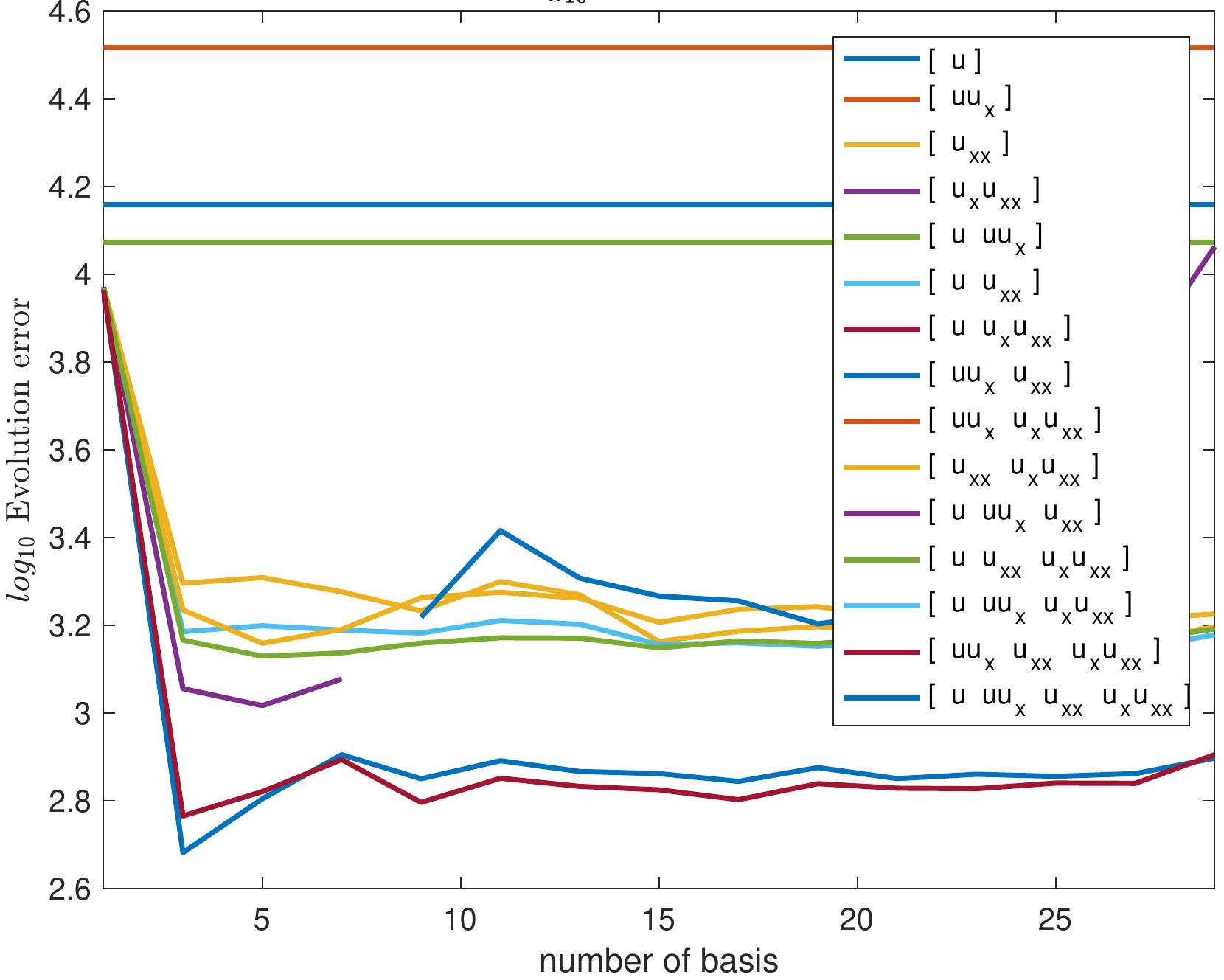} &
\includegraphics[width = 2.05in]{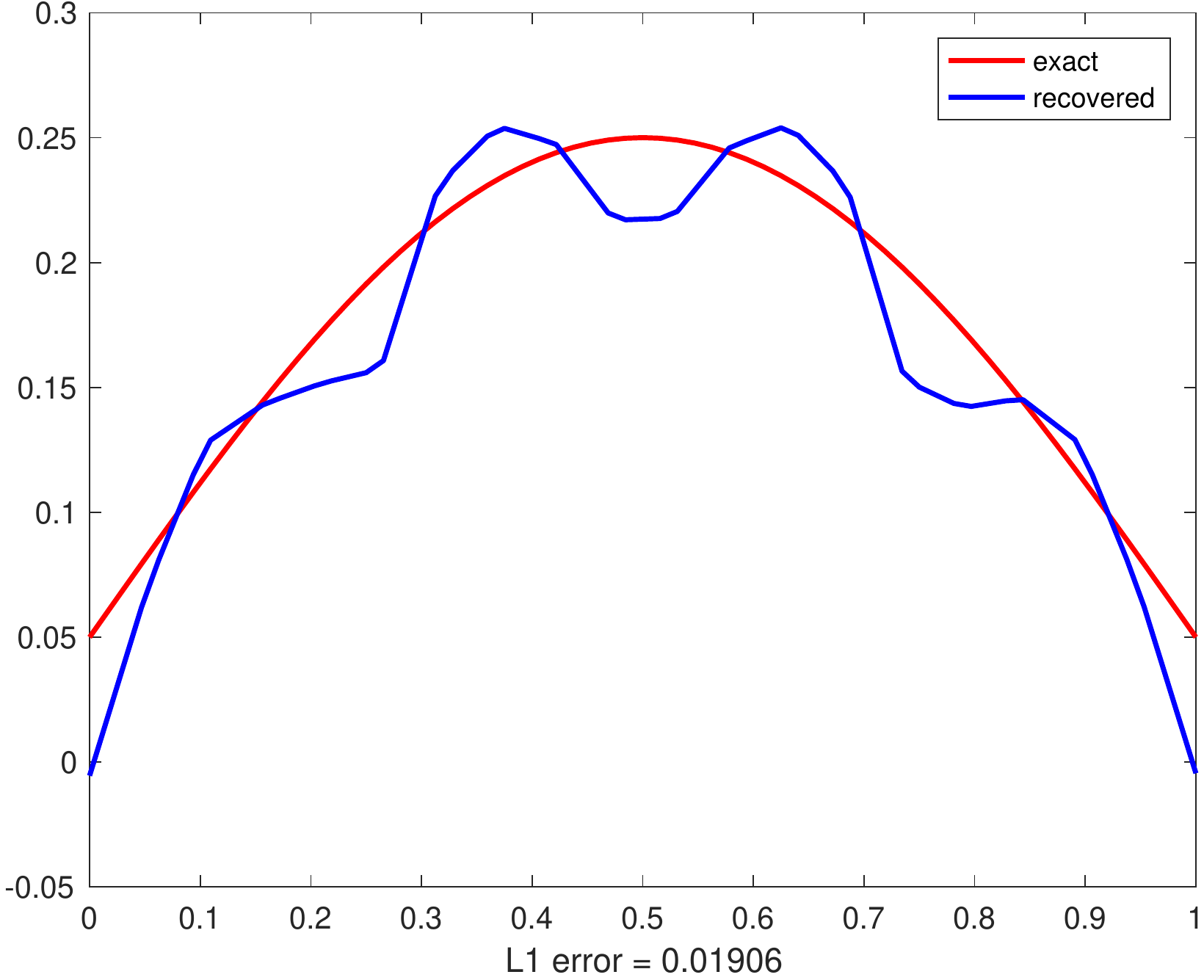} \\
\end{tabular}
\caption{Equation\eqref{E:BurgerDiffVary}, where data are downsampled by a factor of $4$ and $0.2\%$ measurement noise is added.  No denoising is applied. (a) BEE as $L$ increases from $1$ to $30$. (b) TEE versus $L$, for all subsets of four terms selected  in (a).  The correct support $[uu_x \ u_{xx}]$ is identified with the lowest TEE when $L \ge 7$.  (c) Recovered diffusion coefficient $\widehat{c}(x)$ when $L=20$ (blue), compared with the true diffusion coefficient $c(x)$ (red). }
\label{Fig-7VaryNoise0p2NoDenoising}
\end{figure}

\begin{figure}[h!]
\centering
\begin{tabular}{ccc}
(a) BEE & 
(b) TEE vs. $L$ &
(c) $\widehat{c}(x)$, when $L=20$ \\ 
\includegraphics[width = 2.05in]{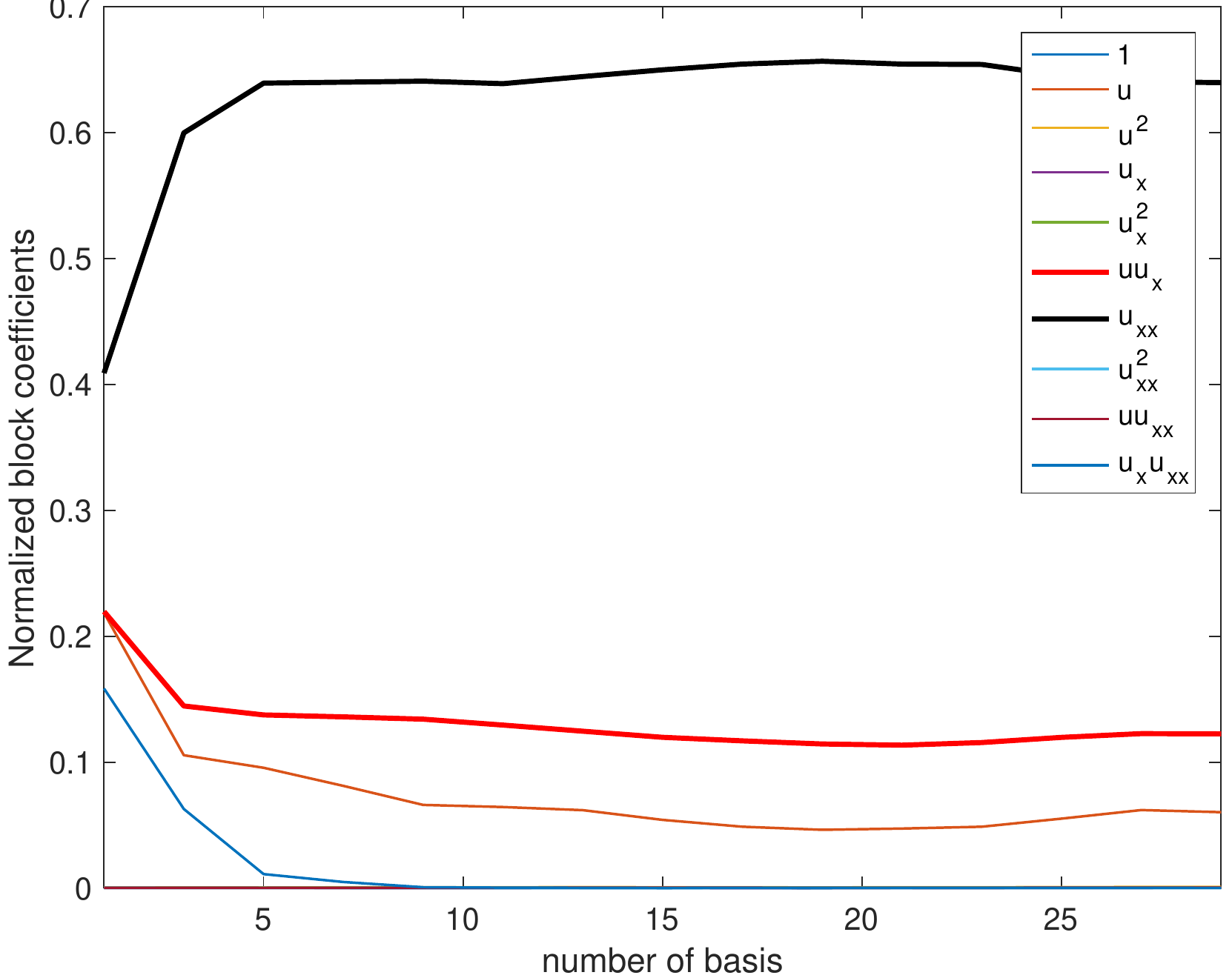} &
\includegraphics[width = 2.05in]{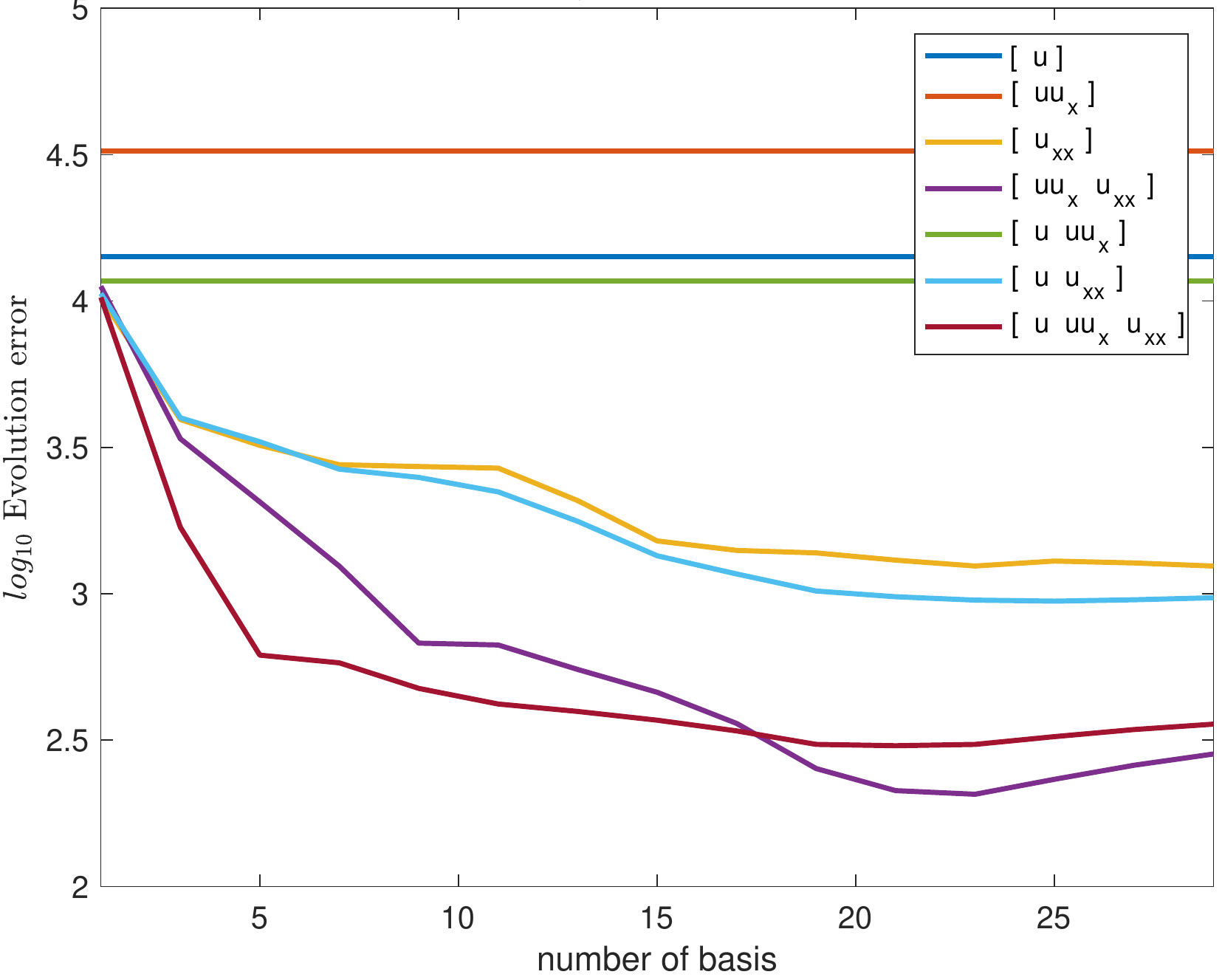} &
\includegraphics[width = 2.05in]{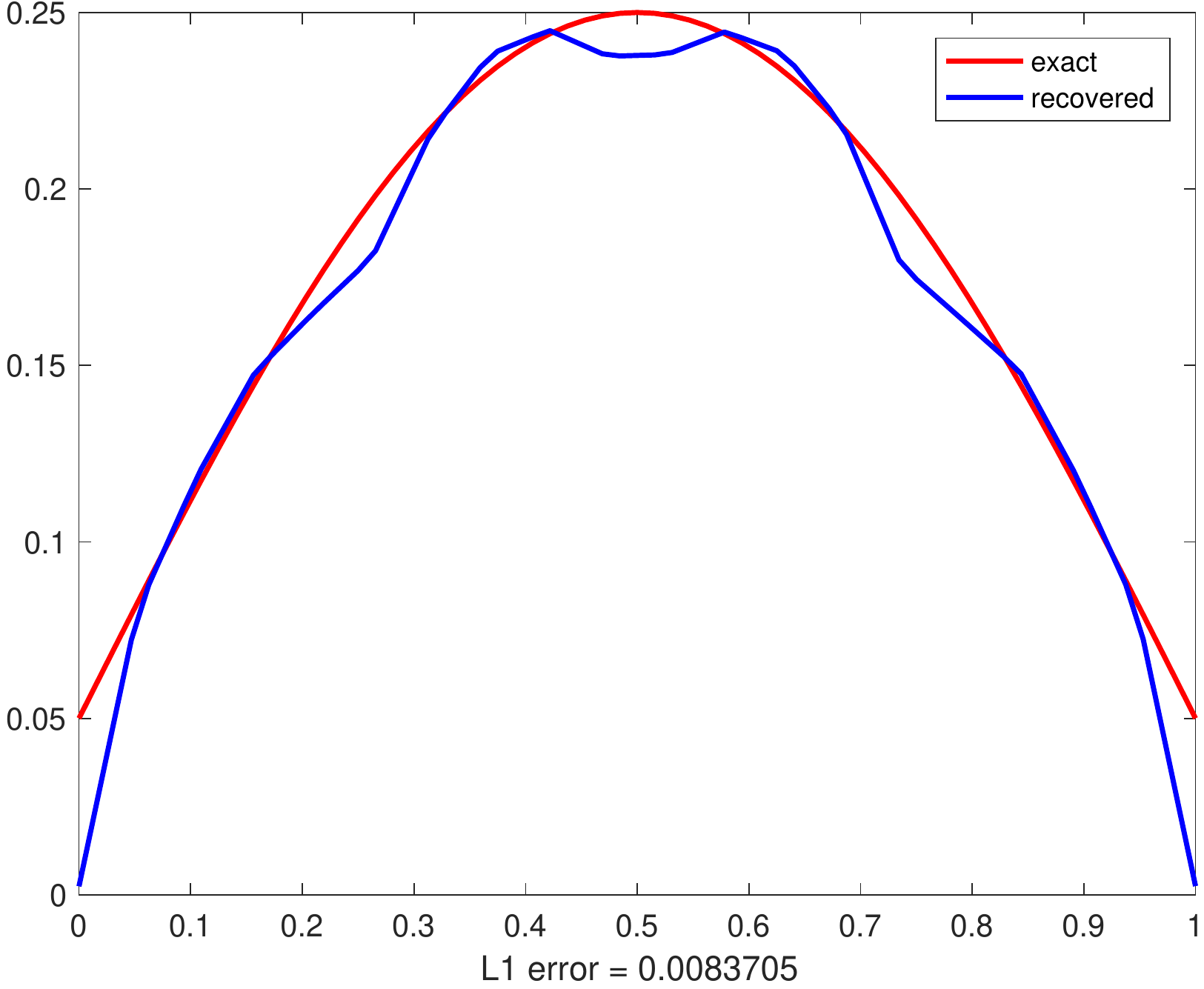} \\
\end{tabular}
\caption{The same experiment as Figure  \ref{Fig-7VaryNoise0p2NoDenoising}, but IDENT+BEE is applied with LSMA denoising.  (a) BEE as $L$ increases from $1$ to $30$. (b) TEE versus $L$, for all subset of coefficients identified in (a).  The correct support $[uu_x \ u_{xx}]$ is identified since it gives rise to the lowest TEE when $L \ge 19$.  (c) The recovered diffusion coefficient when $L=20$, compared with the true diffusion coefficient (red), which shows a clear improvement compared to Figure \ref{Fig-7VaryNoise0p2NoDenoising} (c).  }
\label{Fig-7VaryNoise0p2MALS}
\end{figure}

Figure \ref{Fig-7VaryNoise0} (b) presents BEE as $L$ increases from  $1$ to $30$.  This graph clearly shows  that BEE stabilizes when $L \ge 5$.   (c) is an example of the group Lasso result, the normalized block magnitudes, when $L = 20$.  The magnitudes of $uu_x$ and $u_{xx}$ are significantly larger than the others, so they are picked for TEE in Step 3.   
Figure \ref{Fig-7VaryNoise0} (d) presents TEE for all different $L$: TEE of $uu_x$, $u_{xx}$ and $[uu_x \ u_{xx}]$ in $\log_{10}$ scale, respectively.  The correct set, $[uu_x \ u_{xx}]$, is identified as the recovered feature set with the smallest TEE.  
The coefficient $[\widehat{a}_6 \ \widehat{a}_{7,1} \ldots \widehat{a}_{7,L}]$ is computed by least squares, and (e) displays the recovered  diffusion coefficient $\widehat{c}(x) = \sum_{l=1}^L \widehat{a}_{7,l}\phi_l(x)$ when $L=20$, compared with  the true equation $c(x) = 0.05+0.2 \sin \pi x$ given in (\ref{E:BurgerDiffVary}).
Figure \ref{Fig-7VaryNoise0} (f) shows the error $\|c(x)-\widehat{c}(x)\|_{L^1}$ as $L$ increases from $1$ to $30$. The error decreases as $L$ increases, yet does not converge to $0$ due to the errors arising from data generations and finite-difference approximations of $u_t,u_x$ and $u_{xx}$.

For the same equation,  $0.2\%$ noise is added to the next experiments presented in Figure \ref{Fig-7VaryNoise0p2NoDenoising} and \ref{Fig-7VaryNoise0p2MALS}.
Figure \ref{Fig-7VaryNoise0p2NoDenoising} presents the result without any denoising.  (a) shows BEE as $L$ increases from $1$ to $30$, where the magnitudes of  $u,uu_x,u_{xx},u_xu_{xx}$ are not negligible for $L \ge 20$.  These terms are picked for TEE, and Figure \ref{Fig-7VaryNoise0p2NoDenoising} (b) shows TEE versus $L$. The correct support $[uu_x \ u_{xx}]$ is identified with the lowest TEE when $L \ge 7$.    The computed diffusion coefficient $\widehat{c}(x)$ is compared  to the true one  in (c), which has the error $\|c(x)-\widehat{c}(x)\|_{L^1} \approx 0.019$.  Even for the data with noise, IDENT+BEE without any denoising gives a good identification of the general form of the PDE.  However, the varying coefficient approximation can be improved if LSMA denoising applied to the data as discussed in Section \ref{sec:noise}.

Figure \ref{Fig-7VaryNoise0p2MALS} presents the same experiment with  LSMA denoising.  In (a), BEE  picks $u,uu_x,u_{xx}$ for TEE. Notice that the coefficient of $u_x u_{xx}$ almost vanishes after denoising compared to Figure \ref{Fig-7VaryNoise0p2NoDenoising}.  (b) shows TEE versus $L$, where the correct support $[uu_x \ u_{xx}]$ gives the lowest TEE, when $L \ge 19$. The recovered diffusion coefficient when $L=20$ is shown in (c), which yields the error $\|c(x)-\widehat{c}(x)\|_{L^1} \approx 0.008$. 
In comparison with the results in Figure \ref{Fig-7VaryNoise0p2NoDenoising} without denoising, LSMA reduces the error of the recovered diffusion coefficient from $0.019$ to $0.008$. 

\begin{figure}[h]
\centering
\begin{tabular}{ccc}
(a) Numerical solution &
(b) $\widehat{b}(x)$, $\widehat{c}(x)$ vs. $b(x)$, $c(x)$  &
(b) $\widehat{b}(x)$, $\widehat{c}(x)$ vs. $b(x)$, $c(x)$  \\ 
\includegraphics[width = 2.05in]{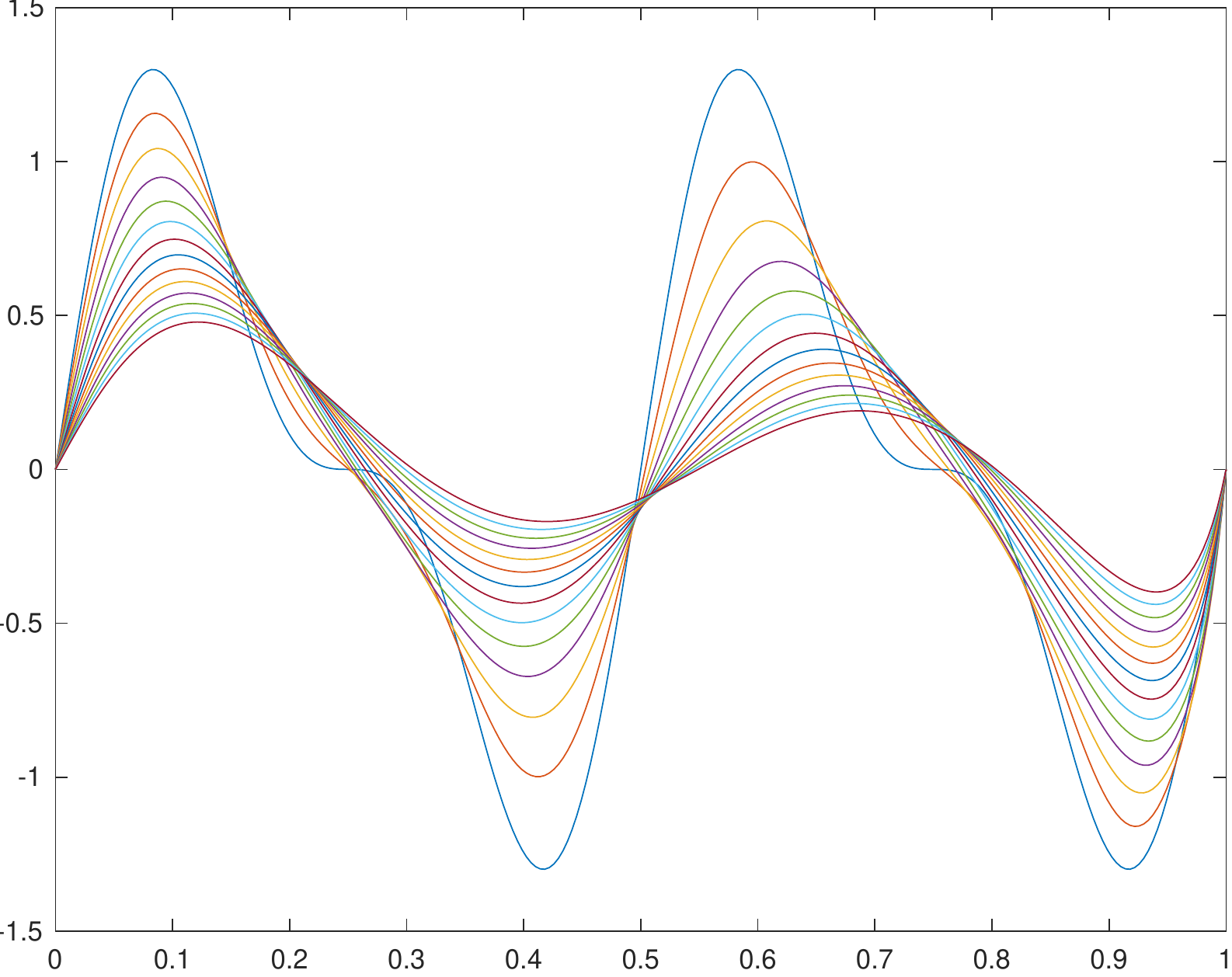} &
\includegraphics[width = 2.05in]{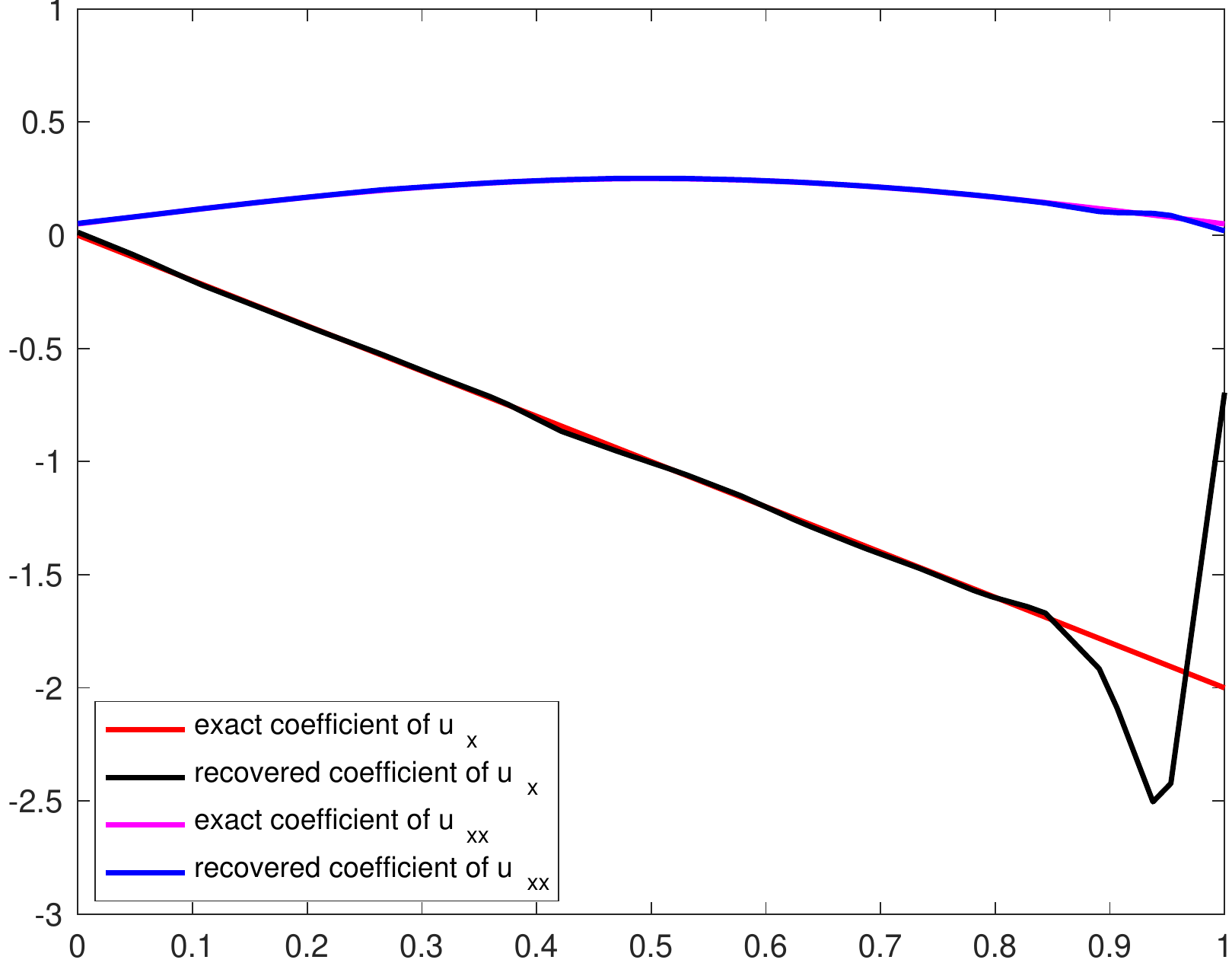} & 
\includegraphics[width = 2.05in]{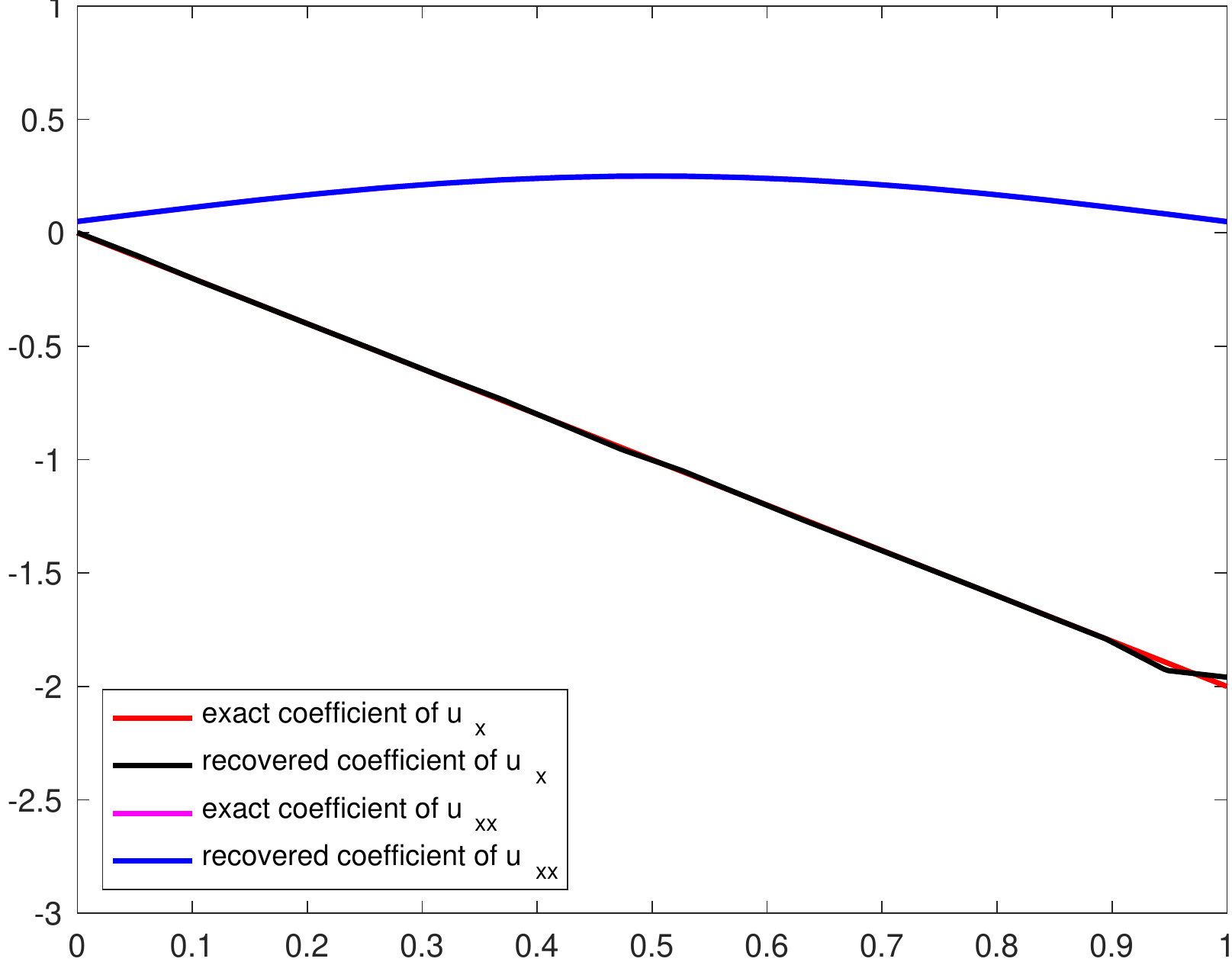} \\
\end{tabular}
\caption{Equation \eqref{E:47Vary} with varying convection and diffusion coefficients. (a)  The numerical solution of \eqref{E:47Vary}.  
(b) With data downsampled by a factor of $4$ in time and space, the recovered coefficient  $\widehat{b}(x)$  of $u_x$ is not accurate near  $x=1$. The downsampling rate is too high near  $x=1$ so that details of the solution are lost.  (c) The same experiment without any downsampling, and the recovered coefficients $\widehat{b}(x)$ and $\widehat{c}(x)$ are more accurate than (b). }
\label{Fig-47VaryDownsample}
\end{figure}

In Figure \ref{Fig-47VaryDownsample}, we experiment on the following PDF with two varying coefficients:
\begin{align}
& u_t  =b(x)u_x + c(x)u_{xx}, \text{ where } b(x)=-2x \text{ and } c(x) = 0.05+0.2 \sin \pi x, \label{E:47Vary}
\\ 
& x \in [0,1], \ u(x,0)= \sin(4\pi x)+0.5\sin(8 \pi x) \text{ and } u(0,t) = u(1,t) = 0. \nonumber
\end{align}
The given data are simulated by a first-order method with spacing $\delta x = 1/256$ and $\delta t = (\delta x)^2/2$ for $t \in [0,0.05]$.  The vectors to be identified are  $\ba = [a_1  \ldots a_3 \ a_{4,1} \ \ldots \ a_{4,L} \ a_5 \ a_6 \ a_{7,1} \ \ldots \ a_{7,L} \ a_8 \ldots a_{10}]^T$ where $b(x) \approx \sum_{\ell=1}^L a_{4,l}\phi_l(x)$ and $c(x) \approx \sum_{\ell=1}^L a_{7,l}\phi_l(x)$.  Figure \ref{Fig-47VaryDownsample} (a) shows the numerical solution of \eqref{E:47Vary}. In (b) the given data are downsampled by a factor of $4$ in time and space, and in (c) data not downsampled.   BEE and TEE successfully identifies the correct features.  Figure \ref{Fig-47VaryDownsample} (b)
plots both the recovered coefficients $\widehat{b}(x)$, $\widehat{c}(x)$ and the true coefficients $b(x)$ and $c(x)$ when data are downsampled.  Notice that the coefficient $\widehat{b}(x)$ of $u_x$ is not accurate when $x$ is close to $1$.  
The result is improved in (c) where data are not downsampled.
No downsampling helps to keep details of the solution around $x=1$ and reduces the finite-difference approximation errors.

Our final experiment is on Equation \eqref{E:BurgerDiffVary},  but all coefficients are allowed to vary in $x$.  The numerical solution is simulated in the same way as Figure \ref{Fig-7VaryNoise0} and the given data are downsampled by a factor of $4$ in time and space.  After all  coefficients are expanded in terms of $L$ finite element bases, the vectors to be identified is $\ba = \{a_{k,l}\}_{k=1,\ldots,10,\ l=1,\ldots,L}$ where $-1 = b(x)  \approx \sum_{l=1}^L a_{4,l}\phi(x)$ and $c(x) \approx \sum_{\ell=1}^L a_{7,l}\phi_l(x)$.  
Figure \ref{Fig-AllVaryDownsample4} (a) shows BEE, (b) shows group Lasso result, and (c) shows TEE. 
TEE identifies the correct support $[u_x \ u_{xx}]$ since it yields the smallest error. The coefficients $[\widehat{a}_{6,1} \ \ldots \ \widehat{a}_{6,L} \ \widehat{a}_{7,1} \ldots \widehat{a}_{7,L}]$ is computed by least squares. Figure  \ref{Fig-AllVaryDownsample4} (d) displays the computed coefficients $\widehat{b}(x) = \sum_{l=1}^L \widehat{a}_{6,l}\phi_l(x)$ and  $\widehat{c}(x) = \sum_{l=1}^L \widehat{a}_{7,l}\phi_l(x)$ when $L=20$, and (e) shows the coefficient recovery errors $\|-1-\widehat{b}(x)\|_{L^1}$ and $\|c(x)-\widehat{c}(x)\|_{L^1}$ as $L$ increases from $1$ to $30$. 
IDENT with BEE successfully identifies the correct terms even when all coefficients are free to vary in space.  The accuracy of the recovered coefficients has room for improvement if data are simulated and sampled on a finer grid.  
\begin{figure}[h]
\centering
\begin{tabular}{cc}
(a) BEE & 
(b) Group Lasso when $L=20$ \\
\includegraphics[width = 2.05in]{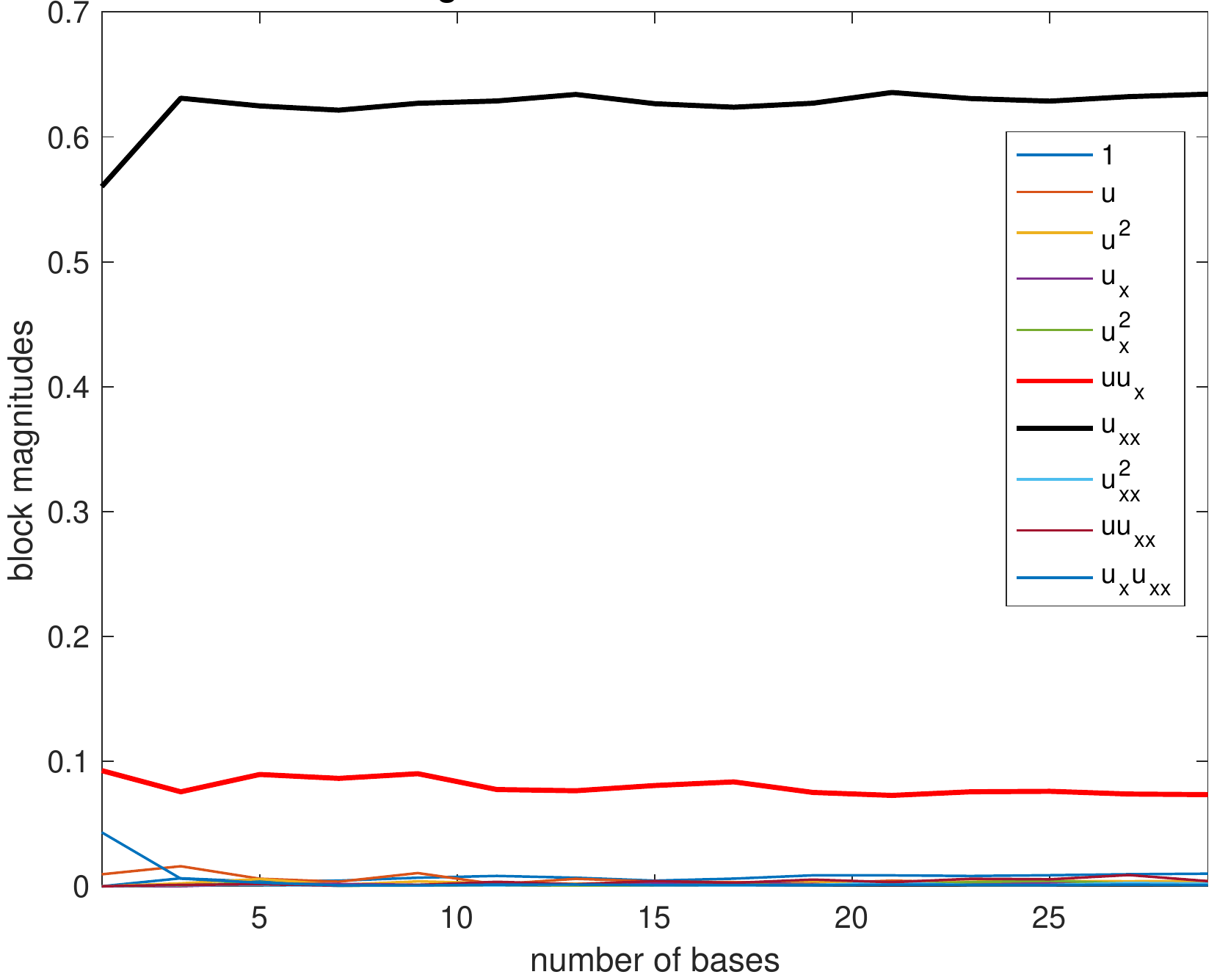} &
\includegraphics[width = 2.05in]{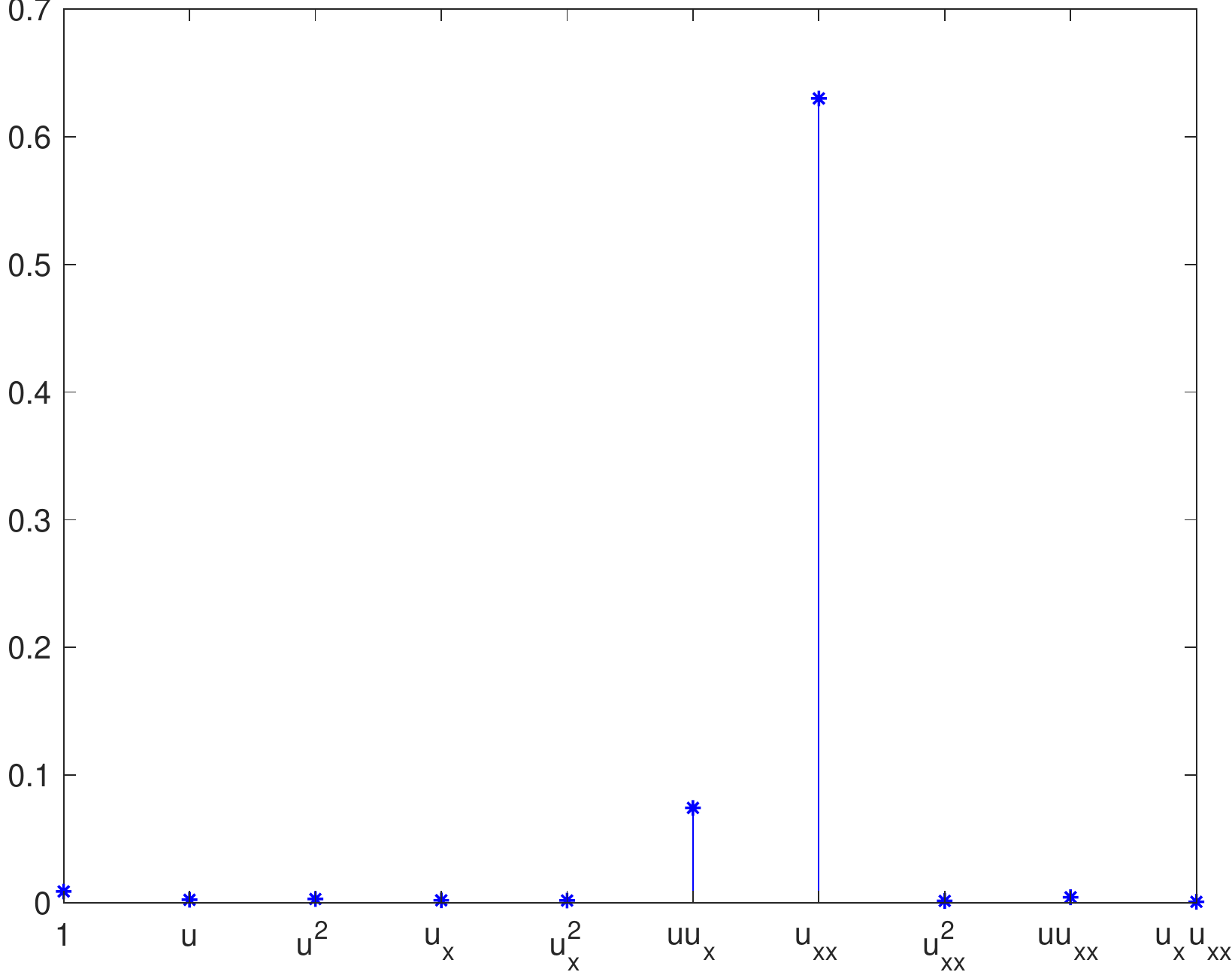}\\
\end{tabular}
\begin{tabular}{ccc}
(c) TEE in $\log_{10}$ scale vs. $L$& 
(d) $\widehat{c}(x)$ vs. $c(x)$  & 
(e)  $\|c(x)-\widehat{c}(x)\|_{L^1}$ vs. $L$ \\
\includegraphics[width = 2.05in]{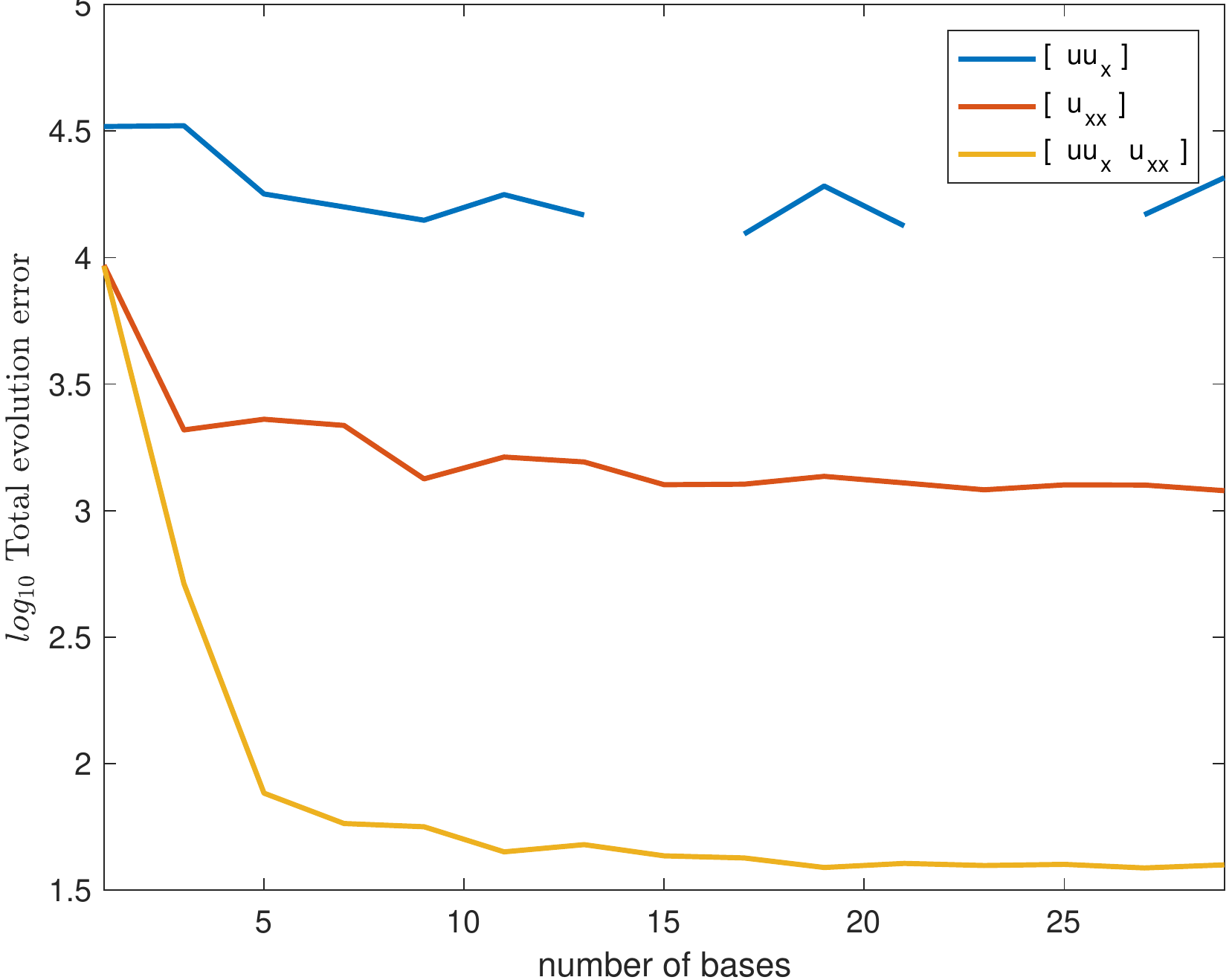} &
\includegraphics[width = 2.05in]{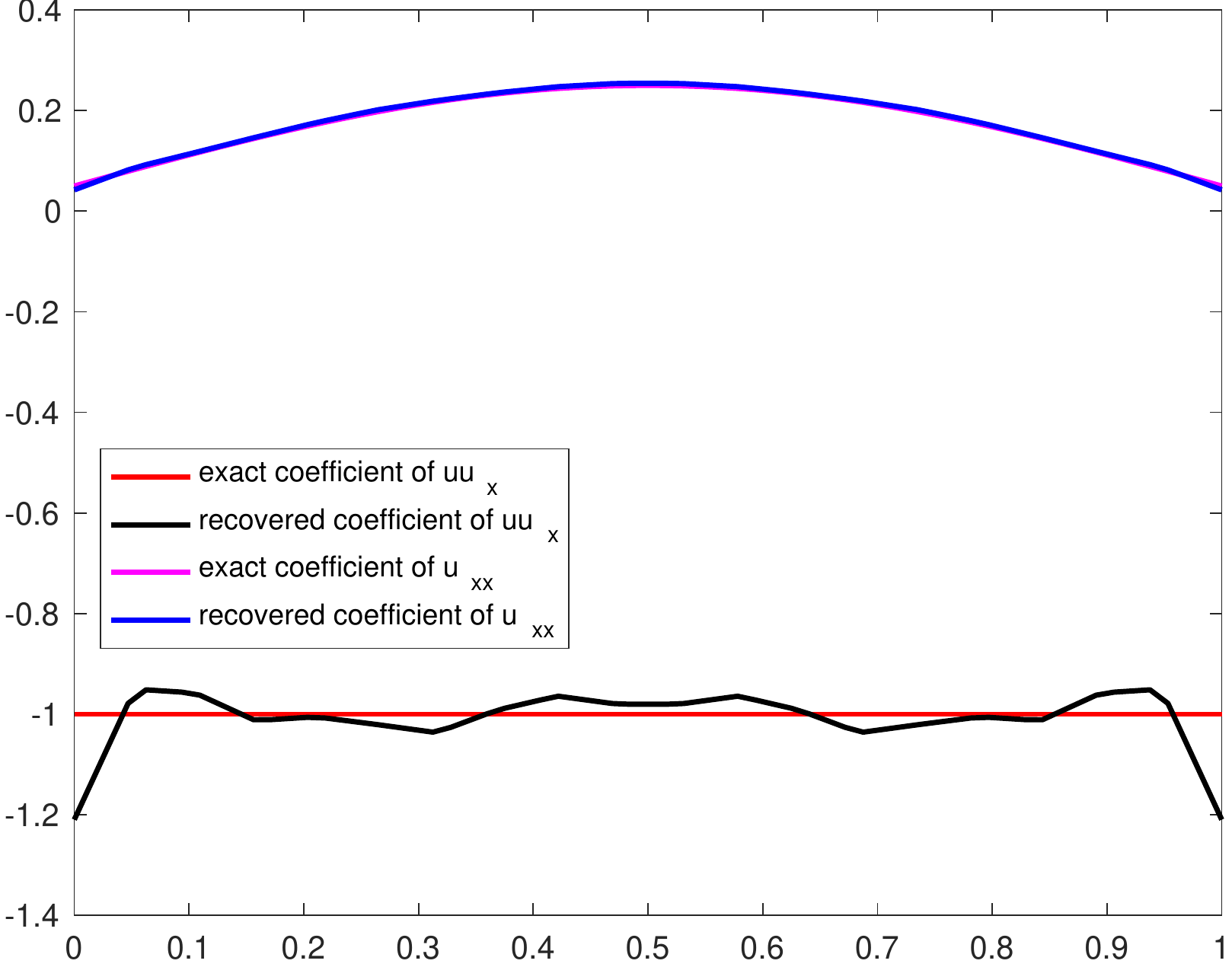} &
\includegraphics[width = 2.05in]{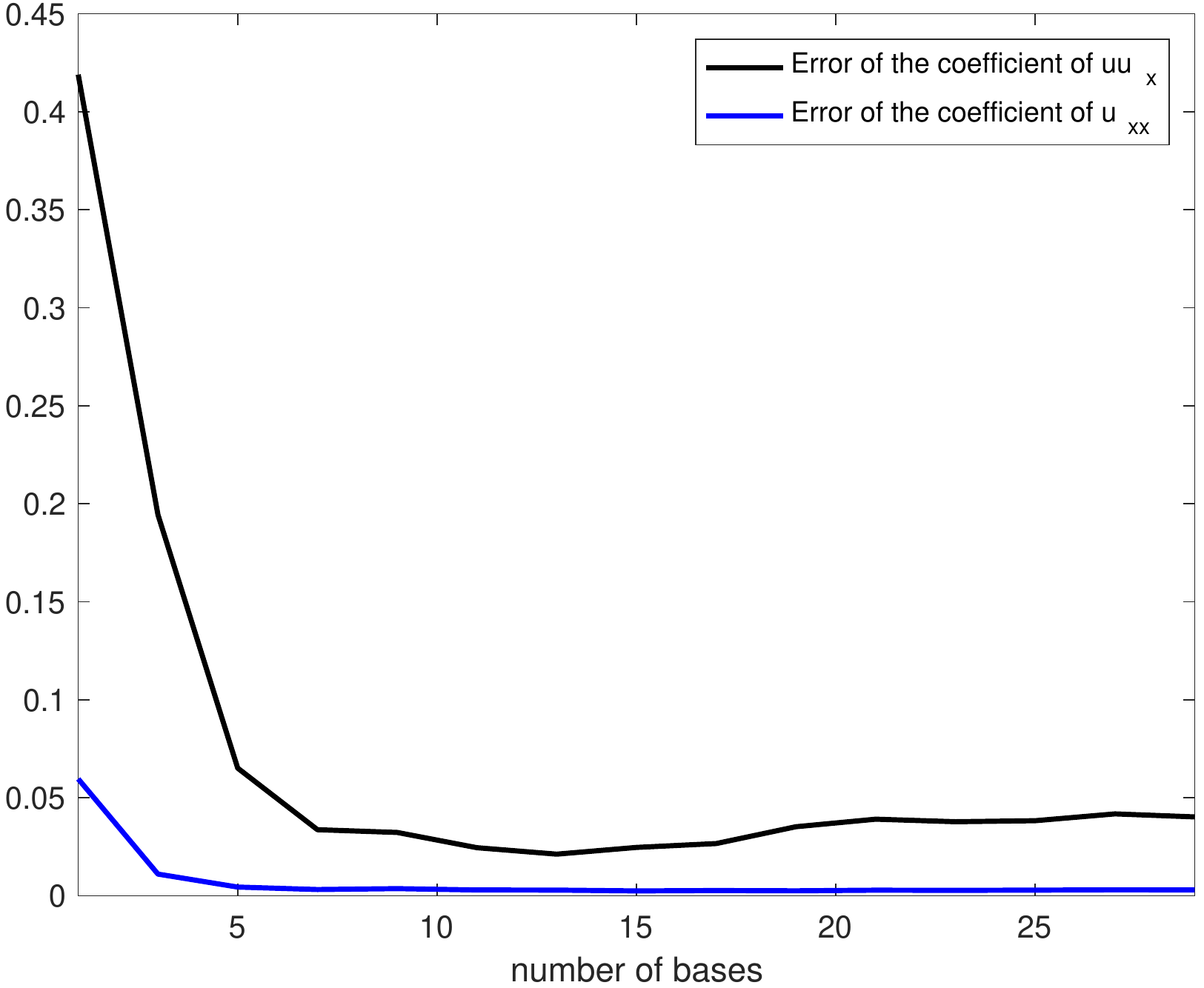}\\
\end{tabular}
\caption{Equation \eqref{E:BurgerDiffVary} where we are a priori given that all coefficients are free to vary with respect to $x$. Data are downsampled by a factor of $4$ in time and space. 
}
\label{Fig-AllVaryDownsample4}
\end{figure}

\section{Concluding remarks}\label{sec:summary}

We proposed a new method to identify PDEs from a given set of time dependent data, with techniques from numerical PDEs and fundamental ideas of convergence.  Assuming that the PDE is spanned by a few active terms in a prescribed dictionary, we used finite differences, such as the ENO scheme, to form an empirical version of the dictionary, and utilize $L^1$ minimization for efficiency.  Time Evolution Error (TEE)  was proposed as an effective tool to pick the correct set of terms.  

Starting with the first set of basic experiments, we considered noisy data, downsampling effect and PDEs with varying coefficients.    By establishing the recovery theory of Lasso for PDE recovery, a new noise-to-signal ratio (\ref{eqnsr}) is proposed, which measures the noise level more accurately in the setting of PDE identification.  We derived an error formula in (\ref{E:enoise}) and analyzed the effects of noise and downsampling.  A new order preserving denoising method called LSMA was proposed in subsection \ref{subsec:lsma} to aid the identification with noisy data.  IDENT can be applied to PDEs with varying coefficients and  BEE procedure helps to stabilize the result and reduce the computational cost.

\appendix

\begin{appendices}

\section{Recovery Theory of Lasso with  a weighted $L^1$ norm}\label{A:recovery}

In the field of compressive sensing, performance guarantees for the recovery of sparse vectors from a small number of noisy linear measurements by Lasso have been established when the sensing matrix satisfies an incoherence property \cite{donoho2001uncertainty} or a restricted isometry property \cite{candes2006robust}. We establish  the incoherence property of Lasso for the case of identifying PDE, where a weighted $L^1$ norm is used. 

Given a sensing matrix $\Phi \in \mathbb{R}^{n \times m}$ and the noisy measurement $$\bb = \Phi \bx^{\opt} + \be$$ where $\bx^\opt$ is $s$-sparse ($\|\bx^\opt\|_0 =s$), the goal is to recover $\bx_{\rm opt}$ in a robust way. Denote the support of $\bx^{\opt}$ by $\Lambda$ and let $\Phi_\Lambda$ be the submatrix of $\Phi$ whose columns are restricted on $\Lambda$. Suppose $\Phi = [\phi[1] \  \phi[2] \ \ldots \phi[m]]$ where all $\phi[j]$'s have unit norm.  Let the mutual coherence of $\Phi$ be 
$$\mu(\Phi) = \max_{j\neq l} |\phi[j]^T \phi[l]|.$$

The principle of Lasso with a weighted $L^1$ norm is to solve 
\begin{equation}
\label{E:QP}
\tag{W-Lasso}
\min_{\bx} \frac 1 2 \|\Phi {\bx}-\bb\|_2^2 + \gamma \|W\bx\|_1
\end{equation}
where $W = {\rm diag}(w_1,w_2,\ldots,w_m), w_j \neq 0, j=1,\ldots,m$ and $\gamma$ is a balancing parameter. Let $w_{\max} = \max_{j}|w_j|$ and $w_{\min} = \min_{j}|w_j|$. Lasso successfully recovers the support of $\bx^\opt$ when $\mu(\Phi)$ is sufficiently small.   The following proposition is a generalization of Theorem 8 in \cite{tropp2006just} from $L^1$ norm regularization to weighted $L^1$ norm regularization.

\begin{proposition}
\label{propLasso}
Suppose the support of $\bx^{\opt}$, denoted by $\Lambda$, contains no more than $s$ indices, $\mu(s-1)<1$ and
$$\frac{\mu s}{1-\mu (s-1)}< \frac{w_{\min}}{w_{\max}}.$$
Let 
\begin{equation}
\label{propLassogamma}
\gamma = \frac{1-\mu(s-1)}{w_{\min}[1 -\mu (s-1)] - w_{\max} \mu s}\|e\|_2^+,
\end{equation}
and $\bx(\gamma)$ be the minimizer of \eqref{E:QP}. Then
\begin{enumerate}

\item[1)] the support of $\bx(\gamma)$ is contained in $\Lambda$;

\item[2)] the distance between $\bx(\gamma)$ and $\bx^{\opt}$ satisfies
\begin{equation}
\label{propLassodist}
\|\bx(\gamma) - \bx^{\opt}\|_\infty \le \frac{w_{\max}}{w_{\min}[1- \mu (s-1)] - w_{\max} \mu s}\|\be\|_2;
\end{equation}

\item[3)] if $$\bx^\opt_{\min} := \min_{j \in \Lambda} |x_j^\opt|  > \frac{w_{\max}}{w_{\min}[1 -\mu (s-1)] - w_{\max} \mu s}\|\be\|_2,$$
then $\supp(\bx(\gamma)) = \Lambda$.

\end{enumerate}
\end{proposition}

\begin{proof}
Under the condition $\mu(s-1)<1$, $\Lambda$ indexes a linearly independent collection of columns of $\Phi$. Let $\bx^\star$ be the minimizer of \eqref{E:QP} over all vectors supported on $\Lambda$. A necessary and sufficient condition on such a minimizer is that
\begin{equation}
\label{propLassop1}
\bx^\opt - \bx^\star = \gamma (\Phi_\Lambda^*\Phi_\Lambda)^{-1} \bg - (\Phi_\Lambda^*\Phi_\Lambda)^{-1} \Phi_\Lambda^*\be
\end{equation}
where $\bg \in \partial \|W\bx^\star\|_1$, meaning $g_j = w_j{\rm sign}(x^\star)$ whenever $x^\star_j \neq 0$ and $|g_j| \le w_j$ whenever $x^\star_j = 0$. It follows that $\|\bg\|_\infty \le w_{\max}$ and
\begin{equation}
\label{propLassop2}
\|\bx^\star -\bx^\opt\|_\infty \le \gamma \|(\Phi_\Lambda^*\Phi_\Lambda)^{-1}\|_{\infty,\infty} (w_{\max}+\|\be\|_2). 
\end{equation}

Next we prove $x^\star$ is also the global minimizer of \eqref{E:QP} by demonstrating that the objective function increases when we change any other component of $\bx^\star$. Let  
$$L(\bx) = \frac 1 2 \|\Phi \bx-\bb\|_2^2 + \gamma \|W\bx\|_1.$$
Choose an index $\omega \notin \Lambda$ and let $\delta$ be a nonzero scalar. We will develop a condition which ensures that 
$$L(\bx^\star + \delta \be_\omega) - L(\bx^\star)>0$$
where $\be_\omega$ is the $\omega$th standard basis vector. Notice that
\begin{align*}
L(\bx^\star + \delta \be_\omega) - L(\bx^\star)
& = \frac 1 2\left[ \|\Phi (\bx^\star+\delta \be_\omega)-\bb\|_2^2 - \|\Phi \bx^\star-\bb\|_2^2\right] + \gamma\left(\|W(\bx^\star+\delta \be_\omega)\|_1- \|W\bx\|_1\right)
\\
& =  \frac 1 2 \|\delta \phi[\omega]\|^2 + {\rm Re} \langle \Phi \bx^\star -\bb ,\delta \phi[\omega] \rangle
+ \gamma |w_\omega\delta|
\\
& > {\rm Re} \langle \Phi \bx^\star -\bb ,\delta \phi[\omega] \rangle
+ \gamma |w_\omega\delta|
\\
& \ge \gamma w_{\min}|\delta| -  |\langle \Phi \bx^\star -\Phi \bx^\opt - \be ,\delta \phi[\omega] \rangle| 
\text{ since } \bb = \Phi \bx^\opt +\be
\\
& =  \gamma w_{\min}|\delta| -  |\langle \Phi_\Lambda \bx_{\Lambda}^\star -\Phi_\Lambda \bx_{\Lambda}^\opt - \be ,\delta \phi[\omega] \rangle| 
\\
& \ge \gamma w_{\min}|\delta| - |\langle \Phi_\Lambda (\bx_{\Lambda}^\star - \bx_{\Lambda}^\opt) ,\delta \phi[\omega] \rangle| - |\langle \be,\delta \phi[\omega]\rangle|
\\
& =  \gamma w_{\min}|\delta| - |\langle \gamma \Phi_\Lambda ( \Phi_\Lambda^* \Phi_\Lambda)^{-1}\bg ,\delta \phi[\omega] \rangle| - |\langle \be,\delta \phi[\omega]\rangle| \text{ thanks to \eqref{propLassop1}}
\\
& \ge  \gamma w_{\min}|\delta| - \gamma |\delta|\cdot |\langle  \Phi_\Lambda ( \Phi_\Lambda^* \Phi_\Lambda)^{-1}\bg , \phi[\omega] \rangle| - |\delta|\|\be\|_2
\\
& = \gamma w_{\min}|\delta| - \gamma |\delta|\cdot |\langle    (\Phi_\Lambda^\dagger)^*\bg , \phi[\omega] \rangle| - |\delta|\|\be\|_2
\\
& =  \gamma w_{\min}|\delta| - \gamma |\delta|\cdot |\langle    \bg , \Phi_\Lambda^\dagger \phi[\omega] \rangle| - |\delta|\|\be\|_2 
\\
& \ge \gamma w_{\min}|\delta| - \gamma |\delta|\|\bg\|_\infty \|\Phi_\Lambda^\dagger \phi[\omega] \|_1 - |\delta|\|\be\|_2
\\
& \ge \gamma w_{\min}|\delta| - \gamma |\delta|w_{\max}\max_{\omega \notin \Lambda} \|\Phi_\Lambda^\dagger\phi[\omega]\| - |\delta|\|\be\|_2.
\end{align*}
According to \cite{fuchs2004sparse,tropp2004just}, $\max_{\omega \notin \Lambda} \|\Phi_\Lambda^\dagger\phi[\omega]\| <\frac{\mu s}{1-\mu(s-1)}$.
A sufficient condition to guarantee $L(\bx^\star + \delta \be_\omega) - L(\bx^\star)>0$ is $$\gamma\left(w_{\min} - w_{\max} \frac{\mu s}{1-\mu(s-1)} \right) > \|\be\|_2,$$
which gives rise to \eqref{propLassogamma}. This establishes that $\bx^\star$ is the global minimizer of \eqref{E:QP}. \eqref{propLassodist} is resulted from \eqref{propLassop2} along with $\|(\Phi_\Lambda^*\Phi_\Lambda)^{-1}\|_{\infty,\infty} \le [1-\mu(s-1)]^{-1}$.
\end{proof}

We prove Theorem \ref{thmLasso} based on Proposition \ref{propLasso}.
\begin{proof}[Proof of Theorem \ref{thmLasso}]
Suppose $\hF_{\rm unit}$ is obtained from $\hF$ with the columns normalized to unit $L^2$ norm and let $W \in \mathbb{R}^{N_3 \times N_3}$ be the diagonal matrix with $W_{jj} =\|\hF[j]\|_\infty \|\hF[j]\|_2^{-1}$. 
The Lasso we solve is equivalent to
$$\widehat\by = \arg \min \frac 1 2 \|\widehat\bb - \hF_{\rm unit} \by\| + \lambda \|W \by\|_1$$
where $\bz = W \by$, $\by^{\rm opt}_j  = \ba_j \|\hF[j]\|_2$ and $\be = \widehat\bb - \hF_{\rm unit} \by^{\rm opt} $. Then we apply Proposition \ref{propLasso}. The choice of balancing parameters in \eqref{propLassogamma} suggests 
$$\lambda = \frac{1-\mu(s-1)}{\min_j\frac{\|\hF[j]\|_\infty}{\|\hF[j]\|_2}[1-\mu(s-1)] -\max_j\frac{\|\hF[j]\|_\infty}{\|\hF[j]\|_2}\mu s}\|\be\|_2^+,$$
which gives rise to \eqref{thmlambda}. The error bound in \eqref{propLassodist} gives
$$\|\widehat\by - \by^{\rm opt}\|_\infty  \le  \frac{(\max_j\|\hF[j]\|_\infty \|\hF[j]\|_2^{-1}+\|\be\|_2)}{\min_j \|\hF[j]\|_\infty \|\hF[j]\|_2^{-1} [1- \mu (s-1)] - \max_j \|\hF[j]\|_\infty \|\hF[j]\|_2^{-1} \mu s}\|\be\|_2$$
which implies
$$\max_j \|\hF[j]\|_{L^2} \left|\|\hF[j]\|_\infty^{-1}\widehat{\ba}_{\text{Lasso}}(\lambda)_j-\ba_j\right| 
\le \frac{w_{\max}+\varepsilon/\sqrt{\Delta t \Delta x}}{w_{\min}[1-\mu(s-1)] - w_{\max}\mu s } \varepsilon,$$
which yields \eqref{thmdist}.
\end{proof}

\end{appendices}

\bibliographystyle{plain}
\bibliography{cite_findingPDE}  

\end{document}